\documentclass{amsart}
\usepackage{graphicx}
\usepackage{bm,caption,amsbsy,enumerate,amsmath,amsthm,
amssymb,mathtools,amsfonts,multirow,verbatim,tikz,tikz-cd,thmtools,thm-restate}
\usepackage[alphabetic]{amsrefs}
\usetikzlibrary{matrix,arrows,decorations.pathmorphing}
\usepackage[top=1.25in, bottom=1in, left=1.25in, right=1.25in,marginpar=1in]{geometry}
\usepackage{url}
\usepackage[hidelinks,colorlinks=true,linkcolor=blue, citecolor=black,linktocpage=true]{hyperref}
\usepackage{chngcntr}
\counterwithin{figure}{section}

\newcommand\myurl[1]{\url{#1}}

\BibSpec{webpage}{%
  +{}{\PrintAuthors} {author}
  +{,}{ \textit} {title}
  +{}{ \parenthesize} {date}
  +{,}{ \myurl} {myurl}
  +{,}{ } {note}
  +{.}{ } {transition}
}

\newtheorem{innercustomthm}{Theorem}
\newenvironment{customthm}[1]
  {\renewcommand\theinnercustomthm{#1}\innercustomthm}
  {\endinnercustomthm}

  \newtheorem{innercustomprop}{Proposition}
\newenvironment{customprop}[1]
  {\renewcommand\theinnercustomprop{#1}\innercustomprop}
  {\endinnercustomprop}
  
    \newtheorem{innercustomlem}{Lemma}
  \newenvironment{customlem}[1]
    {\renewcommand\theinnercustomlem{#1}\innercustomlem}
    {\endinnercustomlem}

\newtheorem{thm}{Theorem}[section]
\newtheorem{prop}[thm]{Proposition}

\newtheorem{cor}[thm]{Corollary}
\newtheorem{lem}[thm]{Lemma}
\newtheorem{sublem}[thm]{Sublemma}

\theoremstyle{definition}
\newtheorem{define}[thm]{Definition}

\theoremstyle{remark}
\newtheorem{rem}[thm]{Remark}

\newtheorem{question}[thm]{Question}

\newcommand{\ve}[1]{\boldsymbol{\mathbf{#1}}}
\newcommand{\R}{\mathbb{R}}

\newcommand{\Z}{\mathbb{Z}}

\newcommand{\Q}{\mathbb{Q}}

\renewcommand{\d}{\partial}
\renewcommand{\subset}{\subseteq}

\renewcommand{\tilde}{\widetilde}
\renewcommand{\bar}{\overline}
\renewcommand{\hat}{\widehat}
\newcommand{\iso}{\cong}

\DeclareMathOperator{\cotr}{{cotr}}
\DeclareMathOperator{\Cone}{{Cone}}

\DeclareMathOperator{\gr}{{gr}}

\DeclareMathOperator{\Hom}{{Hom}}

\DeclareMathOperator{\id}{{id}}

\DeclareMathOperator{\im}{{im}}

\DeclareMathOperator{\Mod}{{mod}}

\DeclareMathOperator{\Pin}{{Pin}}

\DeclareMathOperator{\Spin}{{Spin}}

\DeclareMathOperator{\Span}{{Span}}
\DeclareMathOperator{\Sym}{{Sym}}

\DeclareMathOperator{\Tors}{{Tors}}

\DeclareMathOperator{\tr}{{tr}}

\DeclareMathOperator{\Mor}{{Mor}}

\newcommand{\bF}{\mathbb{F}}

\newcommand{\bK}{\mathbb{K}}
\newcommand{\bL}{\mathbb{L}}

\newcommand{\bS}{\mathbb{S}}
\newcommand{\bT}{\mathbb{T}}
\newcommand{\bU}{\mathbb{U}}

\newcommand{\cA}{\mathcal{A}}

\newcommand{\cC}{\mathcal{C}}

\newcommand{\cF}{\mathcal{F}}
\newcommand{\cG}{\mathcal{G}}
\newcommand{\cH}{\mathcal{H}}

\newcommand{\cL}{\mathcal{L}}
\newcommand{\cM}{\mathcal{M}}
\newcommand{\cN}{\mathcal{N}}

\newcommand{\cR}{\mathcal{R}}

\newcommand{\cT}{\mathcal{T}}

\newcommand{\frI}{\mathfrak{I}}

\newcommand{\frs}{\mathfrak{s}}
\newcommand{\frt}{\mathfrak{t}}

\newcommand{\cCFL}{\cC\!\cF\!\cL}
\newcommand{\cHFL}{\cH\!\cF\!\cL}
\newcommand{\CF}{\mathit{CF}}
\newcommand{\HF}{\mathit{HF}}
\newcommand{\CFI}{\mathit{CFI}}
\newcommand{\CFK}{\mathit{CFK}}
\newcommand{\CFL}{\mathit{CFL}}
\newcommand{\HFI}{\mathit{HFI}}
\newcommand{\HFL}{\mathit{HFL}}

\newcommand{\PD}{\mathit{PD}}

\newcommand{\xs}{\ve{x}}
\newcommand{\ys}{\ve{y}}
\newcommand{\zs}{\ve{z}}
\newcommand{\ws}{\ve{w}}

\newcommand{\ps}{\ve{p}}
\newcommand{\qs}{\ve{q}}
\newcommand{\as}{\ve{\alpha}}
\newcommand{\bs}{\ve{\beta}}
\newcommand{\gs}{\ve{\gamma}}
\newcommand{\ds}{\ve{\delta}}

\title{Connected sums and involutive knot Floer homology}
\author{Ian Zemke}
\address{Department of Mathematics\\Princeton University\\  Princeton, NJ 08544, USA}
\email{izemke@math.princeton.edu}
\thanks{This research was supported by NSF grant DMS-1703685}
\begin{document}

	\begin{abstract}
		 We prove a formula for the conjugation action on the knot Floer complex of the connected sum of two knots. Using the formula we construct a homomorphism from the smooth concordance group to an abelian group consisting of chain complexes with homotopy automorphisms,  modulo an equivalence relation. Using our connected sum formula, we perform some example computations of Hendricks and Manolescu's involutive invariants on large surgeries of connected sums of knots.
		\end{abstract}
\maketitle%
\tableofcontents

\section{Introduction}

Heegaard Floer homology is an invariant associated to 3-manifolds, introduced by Ozsv\'{a}th and Szab\'{o} \cite{OSDisks} \cite{OSProperties}. To a 3-manifold $Y$ with a $\Spin^c$ structure $\frs$, Ozsv\'{a}th and Szab\'{o} construct modules $\HF^-(Y,\frs)$ over the ring $\Z[U]$. In this paper we will work over $\bF_2[U]$. The module $\HF^-(Y,\frs)$ is the homology of a chain complex $\CF^-(Y,\frs)$. There is a natural conjugation action on the set of $\Spin^c$ structures, which is reflected on the Heegaard Floer complexes by a chain homotopy equivalence
\[
\iota\colon \CF^-(Y,\frs)\to \CF^-(Y,\bar{\frs}).
\] When $\frs$ is self conjugate, i.e. $\bar{\frs}=\frs$, $\iota$ induces a homotopy involution.

 Hendricks and Manolescu construct a refinement of Heegaard Floer homology, called \emph{involutive Heegaard Floer homology} \cite{HMInvolutive}. They define a module $\HFI^-(Y,\frs)$ over $\bF_2[U,Q]/(Q^2)$ as the homology of the mapping cone
\[
\CFI^-(Y,\frs):=\Cone\big(\CF^-(Y,\frs)\xrightarrow{Q\cdot(\id+\iota)} Q\cdot \CF^-(Y,\frs)\big),
\] 
where $Q$ is a formal variable.  They define two correction terms, $\bar{d}$ and $\underline{d}$, analogous to the Ozsv\'{a}th-Szab\'{o} $d$ invariant from Heegaard Floer homology \cite{OSIntersectionForms}.

An important property of Heegaard Floer homology is that it satisfies a K\"{u}nneth formula for connected sums.  If $Y_1$ and $Y_2$ are two connected 3-manifolds, Ozsv\'{a}th and Szab\'{o}  \cite{OSProperties} construct a quasi-isomorphism
\[
\CF^-(Y_1\# Y_2, \frs_1\# \frs_2)\iso \CF^-(Y_1,\frs_1)\otimes_{\bF_2[U]} \CF^-(Y_2,\frs_2).
\] 
In \cite{HMZConnectedSum}, the behavior of the conjugation involution $\iota$  on connected sums is described in terms of the above quasi-isomorphism, giving a version of the K\"{u}nneth theorem for involutive Heegaard Floer homology.

Knot Floer homology is a refinement of Heegaard Floer homology for knots embedded in 3-manifolds, introduced by Ozsv\'{a}th and Szab\'{o} \cite{OSKnots} and independently by Rasmussen \cite{RasmussenKnots}. Link Floer homology is a generalization of knot Floer homology for links in 3-manifolds, developed by Ozsv\'{a}th and Szab\'{o} \cite{OSLinks}. There is a similar conjugation action on knot Floer homology, and also a version of the K\"{u}nneth theorem for knot Floer homology.

 In this paper we consider an analogous conjugation action on knot Floer homology, its form on connected sums of knots, and more generally the interaction between the conjugation action and maps induced by decorated link cobordisms.

To an oriented knot $\bK=(K,p,q)$, decorated with two basepoints, in a 3-manifold $Y$ equipped with a $\Spin^c$ structure $\frs\in \Spin^c(Y)$, we consider a chain complex
\[
\cCFL^\infty(Y,\bK,\frs),
\] 
over the ring
 \[
 \cR=\bF_2[U,V,U^{-1},V^{-1}],
 \]
 the Laurent polynomial ring generated by two variables. This is a slight variation of the full knot Floer complex described by Ozsv\'{a}th and Szab\'{o} \cite{OSKnots}. The chain complex $\cCFL^\infty(Y,\bK,\frs)$ has a filtration over $\Z\oplus \Z$.
 If $K$ is null-homologous and $\frs$ is a self conjugate $\Spin^c$ structure, Hendricks and Manolescu \cite{HMInvolutive} consider a conjugation map
\begin{equation}
\iota_K\colon\cCFL^\infty(Y,\bK, \frs)\to \cCFL^\infty(Y,\bK,\frs).\label{eq:iotadef}
\end{equation}  Unlike the map $\iota$ on $\CF^-(Y,\frs)$, the map $\iota_K$ is not a homotopy involution. Instead, the map $\iota_K$ satisfies
\[
\iota_K^2\simeq \rho_*
\]
 where $\rho_*$ is the map induced by the diffeomorphism $\rho\colon(Y,K,p,q)\to (Y,K,p,q)$ obtained by twisting the knot $K$ in one full twist, in the direction of its orientation. The diffeomorphism map $\rho_*$ is considered in \cite{SarkarMovingBasepoints} and \cite{ZemQuasi}. Over $\bF_2,$ however, one has $\iota^4_K\simeq \id.$

\subsection{The conjugation map on connected sums of knots}

Our first result is the chain homotopy type of the map $\iota_K$ on connected sums of knots:

\begin{thm}\label{thm:B}Suppose that $(Y_1,\bK_1)$ and $(Y_2,\bK_2)$ are two pairs of 3-manifolds with embedded, doubly based, null-homologous knots, and  $\frs_1$ and $\frs_2$ are self-conjugate $\Spin^c$ structures on $Y_1$ and $Y_2$. Writing $\bK_1\# \bK_2$ for the connected sum, with exactly two basepoints, there are filtered, $\cR$-equivariant chain homotopy equivalences between
\[
\cCFL^\infty(Y_1,\bK_1,\frs_1)\otimes_{\cR}\cCFL^\infty(Y_2,\bK_2,\frs_2)\qquad \text{and} \qquad \cCFL^\infty(Y_1\# Y_2,\bK_1\# \bK_2,\frs_1\# \frs_2)
\]
 which intertwine $\iota_{K_1\# K_2}$, on the latter chain complex, and
\[
(\id|\id+\Phi_1|\Psi_2)(\iota_{K_1}| \iota_{K_2})
\] 
on the tensor product complex, for  endomorphisms $\Phi_i$ and $\Psi_i$ of $\cCFL^\infty(Y_i,\bK_i,\frs_i)$. There are also (different) chain homotopy equivalences of the two above complexes which intertwine  $\iota_{K_1\# K_2}$ with
\[
(\id|\id+\Psi_1|\Phi_2)(\iota_{K_1}| \iota_{K_2}).
\]
\end{thm}

In the above theorem,  the vertical bar $|$ denotes the tensor product of maps. The maps $\Phi_{i}$ and $\Psi_{i}$ appear frequently in the link Floer TQFT \cite{JCob} \cite{ZemCFLTQFT}, however they appeared earlier in work of Sarkar on diffeomorphism maps on knot Floer homology \cite{SarkarMovingBasepoints}. Using our formulation, they can algebraically be described as the formal derivatives of the differential with respect to $U$ or $V$, respectively. Notably, the maps $\Phi_i$ and $\Psi_i$ are the link cobordism maps for the decorated link cobordisms shown in Figure \ref{fig::17}. This interpretation of the maps $\Phi_i$ and $\Psi_i$ in terms of decorated link cobordisms turns out to be useful in our proof of Theorem \ref{thm:B}. The maps $\Phi_i$ and $\Psi_i$ also have a simple description in terms of counting holomorphic disks; see Section \ref{sec:PhiPsi}.

\begin{figure}[ht!]
\centering
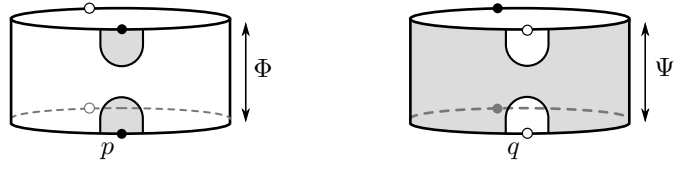
\caption{\textbf{Decorated link cobordisms for the endomorphisms $\Phi$ and $\Psi$.} The underlying undecorated link cobordism is $[0,1]\times K\subset [0,1]\times Y$.\label{fig::17}}
\end{figure}

\subsection{Conjugation invariance of the link Floer TQFT}

To arrive at Theorem \ref{thm:B}, we will study the interaction between the link Floer TQFT and the conjugation action more generally. If $\bL=(L,\ve{p},\ve{q})$ is an oriented link in $Y$ with collections of basepoints $\ve{p}$ and $\ve{q}$ in $L$ which alternate between $\ve{p}$ and $\ve{q}$ as one travels along $L$, we consider a  version of the full link Floer complex, which we denote by
\[
\cCFL^\infty(Y,\bL,\frs).
\] This is a module over $\cR=\bF_2[U,V,U^{-1},V^{-1}]$, and has a natural filtration by $\Z\oplus \Z$. 

We note that the multi-based knots $(L,\ps,\qs)$ and $(L,\qs,\ps)$ are not equal, in the framework of link Floer homology. There is a tautological conjugation action on link Floer homology
\[
\eta\colon\cCFL^\infty(Y,L,\ve{p},\ve{q},\frs)\to \cCFL^\infty(Y,L,\ve{q},\ve{p},\bar{\frs}+\PD[L]),
\] which is filtered and $\cR$-equivariant.  The map $\eta$ canonically squares to the identity, but is not an involution since it maps between the link Floer complexes of two different based links.

When $\bar{\frs}=\frs$ and  $\bK=(K,p,q)$ is a doubly based knot which is null-homologous, the endomorphism $\iota_K$ from Equation~\eqref{eq:iotadef} is defined as the composition
\[
\iota_K:=(\tau_K)_*\circ \eta,
\]
 where $\tau_K$ is the diffeomorphism associated to a half twist of $K$, in the direction of the orientation of $K$.

The chain homotopy equivalences appearing in Theorem \ref{thm:B} are the link cobordism maps for certain decorated link cobordisms, using the link cobordism maps from \cite{ZemCFLTQFT}.   The first step towards proving Theorem \ref{thm:B} is to analyze the interaction between the link cobordism maps from \cite{ZemCFLTQFT} and the conjugation map $\eta$. The maps from \cite{ZemCFLTQFT} use the following notion of  decorated link cobordism  (adapted from \cite{JCob}*{Definition~4.5}):

\begin{define}\label{def:decoratedlinkcob} We say a pair $(W,\cF)$ is  a \emph{decorated link cobordism}  between two 3-manifolds with multi-based links, and write $(W,\cF)\colon(Y_1,L_1,\ps_1,\qs_1)\to (Y_2,L_2,\ps_2,\qs_2),$ if the following are satisfied:
\begin{enumerate}
\item $W$ is a 4-dimensional cobordism from $Y_1$ to $Y_2$.
\item $\cF=(S, \cA)$, where $ S $ is an oriented surface, properly embedded in $W$, with $\d  S =-L_1\sqcup L_2$.
\item $\cA\subset  S $ is a properly embedded 1-manifold dividing $ S $ into two disjoint subsurfaces $S_{\ve{w}}$ and $ S_{\ve{z}}$, which meet along $\cA$.
\item Each component of $L_i\setminus \cA$ contains exactly one basepoint.
\item  $\ve{p}_1\cup \ve{p}_2\subset S_{\ve{w}}$ and $\ve{q}_1\cup \ve{q}_2\subset  S_{\ve{z}}$.
\end{enumerate}
\end{define}

 As a first step to proving Theorem \ref{thm:B}, we prove that the link cobordism maps from \cite{ZemCFLTQFT} satisfy a version of conjugation invariance similar to the one satisfied by the cobordism maps of Ozsv\'{a}th and Szab\'{o} \cite{OSTriangles}*{Theorem~3.6}.

\begin{thm}\label{thm:C}Suppose that $(W,\cF)\colon (Y_1,L_1,\ps_1,\qs_1)\to (Y_2,L_2,\ps_2,\qs_2)$ is a decorated link cobordism, and let $(W,\bar{\cF})\colon (Y_1,L_1,\qs_1,\ps_1)\to (Y_2,L_2,\qs_2,\ps_2)$ denote the cobordism obtained by switching the roles of $ S_{\ve{w}}$ and $ S_{\ve{z}}$. Write $\frs_i=\frs|_{Y_i}$. The following diagram commutes up to filtered, $\cR$-equivariant chain homotopy:
\[
\begin{tikzcd}\cCFL^\infty(Y_1,L_1,\ps_1,\qs_1,\frs_1)\arrow{d}{F_{W,\cF,\frs}}\arrow{r}{\eta}& \cCFL^\infty(Y_1,L_1,\qs_1,\ps_1,\bar{\frs}_1+\PD[L_1])\arrow{d}{F_{W,\bar{\cF},\bar{\frs}+\PD[ S]}}\\
\cCFL^\infty(Y_2,L_2,\ps_2,\qs_2,\frs_2) \arrow{r}{\eta} & \cCFL^\infty(Y_2,L_2,\qs_2,\ps_2,\bar{\frs}_2+\PD[L_2]).
\end{tikzcd}
\]
\end{thm}

\subsection{A bypass relation for link Floer homology}

To prove Theorem \ref{thm:B}, we will need to relate the maps induced by three non-isotopic dividing sets on a pair-of-pants cobordism. Such relations turn out to be quite common in the link Floer TQFT, and can be conveniently encoded in the following result:

\begin{lem}\label{thm:D}Suppose $(W,\cF_1)$, $(W,\cF_2)$ and $(W,\cF_3)$ are decorated link cobordisms from $(Y_1,\bL_1)$ to $(Y_2,\bL_2)$, which form a bypass triple, as in Figure \ref{fig::33}. Then
\[
F_{W,\cF_1,\frs}+F_{W,\cF_2,\frs}+F_{W,\cF_3,\frs}\simeq 0.
\]
\end{lem}

\begin{figure}[ht!]
	\centering
	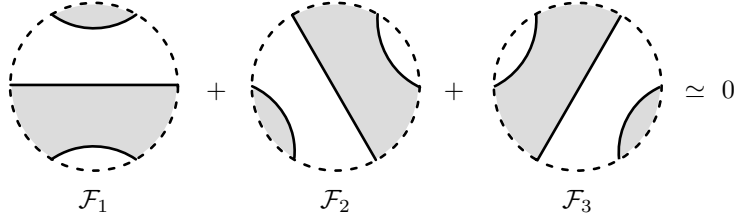
	\caption{\textbf{The bypass relation} for decorated link cobordisms $(W,\cF_1),$  $(W,\cF_2)$ and $(W,\cF_3)$ which agree outside of the region shown.\label{fig::33}}
\end{figure}

One motivation for Lemma~\ref{thm:D} is from Juh\'{a}sz's construction \cite{JCob} of cobordism maps on $\hat{\HFL}$  using the contact gluing map of Honda, Kazez and Mati\'{c} \cite{HKMTQFT}. Juh\'{a}sz uses the $S^1$-invariant contact structure on the boundary of a regular neighborhood of $ S\subset W$ which is determined by the dividing set on $ S$. The contact invariants in the sutured Floer homology of $S^1\times  S$ for a triple of contact structures fitting into a bypass triple (see \cite{HKMTQFT}*{Section~7}) satisfy a similar relation. Despite it's contact geometric inspiration, our proof of Lemma~\ref{thm:D} is a direct computation, and involves no contact geometry.

Using the bypass relation, in Section \ref{sec:bypassrelation} we give a pictorial proof of a special case of Sarkar's formula \cite{SarkarMovingBasepoints}*{Theorem~1.1} for the diffeomorphism map induced by a Dehn twist $\rho$ which twists a link component in one full twist,  when the link component has just two basepoints, $p$ and $q$. The induced diffeomorphism map satisfies the formula
\[
\rho_*\simeq \id+\Phi_p\circ \Psi_q.
\] 
Sarkar proved this on the associated graded version of the link Floer complex for links in $S^3$, using grid diagrams. The author extended the result to  the full link Floer complex for links in arbitrary 3-manifolds \cite{ZemQuasi}.

\subsection{A homomorphism from the smooth concordance group}

Using Theorem \ref{thm:B} we construct a homomorphism from the smooth concordance group $\cC$ to an algebraically defined, abelian group $\frI_K$. The construction is inspired by constructions of Hom \cite{HomEpsilon} and Stoffregen \cite{StoffregenSeifertFibered}. See also \cite{HMZConnectedSum} for a similar construction.

 In Section \ref{sec:algebraicpreliminaries}, we define a notion of an $\iota_K$-complex, which is a tuple $(C,\d,B,\iota)$, where $(C,\d)$ is a free chain complex over $\cR$ with basis $B$,  and $\iota$ is an $\cR$-skew equivariant (i.e. it switches the actions of $U$ and $V$) and $\Z\oplus\Z$-skew filtered map, which squares to $\id+\Phi\circ \Psi$.  We  define an abelian group $\frI_K$, consisting of $\iota_K$-complexes modulo a relation we call local equivalence (see Section \ref{subsec:thegroupI_K} for a precise definition). To a smooth concordance between two knots (or more generally a homology concordance; see below) the link cobordism maps from \cite{ZemCFLTQFT} induce a local equivalence. We prove the following:

\begin{thm}\label{thm:E}The set $\frI_K$ of $\iota_K$-complexes modulo local equivalence is a well defined abelian group, and the map $\cC\to \frI_K$ defined by
	\[
	K\mapsto [\cCFL^\infty(S^3,K)]
	\]
	 is a  homomorphism.
\end{thm}

More generally, we can extend the above theorem to knots in integer homology spheres. One can define a group $\cC_3^\Z$ generated by pairs $(Y,K)$ where $Y$ is an integer homology sphere and $K\subset Y$ is an oriented knot. We identify two pairs $(Y_1,K_1)$ and $(Y_2,K_2)$ if there is an integer homology cobordism between $Y_1$ and $Y_2$, which contains a smoothly embedded annulus $ S$ such that $\d  S=-L_1\sqcup L_2$. Multiplication in the group is given by taking the connected sum of both the 3-manifolds and the knots. It is not hard to see that $\cC_3^\Z$ is an abelian group, where the inverse of $(Y,K)$ is given by $(-Y,-K)$.

\begin{rem}The homomorphism  $\cC\to \frI_K$ from Theorem \ref{thm:E} factors through a homomorphism $\cC_3^\Z\to \frI_K$.
\end{rem}

\subsection{Example computations}

 In Section \ref{sec:examples}, we use the above formula to compute the involutive correction terms $\bar{V}_0$ and $\underline{V}_0$ from large surgeries from \cite{HMInvolutive} for several examples of connected sums of knots. The following is a straightforward consequence of several example computations which we perform:

 \begin{prop}\label{prop:applicationsofcomp}\sloppy None of the knots $T_{4,5}\# T_{4,5}$, $T_{4,5}\# T_{4,5}\# T_{5,6}$, $T_{6,7}\# T_{6,7}$, $T_{4,5}\# T_{6,7}$, and $T_{3,4}^{-1}\# T_{4,5}^{-1}\# T_{5,6}$ are concordant to a thin knot, an $L$-space knot, or the mirror of an $L$-space knot.
 \end{prop}

 The proof of the above proposition uses the computation of the triple $(\bar{V}_0,V_0,\underline{V}_0)$ for thin knots, $L$-space knots, and mirrors of $L$-space knots from \cite{HMInvolutive}. For such knots, the involutive correction terms have a simple pattern: either all three are nonnegative and $\underline{V}_0-\bar{V}_0\le 1$, or $\bar{V}_0\le 0=V_0=\underline{V}_0$. Using our formula we will compute that none of the knots listed above have involutive invariants fitting into one of these two patterns. See \cite{BHRationalCuspidal} for some deeper applications of involutive large surgery invariants as well as some applications of the connected sum formula from this paper.
 
  Note that considering the triple $(\bar{V}_0,V_0,\underline{V}_0)$ is not the only way to detect various parts of the above result. For example $\Upsilon_K(t)$ is very simple for thin knots, so can be used as an obstruction to knots being concordant to thin knots. Similarly $V_0=0$ for the mirror of any $L$-space knot, and $V_0\neq 0$ for the above knots. The function $\Upsilon_K(t)$ is convex for an $L$-space knot \cite{OSSUpsilon}*{Theorem~6.2}, which can provide an obstruction to knots being concordant to $L$-space knots, however some of the examples above also have convex $\Upsilon_K(t)$ invariants.

\subsection{Further Remarks}

In \cite{HMZConnectedSum}*{Theorem 1.1}, a similar formula for the involution $\iota$ on $\CF^-$ of connected sums of closed 3-manifolds is proven, by considering the interaction between the conjugation action and the graph cobordism maps from \cite{ZemGraphTQFT}. Writing $\iota_i$ for the involution on $\CF^-(Y_i,\frs_i)$, we showed that the involution on $\CF^-(Y_1\# Y_2,\frs_1\# \frs_2)$ is $\bF_2[U]$-equivariantly chain homotopic to
\begin{equation}
(\id|\id+U(\Phi_1|\Phi_2))(\iota_1| \iota_2) \label{eq:conjugationclosed3-manifold}
\end{equation}
 under a chain homotopy equivalence with $\CF^-(Y_1,\frs_1)\otimes_{\bF_2[U]} \CF^-(Y_2,\frs_2)$. In the context of the 3-manifold invariants, the maps $\Phi_i$ are analogous to the maps $\Phi_i$ and $\Psi_i$ appearing in Theorem \ref{thm:B}. By specializing Theorem \ref{thm:B} of this paper to the case $\bK_1$ and $\bK_2$ are doubly based unknots, one recovers Equation~\eqref{eq:conjugationclosed3-manifold}. In the setting of 3-manifold invariants, the term $U(\Phi_{1}| \Phi_{2})$ turns out to be $\bF_2[U]$-equivariantly null-homotopic when each $\frs_i$ is self conjugate, by an algebraic computation. As we will see in Section \ref{sec:examples}, the term $\Phi_1|\Psi_2$ is usually non-trivial in the context of link Floer homology.

For a doubly based knot $\bK=(K,p,q)$ in $ S^3$, one more often considers a $\Z\oplus \Z$ filtered chain complex $\CFK^\infty(K)$ over $\bF_2[U]$. The  complexes $\cCFL^\infty$ are graded by the Alexander grading and
\[
\CFK^\infty(K)=\cCFL^\infty(S^3,\bK)_0,
\]
 where $\cCFL^\infty(S^3,\bK)_0\subset\cCFL^\infty(S^3,\bK)$ is the subset in zero Alexander grading. The standard action of $U$ on $\CFK^\infty(K)$ corresponds to the action of the product $UV$ on $\cCFL^\infty(S^3,\bK)$, in the notation of this paper. As the map $\iota_K$ satisfies $A(\iota_K(\ve{x}))=-A(\ve{x})$, the map $\iota_K$ preserves the subset of zero Alexander grading, so  Theorem \ref{thm:B} can be restated for the more standard complex $\CFK^\infty(K)$. Nonetheless it's worth noting that $\Phi_i$ and $\Psi_i$ are not individually endomorphisms of $\CFK^\infty(K_i)$ (as they shift Alexander grading by $+1$ and $-1$ respectively), but the tensor product $\Phi_1|\Psi_2$ shifts Alexander grading by zero, and hence is an endomorphism of the tensor product of the two $\CFK^\infty$ complexes. For the purposes of this paper, it is easier to use the complexes $\cCFL^\infty$ than the $\CFK^\infty$ complexes, since it's convenient to work with the maps $\Phi_i$ and $\Psi_i$ individually.

In a different direction, due to the work of Taubes \cite{TaubesECH=SW1},  Kutluhan, Lee and Taubes  \cites{KLTHF=HM1,KLTHF=HM2,KLTHF=HM3,KLTHF=HM4,KLTHF=HM5}, and Colin, Ghiggini and Honda  \cites{CGHHF=ECH0,CGHHF=ECH1,CGHHF=ECH2,CGHHF=ECH3}  there are isomorphisms between certain versions of Heegaard Floer homology and the monopole Floer homology groups developed by Kronheimer and Mrowka  \cite{KMMonopole}. Monopole Floer homology is an $S^1$-equivariant theory, however the action of $S^1$ can be enlarged to an action of $\Pin(2)$, and there are now constructions which take into account the full $\Pin(2)$ action \cite{ManPin2Triangle} \cite{LinPin2Monopole}. The modules $\HFI^-(Y,\frs)$ from \cite{HMInvolutive} are expected to correspond to the $\Z_4$-equivariant theory for the subgroup $\Z_4=\langle j\rangle\subset \Pin(2)$.  Lin proves a connected sum formula for $\Pin(2)$-equivariant monopole Floer homology  in terms of an $\cA_\infty$ tensor product \cite{LinConnSums}. Our connected sum formula for involutive link Floer homology is different, since the involutive complex is determined entirely by the chain homotopy type of the map $\iota_K$. It would be interesting to have an interpretation of Theorem~\ref{thm:B} in terms of an $\cA_{\infty}$ tensor product, but such an interpretation is not currently available.

\subsection{Organization} In Section \ref{sec:algebraicpreliminaries} we construct the group $\frI_K$ of $\iota_K$-complexes modulo local equivalence. In Section \ref{sec:backgroundandconjugation} we provide background on the link Floer complexes and the link Floer TQFT. In Section \ref{sec:furtherproperties}, we prove some new results about the link Floer TQFT, such as Theorem \ref{thm:C} and Lemma \ref{thm:D} (conjugation invariance and the bypass relation). In Section \ref{sec:connectedsumsandcobordisms}, we prove that several naturally defined link cobordism maps induce filtered chain homotopy equivalences for connected sums of 3-manifolds and knots. In Section \ref{sec:connectedsumformulaforiotaK} we prove Theorem \ref{thm:B}, the connected sum formula for the map $\iota_K$. In Section \ref{sec:homomorphism}, we combine our results to construct a homomorphism from $\cC$ to $\frI_K$. In Section \ref{sec:examples} we compute some examples, mostly using the computer algebra system \textit{Macaulay2} \cite{Mac2}. Finally in Appendix \ref{app:directverification}, we provide some alternate proofs of some of the results of the paper.

\subsection{Acknowledgments}

The author would like to thank Kristen Hendricks, Jen Hom and Maciej Borodzik for suggesting this problem, and for some useful additional suggestions and discussions. In addition, the author owes an intellectual debt to Kristen Hendricks and Ciprian Manolescu, who coauthored a similar paper with the author and provided many of the ideas and arguments in that context. Ciprian Manolescu also wrote parts of the code which we use to compute the invariants $\bar{V}_0$ and $\underline{V}_0$ in the last section. The author would like to thank Sucharit Sarkar and Matt Stoffregen for useful conversations.

\section{The group $\frI_K$}
\label{sec:algebraicpreliminaries}

In this section we define $\frI_K$, the group generated by chain complexes with a homotopy automorphism, up to local equivalence.

\subsection{$\iota_K$-complexes}

In this section we describe the elements of $\frI_K$, and state the group structure. In the following subsections we will prove some basic facts, and prove that $\frI_K$ is an abelian group. Throughout, we let $\cR$ denote the ring 
\[
\cR:=\bF_2[U,V,U^{-1},V^{-1}].
\]

\begin{define}Suppose that $C$ and $C'$ are two chain complexes over $\cR$ and $F\colon C\to C'$ is a homomorphism of abelian groups. We say $F$ is $\cR$ \emph{skew-equivariant} if $F(U\cdot \ve{x})=V\cdot F(\ve{x})$ and $F(V\cdot \ve{x})=U\cdot F(\ve{x})$. If $C$ and $C'$ have filtrations by $\Z\oplus \Z$, denoted $\cG$ and $\cG'$ respectively, we say that $F$ is \emph{skew-filtered} if $F(\cG_{(i,j)})\subset \cG'_{(j,i)}$. If $C$ and $C'$ are both bigraded with two gradings $\gr_U$ and $\gr_V$, we will say that $F$ is \emph{skew-graded} if $\gr_U(F(\ve{x}))=\gr_V(\ve{x})$ and $\gr_V(F(\ve{x}))=\gr_{U}(\ve{x})$ for homogeneous elements $\ve{x}$.
\end{define}

We will write $\simeq$ for $\cR$-equivariant, filtered chain homotopy equivalence, unless we specify otherwise. Similarly $\eqsim$ will mean $\cR$ skew-equivariant, skew-filtered chain homotopy equivalences, unless specify otherwise.

Suppose $(C,\d)$ is a chain complex such that $C$ is a free $\cR$-module and $\d$ is a homomorphism of $\cR$-modules. If $B$ is a basis of $C$ over $\cR$, we can define a filtration $\cG^B$ on $C$ over $\Z\oplus \Z$, by defining $\cG^B_{(i,j)}$ to be the subset generated by elements of the form $U^mV^n\cdot \ve{x}$, for $\ve{x}\in B$ and $m\ge i$ and $n\ge j$. If $\ve{x}\in B$, we can write

\[
\d(\ve{x})=\sum_{\ve{y}\in B} P_{\ve{x}\ve{y}} \cdot \ve{y},
\]
 for a unique $P_{\ve{x}\ve{y}}\in \cR$.  Since $\d$ is a homomorphism of $\cR$-modules, the polynomials $P_{\ve{x}\ve{y}}$ determine $\d$ on all of $C$.

  If all of the $P_{\ve{x}\ve{y}}$ are in $\bF_2[U,V]$ (i.e. involve only nonnegative powers of $U$ and $V$) then $(C,\d)$ is a $\Z\oplus \Z$ filtered chain complex with respect to the filtration $\cG^B$. Given a basis $B$, we can formally differentiate each $P_{\ve{x}\ve{y}}$, with respect to $U$ or $V$. We define the maps
  $\Phi_B$ and $\Psi_B$ on basis elements via the formulas
\begin{equation}\Phi_B(\ve{x})=\sum_{\ve{y}\in B} \left(\tfrac{d}{dU} P_{\ve{x}\ve{y}}\right)\cdot \ve{y}\qquad \text{and} \qquad \Psi_B(\ve{x})=\sum_{\ve{y}\in B} \left(\tfrac{d}{dV} P_{\ve{x}\ve{y}}\right)\cdot \ve{y}.\label{eq:formalderivatives}\end{equation}
We extend $\Phi_B$ and $\Psi_B$ linearly over $\cR$ to obtain $\cR$-equivariant maps. Note that the the derivatives $d/dU$ and $d/dV$ are applied only to the polynomials $P_{\xs\ys}$, not to the arguments of $\Phi_B$ and $\Psi_B$. So, for example,
\[
\Phi_B(U^i\cdot \ve{x})=U^i\sum_{\ys\in B} (\tfrac{d}{dU} P_{\xs\ys})\cdot \ys.
\]
  
 We call $\Phi_B$ and $\Psi_B$ the \emph{formal derivatives} of the differential $\d$.  Note that $\d$ is a filtered map with respect to $\cG^B$ if and only if the matrix for $\d$ in the basis $B$ involves only nonnegative powers of $U$ and $V$. Thus if $\d$ is filtered with respect to $\cG^B$, then the maps $\Phi_B$ and $\Psi_B$ are also filtered with respect to $\cG^B$.

We now define the objects which will form the elements of the group $\frI_K$:

\begin{define}\label{def:iotaKcomplex}
We say a tuple $(C,\d,B,\iota)$ is an $\iota_K$-\emph{complex} if the following hold:
\begin{enumerate}
\item\label{defiota:cond1} $(C,\d)$ is a finitely generated free chain complex over $\cR$ with basis $B$.
\item\label{defiota:cond2}  $\d$ is a filtered map with respect to $\cG^B$.
\item \label{defiota:cond3}The elements of $B$ are assigned two gradings, $\gr_U$ and $\gr_V$, taking values in $\Z$. Extending $\gr_U$ and $\gr_V$ to the entire complex, by declaring $U$ to have $\gr_U$-grading change $-2$, $V$ to have $\gr_U$-grading change 0, and $V$ to have $\gr_V$-grading change $-2$ and $U$ to have $\gr_V$-grading change zero, the differential is $-1$ graded with respect to both $\gr_U$ and $\gr_V$.
\item\label{defiota:cond4} There is a grading preserving isomorphism  $H_*(C,\d)\iso \cR$ where $1\in \cR$ has $\gr_U$ and $\gr_V$-grading 0.
\item\label{defiota:cond5} $\iota$ is a skew-filtered, skew-graded, $\cR$ skew-equivariant endomorphism of $C$.
\item\label{defiota:cond6} $\iota^2\simeq \id+\Phi_B\circ\Psi_B$.
\end{enumerate}
\end{define}

We refer the reader to Section~\ref{sec:simpleexamples} for some examples of $\iota_K$-complexes.

\begin{rem}It will be convenient to combine the two gradings $\gr_U$ and $\gr_V$ into an \emph{Alexander grading}, by defining $A=\tfrac{1}{2}(\gr_U-\gr_V)$. Notice that since $\iota$ switches $\gr_U$ and $\gr_V$, we have that $A(\iota(\ve{x}))=-A(\ve{x})$ for homogeneous $\ve{x}$. Hence given an $\iota_K$ complex $(C,\d,B,\iota)$, we can consider the subset $C_0$ of $C$ concentrated in zero Alexander grading. On $C_0$, the two gradings $\gr_U$ and $\gr_V$ coincide, and furthermore $\iota$ restricts to a homotopy automorphism of $C_0$. The group $C_0$ is no longer an  $\cR$-module, but instead a $\bF_2[\hat{U},\hat{U}^{-1}]$-module, where $\hat{U}=UV$.
\end{rem}

There is an obvious notion of homotopy equivalence between two $\iota_K$ complexes, but this notion is too strong to define a group structure on $\frI_K$. Instead we use a weaker notion of equivalence from \cite{StoffregenSeifertFibered} in the definition of $\frI_K$:

\begin{define}We will say two $\iota_K$ complexes, $\cC_1=(C_1,\d_1,B_1,\iota_1)$ and $\cC_2=(C_2,\d_2,B_2,\iota_2)$, are \emph{locally equivalent} if there are filtered, grading preserving $\cR$-equivariant chain maps
\[
F\colon C_1\to C_2\qquad \text{and} \qquad G\colon C_2\to C_1,
\]
 such that
\[
\iota_2F\eqsim F\iota_1 \qquad \text{and} \qquad \iota_1G\eqsim G\iota_2
\]
 such that $F$ and $G$ are isomorphisms on homology. If in addition we have that  $F\circ G\simeq \id$ and $G\circ F\simeq \id,$  we will say that $\cC_1$ and $\cC_2$ are \emph{homotopy equivalent}.

\end{define}

Note that homotopy equivalent $\iota_K$-complexes are also locally equivalent.

Given two $\iota_K$-complexes $\cC_1=(C_1,\d_1, B_1,\iota_1)$ and $\cC_2=(C_2,\d_2,B_2,\iota_2)$, there are two $\iota_K$-complexes which can naturally be described as the product of $\cC_1$ and $\cC_2$. We define two products, $\times_1$ and $\times_2$, via the formulas
\[
(\cC_1\times _i \cC_2)=(C_1\otimes_{\cR} C_2, \d_1|\id+\id| \d_2,B_1\times B_2, \iota_1\times_i\iota_2),
\]
where
\begin{equation}\iota_1\times_1 \iota_2=\iota_1| \iota_2+\Phi_1\iota_1| \Psi_2\iota_2\qquad \text{and}\qquad \iota_1\times_2 \iota_2=\iota_1| \iota_2+\Psi_1\iota_1| \Phi_2\iota_2.\label{eq:productformula} \end{equation}

We can now define the group $\frI_K$:

\begin{define}\label{def:IK}We define $\frI_K$ to be the set generated local equivalence classes of $\iota_K$-complexes. Group multiplication is given by either of $\times_1$ or $\times_2$. The inverse of $\cC=(C,\d,B,\iota)$ is given by $\cC^\vee=(C^\vee,\d^\vee,B^\vee,\iota^\vee)$, where $C^\vee=\Hom_{\cR}(C,\cR)$ and the other terms are defined similarly. If $\ve{x}\in B$, then one defines $\gr(\ve{x}^\vee)=-\gr(\ve{x})$, for $\gr\in \{\gr_U,\gr_V\}$. The identity complex is $\cC_e=(\cR,0,\{1\},\iota_e)$, where $\iota_e\colon \cR\to \cR$ fixes $1\in \cR$ and switches $U$ and $V$.
\end{define}

We will prove the following over the course of the following subsections:

\begin{prop}\label{prop:IKisagroup} The pair $(\frI_K,\times)$ is a well defined abelian group.
\end{prop}

\subsection{Properties of $\Phi_B$ and $\Psi_B$}
In this section, we describe some useful properties of the maps $\Phi_B$ and $\Psi_B$ defined in Equation~\eqref{eq:formalderivatives}. Recall that if $B$ is a basis of $(C,\d)$ over $\cR$, we defined $\Phi_B$ and $\Psi_B$ to be the formal derivatives of $\d$ with respect to $U$ and $V$ (respectively) in terms of the basis $B$.

 As a first property, we note that the Leibniz rule for differentiating polynomials implies that
\begin{equation}
\Phi_B=\d \circ \tfrac{d}{d U}\big|_B+ \tfrac{d}{d U}\big|_B\circ \d\qquad \text{and} \qquad \Psi_B=\d \circ \tfrac{d}{d V}\big|_B+\tfrac{d}{d V}\big|_B\circ \d.
\label{eq:nonequivchainhomotopy}
\end{equation} In particular the maps $\Phi_B$ and $\Psi_B$ vanish on homology. Note that the derivative maps $(d/dU)|_B$ and $(d/dV)|_B$ are neither filtered nor $\cR$-equivariant, so the maps $\Phi_B$ and $\Psi_B$ may still be nonzero in the group of chain maps modulo equivariant, filtered chain homotopy. 

\begin{define}\label{def:skewedcomplex}Given an $\iota_K$ complex $\cC=(C,\d,B,\iota)$ we define the \emph{skew} of $\cC$, denoted $\tilde{\cC}=(\tilde{C},\tilde{\d},\tilde{B},\tilde{\iota})$ as follows. As an $\bF_2$-module with differential, we define $(\tilde{C},\tilde{\d})=(C,\d)$, but we give $\tilde{C}$ a new $\cR$-action. We declare $U$ to act on $\tilde{C}$ the same as $V$ acts on $C$. Similarly we declare $V$ to  act on $\tilde{C}$ the same as $U$ acted on $C$.   We set $\tilde{B}=B$ and $\tilde{\iota}=\iota$. For $\ve{x}\in B$ we define $\tilde{\gr}_U(\ve{x})=\gr_{V}(\ve{x})$ and $\tilde{\gr}_V(\ve{x})=\gr_{U}(\ve{x})$.
	\end{define}

Note that by definition $\cG_{(i,j)}^{\tilde{B}}=\cG^B_{(j,i)}$. Also, there is a natural map
\begin{equation}
s\colon C\to \tilde{C},\label{eq:skewmap}
\end{equation} which is the identity map on the $\bF_2$-module $C$, but as a map between $\iota_K$-complexes is skew-equivariant, skew-filtered and skew-graded. Tautologically, we have 
\begin{equation}s\circ \Phi_B=\Psi_{\tilde{B}}\circ s\qquad \text{and}\qquad s\circ \Psi_B=\Phi_{\tilde{B}}\circ s.\label{eq:skewedmaps}\end{equation}

More generally, if $(C_1,\d_1)$ and $(C_2,\d_2)$ are two chain complexes over $\cR$, we can consider the morphism complex of $\cR$-equivariant maps
\[
\Mor_{\cR}(C_1,C_2),
\]
which has differential
\begin{equation}
\d_{\Mor}(f)=f\circ \d_1+\d_2\circ f.\label{eq:morcomplexdifferential}
\end{equation} Similarly there is the skewed morphism complex
\[
\tilde{\Mor}_{\cR}(C_1,C_2),
\]
of skew-equivariant morphisms, which can be given a differential using the same formula from Equation~\eqref{eq:morcomplexdifferential}. Notice that
\[
\tilde{\Mor}_{\cR}(C_1,C_2)\iso \Mor_{\cR}(\tilde{C}_1,C_2)\iso \Mor_{\cR}(C_1,\tilde{C}_2).
\]
The map $s\colon C\to \tilde{C}$ from Equation~\eqref{eq:skewmap} is simply the image of $\id_C$ under the map $\tilde{\Mor}_{\cR}(C,\tilde{C})\iso \Mor_\cR(C,C)$.

\begin{lem}\label{lem:equivariantmapscommutewithPhiandPsi}Suppose that $(C_1,\d_1,B_1)$ and $(C_2,\d_2,B_2)$ are two complexes over $\cR$ with bases $B_1$ and $B_2$. If $F\colon C_1\to C_2$ is $\cR$-equivariant then
\[
\Phi_2 \circ F+F\circ \Phi_1\simeq 0\qquad \text{and} \qquad \Psi_2\circ  F+F\circ \Psi_1\simeq 0,
\]
 through $\cR$-equivariant chain homotopies. Similarly, if $G\colon C_1\to C_2$ is $\cR$ skew-equivariant, then
\[
\Psi_2\circ G+G\circ \Phi_1\eqsim 0\qquad \text{and}\qquad \Phi_2\circ  G+G\circ  \Psi_1\eqsim 0,
\] 
through skew-equivariant chain homotopies. If $F$  (resp. $G$) is filtered (resp. skew-filtered) with respect to $\cG^B$ then the chain homotopies can be taken to be filtered (resp. skew-filtered).
\end{lem}
\begin{proof} Suppose $F$ is $\cR$-equivariant. Writing $\d_1,$ $\d_2$ and $F$ in terms of matrices with respect to $B_1$ and $B_2$, we take the expression $\d_2\circ F+F\circ\d_1=0$ and differentiate with respect to $U$, using the Leibniz rule, to see
\[
\Phi_{B_2}\circ F+F\circ \Phi_{B_1}=F'\circ \d_1+\d_2\circ F',
\]
 where $F'$ denotes the result of differentiating the matrix for $F$ with respect to $U$. If $F$ is a filtered map, then $F'$ will be as well. The analogous argument works for $\Psi_B$.

Now suppose that $G\colon C_1\to C_2$ is skew-equivariant. Note that $G=G\circ s \circ s$, and $G\circ s$ is an $\cR$-equivariant map. Hence we just apply the previous result together with Equation \eqref{eq:skewedmaps} to see that
\[
\Psi_2\circ G+G\circ\Phi_1\eqsim 0\qquad \text{and}\qquad \Phi_2 \circ G+G\circ \Psi_1\eqsim 0.
\] 
If $G$ is skew-filtered, then $G\circ s$ is filtered, so the previous argument applies to show that the chain homotopies can be taken to be skew-filtered.
\end{proof}

We now prove that the maps $\Phi_B$ and $\Psi_B$ depend only on the filtration $\cG^B$ and not on the particular choice of basis, in the following sense:

\begin{cor}If $B$ and $B'$ are two bases for the complex $(C,\d)$, then
\[
\Phi_B\simeq \Phi_{B'}\qquad \text{and} \qquad \Psi_B\simeq \Psi_{B'},
\]
 through $\cR$-equivariant chain homotopies. If in addition $\cG^B=\cG^{B'}$, then the chain homotopies can be taken to be filtered with respect to $\cG^B=\cG^{B'}$.

\end{cor}

\begin{proof} If $B$ and $B'$ are arbitrary bases, we apply the first part of Lemma \ref{lem:equivariantmapscommutewithPhiandPsi} to the map $\id\colon C\to C$, which is an $\cR$-equivariant map, to see that $\Phi_B\simeq \Phi_{B'}$ and $\Psi_B\simeq \Psi_{B'}$ through $\cR$-equivariant chain homotopies. If $\cG^B=\cG^{B'}$, then $\id$ is also filtered, so applying the second part of Lemma \ref{lem:equivariantmapscommutewithPhiandPsi} shows that the chain homotopies can be taken to be filtered.
\end{proof}

Note that if $B$ and $B'$ are two different bases, the matrix for the identity map between $B$ and $B'$ will not be the identity matrix. As such, the chain homotopy in the previous lemma will not  vanish in general.

\begin{lem}\label{lem:PhiPsicommute}If $(C,\d)$ is a free chain complex over $\cR$ with basis $B$, then $\Phi_B$ and $\Psi_B$ are chain maps, and
\[
\Phi_B\circ\Psi_B\simeq \Psi_B\circ\Phi_B,
\] 
through $\cR$-equivariant chain homotopies. If $\d$ is filtered with respect to $\cG^B$, then the chain homotopies can be taken to be filtered with respect to $\cG^B$.
\end{lem}
\begin{proof}Taking the matrix expression for $\d$ in the basis $B$, one differentiates the expression $\d^2=0$  with respect to $U$ to get $\Phi_B\circ\d+\d\circ\Phi_B=0$. Taking $\Phi_B\circ\d+\d\circ \Phi_B=0$ and differentiating it with respect to $V$, one sees that $\Phi_B\circ\Psi_B+\Psi_B\circ\Phi_B\simeq 0$. If $\d$ is filtered, then so is $\Phi_B$. Similarly if $\Phi_B$ is filtered, then so is the chain homotopy between $\Phi_B\circ \Psi_B+\Psi_B\circ \Phi_B$ and $0$, since it is obtained by differentiating the matrix for $\Phi_B$ with respect to $V$.
\end{proof}

From Equation~\eqref{eq:nonequivchainhomotopy}, we see that the maps $\Phi_B$ and $\Psi_B$ are chain homotopic to zero, though not necessarily through filtered $\cR$-equivariant chain homotopies. On the other hand, we have the following:

\begin{lem}\label{lem:Phi^2Psi^2=0}If $(C,\d)$ is a free chain complex over $\cR$, with a basis $B$, then $(\Phi_B)^2 \simeq 0$ and $(\Psi_B)^2\simeq 0$, through $\cR$-equivariant chain homotopies. If $\d$ is a filtered map with respect to $\cG^B$, then the chain homotopies can also be taken to be filtered.
\end{lem}
\begin{proof}The proof is the same as \cite{ZemCFLTQFT}*{Lemma~4.9}. For $\Phi_B$, the chain homotopy is defined by first writing $\d=\sum_{n=-\infty}^\infty P_n \cdot U^n$ where $P_n$ is a matrix (in terms of the basis $B$) with entries in $\bF_2[V,V^{-1}]$. We define
\[
H_B:=\sum_{n=-\infty}^{\infty}\left( \frac{n(n-1)}{2}\right) P_n\cdot U^{n-2}.
 \]
  It is easy to check that
\[
\Phi_B^2=\d\circ H_B+H_B\circ \d.
\]
 Finally, if $\d$ is filtered with respect to $\cG^B$, then $P_n$ vanishes for $n<0$ and $P_n$ involves only nonnegative powers of $V$, so $H_B$ is filtered as well.
\end{proof}

\begin{rem} Since we are working over $\bF_2$, Lemmas~\ref{lem:PhiPsicommute} and~\ref{lem:Phi^2Psi^2=0} imply that if $(C,\d,B,\iota)$ is an $\iota_K$-complex, then $\iota^4\simeq \id$, since
\[
\iota^4\simeq (\id+\Phi \Psi)^2\simeq \id+2\Phi\Psi+\Phi^2\Psi^2\simeq \id.
\]
\end{rem}

\subsection{$\frI_K$ is an abelian group}\label{subsec:thegroupI_K}

In this section we verify  $\frI_K$ is a well defined abelian group.

We first show that the product of two $\iota_K$-complexes is an $\iota_K$-complex:

\begin{lem}\label{lem:multwelldefined} If $\cC_1=(C_1,\d_1,B_1,\iota_1)$ and $\cC_2=(C_2,\d_2,B_2,\iota_2)$ are two $\iota_K$-complexes, then $\cC_1\times_1\cC_2$ and $\cC_1\times_2\cC_2$ are both $\iota_K$-complexes.
	\end{lem}
	\begin{proof} Only conditions \eqref{defiota:cond4} and \eqref{defiota:cond6} of Definition~\ref{def:iotaKcomplex} are not entirely straightforward to check. Condition \eqref{defiota:cond4} states $H_*(C_1\otimes_\cR C_2)\iso \cR$.  By our grading assumptions, we can write $C_1$ and $C_2$ each as a direct sum (over $\bF_2$) over Alexander gradings (where $A:=\tfrac{1}{2}(\gr_U-\gr_V)$). Define the subring
	\[ 
	\cR_0:=\bF_2[\hat{U},\hat{U}^{-1}]\subset \cR,
	\]  where $\hat{U}=UV$.	
	Write $(C_i)_{j}$ for the subset of $C_i$ in Alexander grading $j$. We note that each $(C_i)_{j}$ is a finitely generated, free $\cR_0$-module.  Since $\cR_0$ is a PID, the classification of free, finitely generated chain complexes over a PID (see e.g. \cite{HMZConnectedSum}*{Lemma 6.1}) implies that we can write $(C_i)_{0}$ as a direct sum of 1-step and 2-step complexes. We recall that a 1-step complex is a copy of $\cR_0$ with vanishing differential, and a 2-step complex is a complex over $\cR_0$ with two generators $a$ and $b$ over $\cR_0$,   such that $\d(a)=f(\hat{U})\cdot b$ for some $f(\hat{U})\in \cR_0$. Since $(C_i)_{0}$ has a $\Z$-grading, it is not hard to see that the 2-step complexes must be of the form $\d(a)=\hat{U}^n\cdot  b$. This decomposition of $(C_i)_{0}$ as a sum of 1-step and 2-step complexes over $\cR_0$ can be extended linearly over $\cR$ to a decomposition of $C_i$ as a sum of 1-step and 2-step complexes over the ring $\cR$.

 Over $\cR_0$ and $\cR$, 2-step complexes with generators $a$ and $b$ such that $\d(a)=\hat{U}^n \cdot b$ are $\cR$-equivariantly chain homotopy equivalent to the zero complex, since $\hat{U}^n=U^n V^n$ is invertible in $\cR$ (note that of course such a 2-step complex is not necessarily \emph{filtered} chain homotopy equivalent to the zero complex). The homology $H_*(C_1\otimes_{\cR} C_2)$ only depends on the chain homotopy type of $C_i$ over $\cR$, so such 2-step complexes can be deleted from each $C_i$. Since by assumption the homology of $C_{i}$ is isomorphic to $\cR$, with $1\in \cR$ given zero $\gr_U$ and $\gr_V$-grading, we know that both $C_1$ and $C_2$ are $\cR$-equivariantly chain homotopy equivalent  to a 1-step complex. Hence  $H_*(C_1\otimes_\cR C_2)\iso \cR$.

 The other condition worth discussing is condition \eqref{defiota:cond6}, which requires that $(\iota_1\times_i \iota_2)^2\simeq \id+\Phi\circ \Psi$. To see this, note that on $C_1\otimes_{\cR} C_2$, we have 
		\begin{equation}\id+\Phi\circ\Psi=\id|\id+(\Phi_1|\id+\id|\Phi_2)(\Psi_1|\id+\id|\Psi_2).\label{eq:involutionsquared}\end{equation} On the other hand
		\begin{align*}(\iota_1\times_1 \iota_2)^2&=(\iota_1|\iota_2+\Phi_1\iota_1|\Psi_2\iota_2)(\iota_1|\iota_2+\Phi_1\iota_1|\Psi_2\iota_2)\\
		&=\iota_1^2|\iota_2^2 +\Phi_1\iota_1^2|\Psi_2\iota_2^2+\iota_1\Phi_1\iota_1|\iota_2\Psi_2\iota_2+\Phi_1\iota_1\Phi_1\iota_1|\Psi_2\iota_2\Psi_2\iota_2\\
		&\simeq (\id+\Phi_1\Psi_1)|(\id+\Phi_2\Psi_2)+\Phi_1(\id+\Phi_1\Psi_1)|\Psi_2(\id+\Phi_2\Psi_2)+\Psi_1(\id+\Phi_1\Psi_1)|\Phi_2(\id+\Phi_2\Psi_2)\\
		&\qquad+\Phi_1\Psi_1(\id+\Phi_1\Psi_1)|\Psi_2\Phi_2(\id+\Phi_2\Psi_2) \\
		&\simeq \id|\id+\Phi_1|\Psi_2+\Psi_1|\Phi_2+\Phi_1\Psi_1|\id+\id|\Phi_2\Psi_2\end{align*} which is the same as Equation \eqref{eq:involutionsquared}. A similar computation yields that $(\iota_1\times_2\iota_2)^2\simeq \id+\Phi\circ \Psi$, as well.
		\end{proof}

We now show that the two multiplications $\times_1$ and $\times_2$ coincide on homotopy classes of $\iota_K$-complexes (and hence local equivalence classes):

\begin{lem}\label{lem:productsarestronglyequivalent} If $\cC_1=(C_1,\d_1,B_1,\iota_1)$ and $\cC_2=(C_2,\d_2,B_2,\iota_2)$ are two $\iota_K$-complexes, then the $\iota_K$-complexes $\cC_1\times_1 \cC_2$ and $\cC_1\times_2 \cC_2$ are homotopy equivalent.
\end{lem}
\begin{proof} We define  map $F\colon C_1\otimes C_2\to C_1\otimes C_2$ by
\[
F=(\id| \id+\Psi_1| \Phi_2).
\] 
 Note that since $\Psi_1^2\simeq 0$ and $\Phi_2^2\simeq 0$, and since we are working in characteristic 2, it follows that $F^2\simeq \id$. We claim that $F$ induces a homotopy equivalence of $\iota_K$-complexes, with $F$ being it's own homotopy inverse. We need to show that $F$ intertwines $\iota_1\times_1\iota_2$ and $\iota_1\times_2\iota_2$. We first compute
\begin{align*}F(\iota_1\times_1 \iota_2)&\eqsim (\id| \id+\Psi_1| \Phi_2)(\id|\id+\Phi_1|\Psi_2)(\iota_1|\iota_2)\\
&\eqsim (\id| \id+\Psi_1| \Phi_2)(\iota_1|\iota_2)(\id|\id+\Psi_1|\Phi_2)\\
&\eqsim (\iota_1\times_2\iota_2)F.
\end{align*}
In a similar manner, we have
\begin{align*}F(\iota_1\times_2 \iota_2)& \eqsim (\id|\id+\Psi_1|\Phi_2)(\id|\id+\Psi_1|\Phi_2)(\iota_1|\iota_2)\\
&\eqsim (\iota_1|\iota_2)\\
&\eqsim (\iota_1|\iota_2)(\id|\id+\Psi_1|\Phi_2)(\id|\id+\Psi_1|\Phi_2)\\
&\eqsim (\id|\id+\Phi_1|\Psi_2)(\iota_1|\iota_2)(\id|\id+\Psi_1|\Phi_2)\\
&\eqsim (\iota_1\times_1\iota_2)F,
\end{align*}
completing the proof.
\end{proof}

In light of the previous lemma, we will write $\times$ for the multiplication on $\frI_K$, when no confusion will arise. We now show that $\times$ is well defined on local equivalence classes:

\begin{lem}\label{lem:multwelldefinedonlocalclass} The $\times$ operation is a well defined multiplication on $\frI_K$. That is, if $\cC_1$ and $\cC_1'$ are locally equivalent, and $\cC_2$ and $\cC_2'$ are locally equivalent, then $\cC_1\times\cC_2$ and $\cC_1'\times\cC_2'$ are locally equivalent.
\end{lem}
\begin{proof}Let us consider changing the first factor, since changing the second factor follows similarly. Suppose that $F\colon C_1\to C_1'$ and $G\colon C_1'\to C_1$ are maps which induce a local equivalence. Since the maps $H_*(C_1)\otimes H_*(C_2)\to H_*(C_1\otimes C_2)$ and $H_*(C_1')\otimes H_*(C_2)\to H_*(C_1'\otimes C_2)$ are isomorphisms by the argument in Lemma \ref{lem:multwelldefined}, we see that $(F| \id)\circ (G| \id)=(F\circ G)| \id$ and $(G| \id)\circ (F| \id)$ are isomorphisms on homology. It remains only to show that $F| \id$ and $G|\id$ intertwine the involutions, up to skew-equivariant, skew-filtered homotopies.
	
For $F|\id$, this amounts to showing that
	\[
	(F|\id)(\iota_1|\iota_2+\Phi_1\iota_1|\Psi_2\iota_2)\eqsim (\iota_1'|\iota_2+\Phi_1'\iota_1'|\Psi_2\iota_2)(F|\id).\]
	 Rearranging terms reduces this to
	\[
	(F\iota_1+\iota_1'F)|\iota_2+(F\Phi_1\iota_1+\Phi_1'\iota_1'F)|\Psi_2\iota_2\eqsim 0.
	\]
	 By assumption, $F\iota_1\eqsim \iota_1'F$, since $F$ was part of a local equivalence. By Lemma \ref{lem:equivariantmapscommutewithPhiandPsi}, we also know that $F\Phi_1\simeq \Phi_1'F$. Combining these facts shows that the above identity is satisfied, so $F|\id$ intertwines the involutions, up to skew-equivariant, skew-filtered chain homotopies. The same argument works for $G$. The same argument also works for the second factor of the tensor product, and hence the proof is complete.
\end{proof}

\begin{lem}\label{lem:multassoc}
	The multiplication $\times$ is associative on $\frI_K$.
\end{lem}

\begin{proof}
	Suppose that $\cC_i=(C_i,\d_i,B_i,\iota_i)$ are $\iota_K$ complexes for $i\in \{1,2,3\}$. We will construct a homotopy equivalence between the two $\iota_K$ complexes $(\cC_1\times\cC_2)\times\cC_3$ and $\cC_1\times(\cC_2\times\cC_3)$. The map on chain complexes will be the identity. The only nontrivial thing to show is that the identity map intertwines the involutions, up to skew-equivariant, skew-filtered homotopy. This is a computation. We have
	\begin{equation}	\begin{split}(\iota_1\times_1\iota_2)\times_1\iota_3&=(\iota_1|\iota_2+\Phi_1\iota_1|\Psi_2\iota_2)|\iota_3+\big((\Phi_1|\id+\id|\Phi_2)(\iota_1|\iota_2+\Phi_1\iota_1|\Psi_2\iota_2)\big)|\Psi_3\iota_3\\
	&=\iota_1|\iota_2|\iota_3+\Phi_1\iota_1|\Psi_2\iota_2|\iota_3+\Phi_1\iota_1|\iota_2|\Psi_3\iota_3+\iota_1|\Phi_2\iota_2|\Psi_3\iota_3\\
	&
	\qquad +\Phi_1\iota_1|\Phi_2\Psi_2\iota_2|\Psi_3\iota_3+\Phi_1^2\iota_1|\Psi_2\iota_2|\Psi_3\iota_3.\label{eq:associative1}
	\end{split}
	\end{equation}
	On the other hand, we can also compute that
	\begin{equation}
	\begin{split}\iota_1\times_1(\iota_2\times_1\iota_3)&=\iota_1|(\iota_2|\iota_3+\Phi_2\iota_2|\Psi_3\iota_3)+\Phi_1\iota_1|\big((\Psi_2|\id+\id|\Psi_3)(\iota_2|\iota_3+\Phi_2\iota_2|\Psi_3\iota_3)\big)\\
	&=\iota_1|\iota_2|\iota_3+\iota_1|\Phi_2\iota_2|\Psi_3\iota_3+\Phi_1\iota_1|\Psi_2\iota_2|\iota_3+\Phi_1\iota_1|\iota_2|\Psi_3\iota_3\\
	&\qquad+\Phi_1\iota_1|\Psi_2\Phi_2\iota_2|\Psi_3\iota_3+\Phi_1\iota_1|\Phi_2\iota_2|\Psi_3^2\iota_3.
	\end{split}
	\label{eq:associative2}
	\end{equation}
The difference between the expressions in Equation~\eqref{eq:associative1} and \eqref{eq:associative2} is
	\[
	\Phi_1\iota_1|\Phi_2\Psi_2\iota_2|\Psi_3\iota_3+\Phi_1^2\iota_1|\Psi_1\iota_2|\Psi_3\iota_3+\Phi_1\iota_1|\Psi_2\Phi_2\iota_2|\Psi_3\iota_3+\Phi_1\iota_1|\Phi_2\iota_2|\Psi_3^2\iota_3,
	\]
	 which is skew-equivariantly, skew-filtered chain homotopic to 0, by Lemmas \ref{lem:PhiPsicommute} and  \ref{lem:Phi^2Psi^2=0}.
\end{proof}

We now consider inverses on $\frI_K$. We first prove the following:

\begin{lem}\label{lem:inverses1} If $\cC=(C,\d,B,\iota)$ is an $\iota_K$-complex, then the tuple $\cC^\vee=(C^\vee,\d^\vee,B^\vee,\iota^\vee)$ is an $\iota_K$-complex.
	\end{lem}
	\begin{proof}Most of the axioms are straightforward to verify. To see that $H_*(C^\vee)\iso \cR$, one can write down an explicit description of the complex $C$ over $\cR$, using the strategy of Lemma \ref{lem:multwelldefined} to decompose over Alexander gradings, and then apply the classification of chain complexes over the PID $\bF_2[\hat{U},\hat{U}^{-1}]$ where $\hat{U}=UV$,  to the subcomplex $C_0\subset C$ in 0 Alexander grading. The dual chain complex then has a simple description, obtained by dualizing the presentation of the complex $C_0$ over the ring $\bF_2[\hat{U}, \hat{U}^{-1}]$ and then extending linearly over $\cR$. From this description, the result follows immediately.
		\end{proof}

We now show that $\cC^\vee$ is the inverse of $\cC$ in $\frI_K$ (compare \cite{HMZConnectedSum}*{Proposition 8.8}):

\begin{lem}\label{lem:inverses2}If $\cC=(C,\d,B,\iota)$ is an $\iota_K$-complex, the product $\cC\times \cC^\vee$ is locally equivalent to $\cC_e=(\cR,0,\{1\},\iota_e),$ the identity $\iota_K$-complex.
	\end{lem}
	\begin{proof}We define two maps
	\[
F\colon \cR\to C\otimes_{\cR} C^\vee\qquad \text{and} \qquad G\colon C\otimes_{\cR} C^\vee\to \cR.	
	\]
	The map $G$ is defined to be the canonical trace map $\tr(a\otimes b)= b(a)$. The map $F$ is defined to be the cotrace is defined to be the dual of the canonical trace map $C^{\vee}\otimes_{\cR} C\to \cR$. If $B$ is a basis of $C$, the map $F$ satisfies
	\[
	F(1)=\cotr(1)=\sum_{\ve{x}\in B}\ve{x}\otimes \ve{x}^\vee.
	\] Notice that under the canonical identifications of $\Mor_{\cR}(\cR, C\otimes_{\cR}C^\vee)$ and $\Mor_{\cR}(C\otimes C^\vee, \cR)$ with $\Mor_{\cR}(C,C)$, the maps $F$ and $G$ are both identified as the identity map on $C$.
	
	The maps $F$ and $G$ are clearly filtered, and respect the grading. We need to show that they are chain maps, intertwine the involutions, and are isomorphisms on homology. That they are chain maps is an easy computation. To see that they are isomorphisms on homology, we note that the composition $G\circ F$ gives $\chi(C)\cdot \id_{\cR}$, where $\chi(C)$ is Euler characteristic of $C$ over $\cR$, modulo 2, which is $1$ since $H_*(C)\iso \cR$. Since $\cC,$ $\cC^\vee$ and $\cC\times \cC^\vee$ are all $\iota_K$-complexes, we know the associated homology groups are all $\cR$ so we conclude that $F$ and $G$ must be isomorphisms on homology.
	
	To see that $F$ intertwines the involutions, it is sufficient to show that
		\begin{equation}
		F\iota_e+(\iota|\iota^\vee+\Phi\iota|\Psi^\vee\iota^\vee)F\eqsim 0.\label{eq:iotacomplexcomp}
		\end{equation}
			Composing with $\iota_e$ on the right, it is sufficient to instead show that
\begin{equation}
F+(\iota|\iota^\vee+\Phi\iota|\Psi^\vee\iota^\vee)F\iota_e\simeq 0\label{eq:inversecomplexcomp}
\end{equation}
	
We now claim that if $M\in \tilde{\Mor}_{\cR}(C,C)$ and $N\in \tilde{\Mor}_{\cR}(C^\vee,C^\vee)$ then
\begin{equation}
(M|N)F\iota_e\simeq (MN^{\vee}|\id)F\simeq (\id|NM^{\vee})F.\label{eq:adjunction}
\end{equation}
We first consider the relation $(M|N)F\iota_e\simeq (MN^\vee|\id)F$. Since both sides are skew-equivariant, it is sufficient to show that the two maps agree on $1\in \cR$. Write $B=\{\xs_1,\dots, \xs_n\}$.	
 Write $M(\xs_i)=\sum_{j=1}^n m_{j,i}\cdot \xs_j$, for $m_{j,i}\in \cR$, and define coefficients $n_{j,i}\in \cR$ for $N$ similarly. Since $F(1)=\sum_{i=1}^n \xs_i\otimes \xs_i^\vee$, we simply compute
 \[
 (M|N)F(1)=\sum_{i,j,k} (m_{j,i} n_{k,i}) \cdot(\xs_j \otimes \xs_k^\vee)=(MN^\vee|\id)F(1),
\]
establishing the first relation in Equation~\eqref{eq:adjunction}. The second relation of Equation~\eqref{eq:adjunction} is similar.

Using Equation~\eqref{eq:adjunction}, it follow that
\begin{equation}
F+(\iota|\iota^\vee+\Phi\iota|\Psi^\vee\iota^\vee)F\iota_e\simeq ((\id+\iota^2+\Phi\iota^2\Psi)|\id)F\simeq 0.\label{eq:useadjunction}
\end{equation}
The last chain homotopy in Equation~\eqref{eq:useadjunction} follows from the manipulation 
\[
\id+\iota^2+\Phi\iota^2\Psi\simeq \id+(\id+\Phi\Psi)+\Phi(\id+\Phi\Psi)\Psi\simeq 0.
\]
  Equation~\eqref{eq:inversecomplexcomp} now follows, implying Equation~\eqref{eq:iotacomplexcomp} as well, so $F$ intertwines the involutions up to chain homotopy.

		The proof that $G$ intertwines the involutions is similar. 
		\end{proof}

We can now prove that $\frI_K$ is a well defined abelian group:

\begin{proof}[Proof of Proposition \ref{prop:IKisagroup}] Lemma \ref{lem:multwelldefined} shows that the product of two $\iota_K$-complexes under either multiplication $\times_1$ and $\times_2$ is an $\iota_K$-complex. Lemma \ref{lem:productsarestronglyequivalent} shows that the two products yield homotopy equivalent (and hence locally equivalent) $\iota_K$-complexes. Similarly, if $\cC_1$ and $\cC_2$ are two $\iota_K$-complexes, the natural map $ C_1\otimes_\cR C_2\to C_2\otimes_{\cR} C_1$ induces a homotopy equivalence of $\iota_K$-complexes between $\cC_1\times_1\cC_2$ and $\cC_2\times_2\cC_1$. An application of Lemma \ref{lem:productsarestronglyequivalent} now shows that $\cC_1\times_2\cC_2$ is homotopy equivalent to $\cC_1\times_1\cC_2$, from which we conclude that $\cC_1\times_1 \cC_2$ and $\cC_2\times_1\cC_1$ are homotopy equivalent, so $\frI_K$ is abelian. Lemma \ref{lem:multassoc} shows that multiplication is associative. Lemma \ref{lem:multwelldefinedonlocalclass} shows that multiplication is well defined on local equivalence classes. Lemmas \ref{lem:inverses1} and \ref{lem:inverses2} show that $\frI_K$ contains inverses.
\end{proof}

\section{Background on link Floer homology and the link Floer TQFT}\label{sec:backgroundandconjugation}

In this section, we provide background on link Floer homology, and describe some previous results about link cobordisms and link Floer homology.  The original construction of link Floer homology can be found in \cite{OSLinks}.

\subsection{Preliminaries on the link Floer complexes}
\label{sec:curvedcomplexesandeta}

\begin{define} An \emph{oriented multi-based link} in $Y$ is an ordered triple $\bL=(L,\ps,\qs)$ where $\ps$ and $\qs$ are two disjoint collections of basepoints on $L$, such that each component of $L$ contains at least two basepoints, and the basepoints in $\ps$ and $\qs$ alternate along each component of $L$. Furthermore, each component of $Y$ contains at least one component of $L$.
\end{define}

If $\bL=(L,\ps,\qs)$ is an oriented multi-based link in $Y$, we say that the first collection of basepoints, $\ps$, are of \emph{type-$\ws$} and we say the second collection of  basepoints, $\qs$, are of \emph{type-$\zs$}.

 Given a Heegaard diagram $\cH=(\Sigma, \ve{\alpha},\ve{\beta},\ve{p},\ve{q})$ representing the multi-based link $\bL=(L,\ve{p},\ve{q})$ in $Y$, we define the two tori
 \[
 \bT_{\as}:=\alpha_1\times \cdots \times \alpha_{g+n-1} \qquad \text{and} \qquad \bT_{\bs}:=\beta_1\times \cdots \times \beta_{g+n-1},
 \] inside of the symmetric product $\Sym^{g+n-1}(\Sigma)$,  where $n=|\ps|=|\qs|$.
 
 If $(\Sigma,\as,\bs,\ps)$ is a diagram for the multi-based 3-manifold $(Y,\ps)$, Ozsv\'{a}th and Szab\'{o} describe a map
\begin{equation}
\frs_{\ps}\colon \bT_{\as}\cap\bT_{\bs}\to \Spin^c(Y)\label{eq:spincmap}
\end{equation}
in \cite{OSDisks}*{Section~2.6}.

We will  consider a version of the  link Floer complex, which we will denote by
\[
\cCFL^\circ(\cH,\frs),
\]
 for $\frs\in \Spin^c(Y)$ and $\circ \in \{\wedge,-,\infty\}$. The ``$\infty$'' flavor of the complex is generated over $\bF_2$ by monomials of the form $U^i V^j\cdot  \ve{x}$, for arbitrary $i,j\in \Z$, and intersection points $\ve{x}\in \bT_{\as}\cap \bT_{\bs}$ with $\frs_{\ve{p}}(\ve{x})=\frs$. The ``$-$'' flavor is generated over $\bF_2$ by monomials with nonnegative $i$ and $j$. The hat flavor, denoted $\hat{\CFL}$, is generated over $\bF_2$ by intersection points $\xs\in \bT_{\as}\cap \bT_{\bs}$.

After picking a generic 1-parameter family of almost complex structures $J_s$ on $\Sym^{g+n-1}(\Sigma)$, one can define an endomorphism $\d$ which counts pseudo-holomorphic disks of Maslov index 1 inside of $\Sym^{g+n-1}(\Sigma)$, where $g=g(\Sigma)$ and $n=|\ve{p}|=|\ve{q}|$. Equivalently, one can count holomorphic curves inside of $\Sigma\times [0,1]\times \R$, using Lipshitz's cylindrical reformulation \cite{LipshitzCylindrical}.

In both formulations, the differential $\d$ is defined via the formula
\[
\d(\xs)=\sum_{\substack{\phi\in \pi_2(\xs,\ys)\\ \mu(\phi)=1}} \# \hat{\cM}(\phi) U^{n_{\ps}(\phi)} V^{n_{\qs}(\phi)}\cdot \ys.
\]
In the above formula, if $\phi\in \pi_2(\xs,\ys)$ is a class, the quantity $n_{\ps}(\phi)$ is the total multiplicity of $\phi$ over the basepoints in $\ps$, and $n_{\qs}(\phi)$ is defined similarly. The differential on $\hat{\CFL}$ counts disks with $n_{\ve{p}}(\phi)=n_{\qs}(\phi)=0$.

There is a slightly more general version of the full link Floer complex, which  is considered in \cite{ZemCFLTQFT}, where we have a distinct variable for each basepoint. In that setting, we write $U_p$ for the variable associated to a $p\in \ps$, and $V_q$ for variable associated to $q\in \qs$. In this more general setup, counting the ends of the moduli spaces of Maslov index two holomorphic disks shows yields the formula
\begin{equation}\d^2(\ve{x})=\sum_{K\in C(L)} (U_{p_{K,1}}V_{q_{K,1}}+V_{q_{K,1}}U_{p_{K,2}}+U_{p_{K,2}}V_{q_{K,2}}+\cdots +V_{q_{K,n_K}}U_{p_{K,1}})\cdot \ve{x}.\label{eq:d^2=k}
\end{equation} 
See \cite{ZemQuasi}*{Lemma~2.1} for a proof. Here $C(L)$ denotes the set of components of $L$, and if $K\in C(L)$, we are writing $p_{K,1},$ $q_{K,1},$ $\dots,$ $q_{K,n_{K}}$ for the basepoints on $K$, in the order that they appear on $K$. The above formula follows from the fact that the compactification of the moduli spaces of index 2 disks consists of broken strips (corresponding to $\d^2(\ve{x})$) as well as index 2 boundary degenerations (corresponding to the curvature term on the right side of Equation~\eqref{eq:d^2=k}).

In our present paper, since we set all of the $U_{p}$ variables equal and we set all of the $V_{q}$ variables equal, the expression in Equation~\eqref{eq:d^2=k} vanishes, and the module $\cCFL^\circ(\cH,\frs)$ becomes a genuine (i.e. not curved) chain complex.

For $\circ\in \{-,\infty\}$, the $\cR$-modules $\cCFL^\circ(\cH,\frs)$ have a $\Z\oplus \Z$ filtration given by powers of the variables $U$ and $V$. As pseudo-holomorphic disks have only nonnegative multiplicities on the Heegaard diagram, the differential is a filtered map.

Any two Heegaard diagrams for a link $\bL$ in $Y$ can be connected by a sequence of elementary Heegaard moves. To a sequence of elementary Heegaard moves between two diagrams $\cH_1$ and $\cH_2$ for a link $\bL$ in $Y$, we can define a filtered, equivariant chain homotopy equivalence $\Phi_{\cH_1\to \cH_2}$, following the construction in \cite{OSDisks}. A fundamental result is that the map $\Phi_{\cH_1\to \cH_2}$ is independent of the choice of elementary Heegaard moves, up to equivariant, filtered chain homotopy. A summary of the proof in our present context can be found in \cite{ZemCFLTQFT}*{Proposition~3.5}. See \cite{JTNaturality} for more on naturality in Heegaard Floer homology. We write
\[
\cCFL^\circ(Y,\bL,\frs),
\]
 to mean the collection of all the chain complexes $\cCFL^\circ(\cH,\frs)$ ranging over strongly $\frs$-admissible diagrams $\cH$ for $(Y,\bL)$, together with the change of diagrams maps.

\subsection{The conjugation action on link Floer homology}
\label{sec:conjugationaction}

In this section we provide some background about the conjugation action on link Floer homology.  Recall that $\Spin^c$ structures on $Y$ can be described as homology classes of non-vanishing vector fields. Given a non-vanishing vector field $v$ on $Y$ which corresponds to a $\Spin^c$ structure $\frs$,  the conjugate $\Spin^c$ structure $\bar{\frs}$ corresponds to the vector field $-v$.

 The map $\frs_{\ve{p}}$ from Equation~\eqref{eq:spincmap} depends on the choice of basepoints $\ve{p}$, in the following sense:

\begin{lem}\label{lem:changeSpincstructure}If $(\Sigma, \ve{\alpha},\ve{\beta},\ve{p},\ve{q})$ is a diagram for $(Y,L,\ve{p},\ve{q})$ and $\ve{x}\in \bT_{\as}\cap \bT_{\bs}$, then the maps $\frs_{\ve{p}}$ and $\frs_{\ve{q}}$ are related by the formula
\[
\frs_{\ve{p}}(\ve{x})-\frs_{\ve{q}}(\ve{x})=\PD[L].
\]
\end{lem}
\begin{proof}See \cite{RasLSpaceSurgeries}*{Equation (1)} or \cite{OSDisks}*{Lemma 2.19}.
	\end{proof}

 Given a Heegaard diagram $\cH=(\Sigma, \ve{\alpha},\ve{\beta},\ve{p},\ve{q})$ for $\bL=(L,\ve{p},\ve{q})$, we form the conjugate diagram $\bar{\cH}=(-\Sigma, \bar{\ve{\beta}},\bar{\ve{\alpha}},\ve{q},\ve{p})$ for $\bar{\bL}=(L,\ve{q},\ve{p})$. Here $\bL$ and $\bar{\bL}$ have the same underlying oriented link, and both have  basepoints at $\ve{p}\cup \ve{q}$, but the role of $\ve{p}$ and $\ve{q}$  as type-$\ve{w}$ or -$\ve{z}$ is switched.

There is a tautological chain isomorphism
\[
\eta\colon \cCFL_{J_s}^{\circ}(\cH,\frs)\to \cCFL_{-J_s}^{\circ}(\bar{\cH}, \bar{\frs}+\PD[L]).
\]
 The map $\eta$ sends an intersection point in $\ve{x}\in \bT_{\as}\cap \bT_{\bs}$ on $\cH$ to the corresponding intersection point in $ \bT_{\bar{\bs}}\cap \bT_{\bar{\as}}$. The map $\eta$ is $\cR$-equivariant.

\begin{lem}\label{lem:etacommuteswithchangeofdiagram}If $\cH_1$ and $\cH_2$ are two diagrams, then the following diagram commutes up to chain homotopy:
\[
\begin{tikzcd}\cCFL^{\circ}(\cH_1,\frs)\arrow{r}{\eta}\arrow{d}{\Phi_{\cH_1\to \cH_2}}& \cCFL^{\circ}(\bar{\cH}_1,\bar{\frs}+\PD[L])\arrow{d}{\Phi_{\bar{\cH}_1\to \bar{\cH}_2}}\\
\cCFL^{\circ}(\cH_2,\frs)\arrow{r}{\eta}& \cCFL^{\circ}(\bar{\cH}_2,\bar{\frs}+\PD[L]).
\end{tikzcd}
\]
\end{lem}
\begin{proof}The map $\Phi_{\cH_1\to \cH_2}$ is a computed by picking a sequence of intermediate diagrams which differ each from the previous by a single elementary Heegaard move. The map $\Phi_{\bar{\cH}_1\to \bar{\cH}_2}$ can be computed using the same sequence of Heegaard moves, and it is easy to verify that the diagram commutes on the nose for each elementary Heegaard move (though the change of diagrams maps themselves are only defined up to chain homotopy).
\end{proof}

As a consequence of Lemma~\ref{lem:etacommuteswithchangeofdiagram}, the map $\eta$ descends to a morphism between transitive chain homotopy types
\[
\eta\colon \cCFL^\infty(Y,\bL,\frs)\to \cCFL^\infty(Y,\bar{\bL},\bar{\frs}+\PD[L]),
\] 
 since it commutes up to filtered, equivariant chain homotopy, with the change of diagrams maps.

For a doubly based knot $(K,p,q)$ the conjugation action $\iota_K$ is defined as the composition 
\[
\iota_K:=\tau_K\circ \eta,
\]
 where $\tau_K$ denotes a half twist in the direction of the link's orientation, switching $p$ and $q$. The map $\eta$ is filtered and equivariant, while $\tau_K$ is skew-filtered and skew-equivariant since it switches the two basepoints.

\subsection{Maps associated to decorated link cobordisms}
\label{sec:linkcobordismmapsoverview}
In \cite{ZemCFLTQFT}, the author describes $\Spin^c$ functorial maps associated to decorated link cobordisms. To a decorated link cobordism 
\[
(W,\cF)\colon (Y_1,\bL_1)\to (Y_2,\bL_2)
\]
 in the sense of Definition \ref{def:decoratedlinkcob}, there is a  map
\[
F_{W,\cF,\frs}\colon \cCFL^\circ(Y_1,\bL_1,\frs|_{Y_1})\to \cCFL^\circ(Y_2,\bL_2,\frs|_{Y_2}),
\]
 which is an invariant of the decorated link cobordism, up to equivariant, filtered chain homotopy.  The identity cobordism induces the identity map, and together the maps satisfy the $\Spin^c$ composition law \cite{ZemCFLTQFT}*{Theorem~B}: If $(W,\cF)=(W_2,\cF_2)\circ (W_1,\cF_1)$ and $\frs_i\in \Spin^c(W_i)$, then 
\begin{equation}
F_{W_2,\cF_2,\frs_2}\circ F_{W_1,\cF_1,\frs_1}\simeq \sum_{\substack{\frs\in \Spin^c(W)\\
\frs|_{W_i}=\frs_i}} F_{W,\cF,\frs}.\label{eq:compositionlaw}
\end{equation}

We now briefly describe the construction of the link cobordism maps from \cite{ZemCFLTQFT}. More details  (such as explicit formulas)  can be found in our proof of Theorem \ref{thm:C}, below. Further details can be found in \cite{ZemCFLTQFT}.

The cobordism maps are defined by taking a decomposition of the link cobordism into simple pieces, and then defining a map for each piece. One defines maps for 4-dimensional handles attached away from $\bL$, maps associated to band surgery on the link $\bL$, as well as cylindrical cobordisms with simple but nontrivial  dividing sets. 

The 1-handle and 3-handle maps are approximately the same as in \cite{OSTriangles} (the difference being now that we allow 1-handles or 3-handles which connect two components or split a component of a 3-manifold). There are also new 0-handle and 4-handle maps, corresponding to a 4-ball containing a 2-dimensional disk in a standard way.

The 2-handle maps are very similar to those defined by Ozsv\'{a}th and Szab\'{o} in \cite{OSTriangles}, which count holomorphic triangles in a Heegaard triple representing surgery on a framed 1-dimensional link $\bS^1\subset Y$. 
Similarly the 2-handle maps from \cite{ZemCFLTQFT} are also defined by counting holomorphic triangles on a Heegaard triple $(\Sigma, \ve{\alpha},\ve{\beta},\ve{\beta}',\ve{p},\ve{q})$ representing homology classes of triples $\psi$ with $\frs_{\ve{p}}(\psi)=\frs$,
where 
\[
\frs_{\ps}\colon \pi_2(\xs,\ys,\zs)\to \Spin^c(X_{\as,\bs,\bs'})\subset W(\bS^1),\label{eq:swmapmoregeneral}
\]
is a map defined by Ozsv\'{a}th and Szab\'{o} \cite{OSTriangles}*{Section~2.2}. Note that $X_{\as,\bs,\bs'}$ is a 4-manifold with three ends, which is naturally embedded in $W(\bS^1)$.

 We note that the map $\frs_{\ps}$ described by Ozsv\'{a}th and Szab\'{o}  has an important dependence on the basepoints $\ps$. We will say a Heegaard triple $(\Sigma, \ve{\alpha},\ve{\beta},\ve{\gamma},\ve{p},\ve{q})$ is a \emph{link Heegaard triple} if each component of $\Sigma\setminus \ve{\tau}$ has exactly one basepoint in $\ve{p}$ and one basepoint in $\qs$, for each $\ve{\tau}\in \{\ve{\alpha},\ve{\beta},\ve{\gamma}\}$. In \cite{ZemAbsoluteGradings}, the author describes an oriented, properly embedded surface with boundary $S_{\as,\bs,\gs}\subset X_{\as,\bs,\gs}$. Analogous to the map $\frs_{\ps}$, one can consider the map $\frs_{\qs}$. The set $\Spin^c(X_{\as,\bs,\gs})$ is an affine space over $H^2(X_{\as,\bs,\gs};\Z)\iso H_2(X_{\as,\bs,\gs},\d X_{\as,\bs,\gs};\Z)$. The two maps $\frs_{\ve{p}}$ and $\frs_{\ve{q}}$ are related by the following formula:

\begin{lem}[\cite{ZemAbsoluteGradings}*{Lemma 3.3}]\label{lem:trianglespincstructures}If $(\Sigma, \ve{\alpha},\ve{\beta},\ve{\gamma},\ve{p},\ve{q})$ is a link Heegaard triple, and $\psi$ is a homotopy class of triangles, then
\[
\frs_{\ve{p}}(\psi)-\frs_{\ve{q}}(\psi)=\PD[S_{\as,\bs,\gs}].
\]
\end{lem}

Another important component of the construction of the maps in \cite{ZemCFLTQFT} are maps for adding or removing basepoints on a link component. These are  the \emph{quasi-stabilization maps.} The quasi-stabilization operation was first considered in \cite{MOIntSurg}.  The author considered them further in \cite{ZemQuasi}, proving several useful results. According to \cite{ZemQuasi}*{Theorem~A}, if $\bL=(L,\ve{p},\ve{q})$ is a multi-based link in $Y$ and $(p,q)$ is a pair of new basepoints which are contained in a single component of $L\setminus (\ps\cup \qs)$, and $p$ follows $q$ with respect to the link's orientation (note $p$ is of type-$\ws$ and $q$ is of type-$\zs$), then there are chain maps
\begin{align*}
S_{p,q}^{+},\,\, T_{p,q}^+&\colon \cCFL^\circ(Y,L,\ps,\qs,\frs)\to \cCFL^\circ(Y,L,\ps\cup \{p\},\qs\cup \{q\},\frs)\\
S_{p,q}^{-}, \,\, T_{p,q}^-&\colon \cCFL^\circ(Y,L,\ps\cup \{p\},\qs\cup \{q\},\frs)\to \cCFL^\circ(Y,L,\ps,\qs,\frs),
\end{align*}
which  commute with the change of diagrams maps. The quasi-stabilization maps are the link cobordism maps for the decorated link cobordisms shown in Figure \ref{fig::18}.

 If $q$ comes after $p$ (where $p$ is of type-$\ws$ and $q$ is of type-$\zs$), there are also similarly defined maps, for which we write instead $S_{q,p}^+$, $S_{q,p}^-$, $T_{q,p}^{+}$, and $T_{q,p}^{-}$. The maps $S_{q,p}^{\circ}$ and $T_{q,p}^\circ$ have a similar interpretation in terms of decorated link cobordisms (see \cite{ZemCFLTQFT}*{Section~4.4}).

The final ingredient of the link cobordism maps are the band surgery maps \cite{ZemCFLTQFT}*{Section~6}. If $\bL$ is an oriented link, and $B$ is an oriented band which is appropriately positioned with respect to the basepoints, then we can perform band surgery to obtain a new link $\bL(B)$ inside of $Y$. There are two band maps defined in \cite{ZemCFLTQFT} for band surgery, corresponding to adding a band which forms a region of type-$\ve{w}$ or of type-$\ve{z}$. They are denoted $F_{B}^{\ve{w}}$ and $F_{B}^{\ve{z}}$ and they determine maps
\[
F_{B}^{\ve{w}},\,\, F_B^{\ve{z}}\colon \cCFL^\circ(Y,\bL,\frs)\to \cCFL^\circ(Y,\bL(B),\frs).
\]
 Both the type-$\ws$ band maps and the type-$\zs$ band maps are defined by counting holomorphic triangles in a Heegaard triple which respects the band. We will describe the band maps in more detail in the proof of Theorem~\ref{thm:C}.

\subsection{Maslov and Alexander gradings and link cobordisms}
Ozsv\'{a}th and Szab\'{o} defined two gradings on knot and link Floer homology \cite{OSKnots} \cite{OSLinks}, which they call the Maslov and Alexander gradings. It's somewhat more convenient for our purposes to describe these two gradings in terms of three gradings
\[
\gr_{\ve{w}}, \quad \gr_{\ve{z}} \quad \text{and} \quad A,
\]
which satisfy a linear dependency.

The gradings $\gr_{\ve{w}}$ and $\gr_{\ve{z}}$ are called the Maslov gradings, and $A$ is the Alexander grading. If $\ve{x}$ and $\ve{y}$ are intersection points on a Heegaard diagram $(\Sigma,\as,\bs,\ps,\qs)$, and $\phi\in \pi_2(\xs,\ys)$ is an arbitrary class, then relative versions of the gradings can be defined by the formulas
\begin{align*}\gr_{\ve{w}}(\ve{x},\ve{y})&=\mu(\phi)-2n_{\ve{p}}(\phi)\\
\gr_{\ve{z}}(\ve{x},\ve{y})&=\mu(\phi)-2n_{\ve{q}}(\phi)\\
A(\ve{x},\ve{y})&=n_{\ve{q}}(\phi)-n_{\ve{p}}(\phi).
\end{align*}
If $\frs$ is torsion, then the formula for $\gr_{\ve{w}}$ is independent of the disk $\phi$, and if $\frs-\PD[L]$ is torsion, the formula for $\gr_{\ve{z}}$ is independent of $\phi$. If $[L]=0\in H_1(Y;\Z)$, then $A$ is independent of the disk $\phi$.

  If $\frs$ is torsion, then an absolute lift of $\gr_{\ve{w}}$ can be specified by using the absolute grading on $\CF^\infty(Y,\ws,\frs)$ from \cite{OSTriangles}. Similarly if $\frs-\PD[L]$ is torsion, then an absolute lift of $\gr_{\zs}$ can be specified by using the absolute grading on $\CF^\infty(Y,\zs,\frs-\PD[L])$.

 Assume $U$ is assigned to the type-$\ws$ basepoints of $\bL$ and $V$ is assigned to the type-$\zs$ basepoints (note that this will not always be the case in this paper). With respect to  $\gr_{\ve{w}}$, the variable $U$ is then $-2$ graded, while $V$ is 0 graded. With respect to $\gr_{\ve{z}}$, the variable $U$ is 0 graded, while $V$ is $-2$ graded. With respect to the Alexander grading, the variable $U$ is $-1$ graded, and the variable $V$ is $+1$ graded.

 When $\frs$ is torsion and $L$ is null-homologous, all three gradings are defined, and are related via the formula
\[
A(\ve{x})=\frac{1}{2}(\gr_{\ve{w}}(\ve{x})-\gr_{\ve{z}}(\ve{x})).
\]

In \cite{ZemAbsoluteGradings}*{Theorems~1.5, 1.6}, the author shows that if $(W,\cF)\colon (Y_1,K_1)\to (Y_2,K_2)$ is a decorated link cobordism and $\ve{x}$ is a homogeneously graded element then
\begin{equation}A(F_{W,\cF,\frs}(\ve{x}))-A(\ve{x})=\frac{\langle c_1(\frs),[\hat{ S}]\rangle-[\hat{ S}]\cdot [\hat{ S}]}{2}+\frac{\chi( S_{\ve{w}})-\chi( S_{\ve{z}})}{2},\label{eq:grading1}\end{equation}
\begin{equation}\gr_{\ve{w}}(F_{W,\cF,\frs}(\ve{x}))-\gr_{\ve{w}}(\ve{x})=\frac{c_1(\frs)^2-2\chi(W)-3\sigma(W)}{4}+\tilde{\chi}( S_{\ve{w}})\label{eq:grading2}\end{equation} and
\begin{equation}\gr_{\ve{z}}(F_{W,\cF,\frs}(\ve{x}))-\gr_{\ve{z}}(\ve{x})=\frac{c_1(\frs-\PD[ S])^2-2\chi(W)-3\sigma(W)}{4}+\tilde{\chi}( S_{\ve{z}}).\label{eq:grading3}\end{equation} Here $
\hat{ S}$ is the surface obtained by capping $ S$ off with Seifert surfaces,  and $\tilde{\chi}( S_{\ve{w}})$ denotes the reduced Euler characteristic:
\[
\tilde{\chi}( S_{\ve{w}})=\chi( S_{\ve{w}})-\tfrac{1}{2}(|\ve{w}_1|+|\ve{w}_2|).
\]

The conjugation map $\iota_K$ is graded, in the following sense:

\begin{lem}The map $\iota_K$ satisfies \[
\gr_{\ve{w}}(\iota_K(\ve{x}))=\gr_{\ve{z}}(\ve{x}),\qquad  \gr_{\ve{z}}(\iota_K(\ve{x}))=\gr_{\ve{w}}(\ve{x}) \qquad \text{and} \qquad A(\iota_K(\ve{x}))=-A(\ve{x}),
\]
 for homogeneously graded $\ve{x}$.
\end{lem}
\begin{proof}As $A(\ve{x})=\tfrac{1}{2}(\gr_{\ve{w}}(\ve{x})-\gr_{\ve{z}}(\ve{x}))$, the third relation follows from the first two, so let us consider only the relations involving $\gr_{\ve{w}}$ and $\gr_{\ve{z}}$. The map $\iota_K$ is defined as the composition $\tau_K\circ \eta$. Since $\tau_K$ is a diffeomorphism map, it is grading preserving, so it remains only to consider the map $\eta$. Since $\eta$ reverses which basepoints are identified as type-$\ws$ or type-$\zs$, it is straightforward to see from the definition that it switches $\gr_{\ws}$ and $\gr_{\zs}$.
\end{proof}

\section{Further properties of the link Floer TQFT}
\label{sec:furtherproperties}
In this section we describe several properties of the link cobordism maps from \cite{ZemCFLTQFT}. Firstly, we prove Theorem \ref{thm:C},  conjugation invariance for the link Floer TQFT. Then we prove several important relations for the link cobordism maps, such as the bypass relation, as well as an interpretation of the maps $\Phi_p$ and $\Psi_q$ in terms of dividing sets on cylindrical link cobordisms.

\subsection{Conjugation invariance of the link cobordism maps}

We now prove Theorem~\ref{thm:C}:

\begin{customthm}{\ref{thm:C}}Suppose that $(W,\cF)\colon (Y_1,\bL_1)\to (Y_2,\bL_2)$ is a decorated link cobordism. Let 
\[
(W,\bar{\cF})\colon (Y_1,\bar{\bL}_1)\to (Y_2,\bar{\bL}_2)
\]
 denote the conjugate decorated link cobordism (obtained by reversing the designation of regions as type-$\ws$ or type-$\zs$). The following diagram commutes up to filtered, $\cR$-equivariant chain homotopy:
	\[
	\begin{tikzcd}\cCFL^\circ(Y_1,\bL_1,\frs_1)\arrow{d}{F_{W,\cF,\frs}}\arrow{r}{\eta}& \cCFL^\circ(Y_1,\bar{\bL}_1,\bar{\frs}_1+\PD[L_1])\arrow{d}{F_{W,\bar{\cF},\bar{\frs}+\PD[ S]}}\\
	\cCFL^\circ(Y_2,\bL_2,\frs_2) \arrow{r}{\eta} & \cCFL^\circ(Y_2,\bar{\bL}_2,\bar{\frs}_2+\PD[L_2]).
	\end{tikzcd}
	\]
\end{customthm}

\begin{proof}Using the construction from \cite{ZemCFLTQFT}, the theorem is nearly (but not quite) a tautology. Let us first consider the case that each component of the underlying surface of the link cobordism, $S$, intersects both $Y_1$ and $Y_2$ non-trivially.

 The construction from \cite{ZemCFLTQFT} involves taking a decomposition of the link cobordism $(W,\cF)$, into elementary link cobordisms $(W_i,\cF_i)\colon (Y_i,L_i)\to (Y_i,L_{i+1})$. Each $(W_i,\cF_i)$ satisfies one of the following:
\begin{enumerate}
\item $W_i$ has a Morse function $f$ such that $f$ and $f|_{\cF_i}$ have no critical points, and the divides on $\cF_i$ go from the incoming end to the outgoing end.

\item $W_i$ has a Morse function $f$ with a single critical point, which is of index 1 or 3, or an arbitrary number of index 2 critical points. Furthermore $f|_{\cF_i}$ has no critical points and the divides on $\cF_i$ all go from the incoming end to the outgoing end.

\item $W_i$ has a Morse function function $f$ such that $f$ and $f|_{\cF_i}$ have no critical points. Furthermore, all divides on $\cF_i$ go from the incoming boundary to the outgoing boundary, except for one arc, which has both endpoints on the incoming link, or the outgoing link.

\item $W_i$ has a Morse function $f$ with no critical points, such that $f|_{\cF_i}$ has a single critical point, which is index 1. The divides on $\cF_i$ all go from the incoming to the outgoing boundary.
\end{enumerate}

At most one elementary cobordism in the decomposition is allowed to have a Morse function $f$ with index 2 critical points.

The maps induced by link cobordisms of type (1) are the maps induced by diffeomorphisms, so for such link cobordisms the claim follows from naturality of the map $\eta$ from Lemma \ref{lem:etacommuteswithchangeofdiagram}.
	
	The maps for cobordisms of type (2) are the 4-dimensional handle attachment maps. Consider first the 1-handle and 3-handle maps. Suppose $\bS^0\colon S^0\times D^3\to Y\setminus L$ is a framed 0-sphere, and $\cH=(\Sigma,\as,\bs,\ps,\qs)$ is a diagram for $(Y,\bL)$, such that $\Sigma$ intersects the image of $\bS^0$ along two disks, which are disjoint from $\as\cup \bs\cup \ps\cup \qs$. We can obtain a diagram for $(Y(\bS^0), \bL)$ by removing the two disks in $\Sigma$ corresponding to the image of $\bS^0$, and connected the resulting boundary components with an annulus. We add two new curves, $\alpha_0$ and $\beta_0$, which are homologically essential in the annulus and intersect in a pair of points, $\theta^+$ and $\theta^-$. The points $\theta^+$ and $\theta^-$ are distinguished by the Maslov gradings (the designation into higher and lower generators is the same for $\gr_{\ve{w}}$ and $\gr_{\ve{z}}$). The 1-handle map is $\xs\mapsto \xs\times \theta^+$, extended $\cR$-equivariantly, and the 3-handle map satisfies $\xs\times \theta^+\mapsto 0$ and $\xs\times \theta^-\mapsto \xs$. It is easy to see that $\eta$ preserves this decomposition into higher and lower generators. Hence $\eta$ commutes with the 1-handle and 3-handle maps.
	
	The interaction of $\eta$ with the 2-handle maps is slightly more subtle. Suppose $(W,\cF)\colon (Y_1,\bL_1)\to (Y_2,\bL_2)$ is a 2-handle cobordism for a framed 1-dimensional link $\bS^1$ in $Y_1\setminus L_1$.  The map $F_{W,\cF,\frs}$ is defined by counting holomorphic triangles in a Heegaard triple $\cT=(\Sigma, \ve{\alpha}',\ve{\alpha},\ve{\beta},\ve{p},\ve{q})$ which is subordinate to an $\alpha$-bouquet for $\bS^1$ (we recall that an \emph{$\alpha$-bouquet} for $\bS^1\subset Y\setminus \bL$ is an embedded collection of arcs in $Y$, one for each component of $\bS^1$, each with one end on $\bS^1$, and one end on a component of $L\setminus (\ve{p}\cup \ve{z})$ which is oriented to be going from a type-$\ve{w}$ basepoint to a type-$\ve{z}$ basepoint; compare \cite{OSTriangles}*{Definition~4.1}). For such a Heegaard triple, the cobordism map is defined by the formula
	\[
	F_{W,\cF,\frs}(\ve{x}):=F_{\as',\as,\bs,\frs}(\Theta_{\as',\as}^+,\ve{x}),
	\]
	 where $F_{\as',\as,\bs,\frs}$ is the map which counts holomorphic triangles in $\Sigma\times \Delta$ (for an auxiliary choice of almost complex structure $J$) representing Maslov index zero homology classes $\psi$ with $\frs_{\ve{p}}(\psi)=\frs$ on the triple $(\Sigma,\ve{\alpha}',\ve{\alpha},\ve{\beta})$. Here $(\Sigma,\ve{\alpha}',\ve{\alpha},\ps,\qs)$ is a diagram for an unlink in $(S^1\times S^2)^{\# k}$ where each component has exactly two basepoints, and $\Theta_{\as',\as}^+$ is the element of $\cHFL^-(\Sigma,\ve{\alpha}',\ve{\alpha},\ps,\qs)$ of top $\gr_{\ve{w}}$- and $\gr_{\ve{z}}$-grading.
	
	There is a bijection between holomorphic triangles on the triple $\mathcal{T}=(\Sigma,\ve{\alpha}',\ve{\alpha},\ve{\beta})$ with the almost complex structure $ J$, and holomorphic triangles on the conjugate triple $\bar{\cT}=(-\Sigma,\bar{\ve{\beta}},\bar{\ve{\alpha}},\bar{\ve{\alpha}}')$ with complex structure $-J$. Note also that $\bar{\cT}$ is now subordinate to a $\beta$-bouquet for the framed link $\bS^1$, whereas $\cT$ was subordinate to an $\alpha$-bouquet. We note that, as in \cite{OSTriangles}*{Lemma 5.2}, we can use either Heegaard triples which are subordinate to an $\alpha$-bouquet or a $\beta$-bouquet to compute the cobordism maps for surgery on a framed 1-dimensional link (see \cite{ZemCFLTQFT}*{Lemma~5.9}).
	
	However, there is a distinction in how $\Spin^c$ structures are assigned to homology classes of triangles on the triples $\cT$ and $\bar{\cT}$, since the designation of $\ps$ and $\qs$ as type-$\ws$ or type-$\zs$ is switched. Note that the 4-manifolds $X_{\as',\as,\bs}$ and $X_{\bar{\bs},\bar{\as},\bar{\as}'}$, are canonically diffeomorphic, and the diffeomorphism restricts to a diffeomorphism between the embedded surfaces $S_{\as',\as,\bs}$ and $ S_{\bar{\bs},\bar{\as},\bar{\as}'}$. Write 
	\[
	\iota\colon X_{\as',\as,\bs}\to W(Y_1,\bS^1)
	\]
	 for the embedding, which is well defined up to isotopy (see \cite{OSTriangles}*{Proposition~4.3} or \cite{JCob}*{Proposition~6.6}). Since $W(Y_1,\bS^1)$ is formed by attaching 3-handles and 4-handles to $X_{\as',\as,\bs}$, it follows that $\iota$ induces an isomorphism on $H^2(-;\Z)$ and on the set of $\Spin^c$ structures. If $\frt\in \Spin^c(W)$, the map $F_{W, \bar{\cF},\frt}$ can be computed using the triple $\bar{\cT}$ by counting holomorphic triangles representing homology classes $\bar{\psi}$ with $\frs_{\qs}(\bar{\psi})=\frt$. From Lemma~\ref{lem:trianglespincstructures}, we have
	\[
	\frs_{\qs}(\bar{\psi})=\frs_{\ps}(\bar{\psi})+\PD[ S_{\bar{\bs},\bar{\as},\bar{\as}'}].
	\]
	Since $\iota^*\PD[ S]=\PD[ S_{\bar{\bs},\bar{\as},\bar{\as}'}]$, we conclude that
	\[
	\eta(F_{\as',\as,\bs,\frs}(\Theta_{\as',\as}^+,\ve{x}))=F_{\bar{\bs},\bar{\as},\bar{\as}',\bar{\frs}+\iota^*\PD[ S]}(\eta(\ve{x}),\eta(\Theta_{\as',\as}^+)).
	\]
	 A straightforward computation shows that $\eta(\Theta_{\as',\as}^+)=\Theta^+_{\bar{\as},\bar{\as}'}$. Hence, for 2-handle maps we conclude that
	\[
	\eta\circ F_{W,\cF,\frs}\simeq F_{W,\bar{\cF},\bar{\frs}+\PD[ S]}\circ \eta.
	\]

	We now consider link cobordisms of type (3). The induced cobordism maps is equal to one of the quasi-stabilization maps $S_{p,q}^{+}$, $S_{p,q}^-,$ $T_{p,q}^+$ or $T_{p,q}^-$, depending on the configuration of $ S_{\ve{w}}$ and $ S_{\ve{z}}$ regions with respect to the dividing set. The configurations of the regions, as well as the maps induced, are shown in Figure \ref{fig::18}.

	\begin{figure}[ht!]
		\centering
		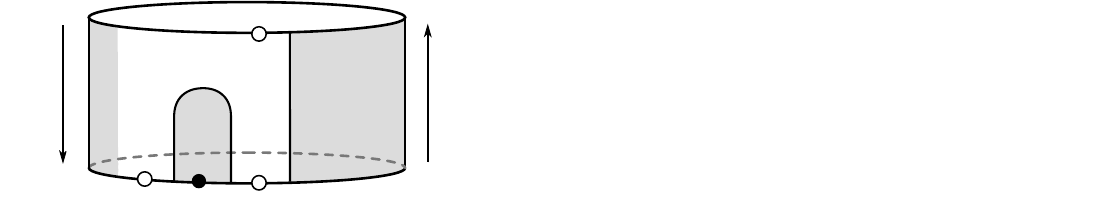
		\caption{\textbf{Decorated link cobordisms for the quasi-stabilization maps.} The underlying undecorated link cobordism is $[0,1]\times  L\subset [0,1]\times  Y$.\label{fig::18}}
	\end{figure}

Since the type of each region of $ S\setminus \cA$ is changed (from $\ve{w}$ to $\ve{z}$, or vice versa), it is sufficient to show that 
	\[
	S_{p,q}^{\circ}\circ \eta\simeq \eta\circ  T_{p,q}^{\circ}\qquad \text{and} \qquad T_{p,q}^{\circ}\circ \eta\simeq \eta\circ  S_{p,q}^{\circ},
	\]
	 for $\circ\in \{+,-\}$. This is essentially a tautology from the definitions. The quasi-stabilization maps are defined by inserting a new $\as$ and $\bs$ curve in the diagram, as shown in Figure~\ref{fig::14}. Let us write $\cH^+$ for the diagram obtained by quasi-stabilizing. There are two new intersection points. We will write $\theta^{\ws}$ for the intersection point with higher $\gr_{\ws}$-grading, and $\xi^{\ws}$ for the intersection point with lower $\gr_{\ws}$-grading. Somewhat redundantly, we will write $\theta^{\zs}$ for the intersection point with higher $\gr_{\zs}$-grading, and $\xi^{\zs}$ for the intersection point with lower $\gr_{\zs}$-grading. Of course, we have $\theta^{\ws}=\xi^{\zs}$ and $\xi^{\ws}=\theta^{\zs}$. 
	 
As $\cR$-modules we have an isomorphism
	 \[
	 \cCFL^-(\cH^+,\frs)\iso \cCFL^-(\cH,\frs)\otimes_{\bF_2} \langle \theta^{\ws}, \theta^{\zs} \rangle
	 \]
	 where $\langle\theta^{\ws},\theta^{\zs}\rangle$ is the 2-dimensional vector space generated by the elements $\theta^{\ws}$ and $\theta^{\zs}$.

	  The maps $S_{p,q}^{+}$ and $S_{p,q}^-$ are defined by the formulas
\begin{equation}
S_{p,q}^+(\xs)= \xs\times \theta^{\ws}, \qquad S_{p,q}^-(\xs\times \theta^{\ws})=0\qquad \text{and} \qquad S_{p,q}^-(\xs\times \xi^{\ws})=\xs,\label{eq:quasidef1'} 
	\end{equation}
extended $\cR$-equivariantly. Analogously the maps $T_{p,q}^+$ and $T_{p,q}^-$ are defined by the formulas \begin{equation}
T_{p,q}^+(\ve{x})=\ve{x}\times \theta^{\zs},\qquad T_{p,q}^-(\ve{x}\times \theta^{\ve{z}})=0\qquad \text{and} \qquad T_{p,q}^-(\ve{x}\times \xi^{\ve{z}})=\ve{x}.\label{eq:quasidef2'}
\end{equation}
Since $\eta$ switches the $\gr_{\ws}$- and $\gr_{\zs}$-gradings, it sends $\theta^{\ws}$ to $\theta^{\zs}$ and $\xi^{\ws}$ to $\xi^{\zs}$.  Hence,  we have $T_{p,q}^+\circ\eta\simeq \eta\circ S_{p,q}^+$ and $T_{p,q}^-\circ \eta\simeq \eta\circ S_{p,q}^-$.
	
	\begin{figure}[ht!]
		\centering
		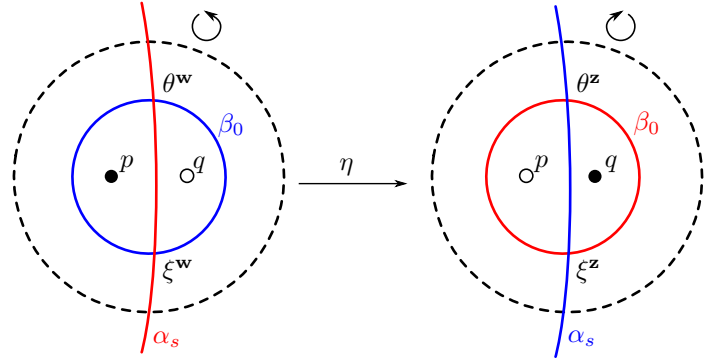
		\caption{\textbf{The effect of conjugation on a quasi-stabilized diagram.} The map $\eta$ sends $\xs\times \theta^{\ws}$ to $\xs\times \theta^{\zs}$. The outer dashed circle represents where the connected sum with the rest of the diagram takes place. The circular arrows indicate the orientation of the Heegaard surface.\label{fig::14}}
	\end{figure}
	
	We now consider the final type of elementary link cobordism, type (4), corresponding to a saddle cobordism, which induces a band map.  As defined in \cite{ZemCFLTQFT}*{Section~6}, there are two band maps, $F_{B}^{\ve{w}}$ and $F_{B}^{\ve{z}}$, both defined by counting holomorphic triangles. Supposing that $B$ is an $\alpha$-band (i.e. the ends of $B$ are in segments of $\bL$ going from type-$\ve{w}$ basepoints to type-$\ve{z}$ basepoints) then we can pick a Heegaard triple $\cT=(\Sigma, \ve{\alpha}',\ve{\alpha},\ve{\beta},\ps,\qs)$ representing surgery of the link on the band. The band maps $F_{B}^{\ve{z}}$ and $F_{B}^{\ve{w}}$ are defined by the formulas
	\[
	F_{B}^{\ve{z}}(\ve{x}):= F_{\as',\as,\bs,\frs}(\Theta_{\as',\as}^{\ve{w}},\ve{x})\qquad \text{and} \qquad F_{B}^{\ve{w}}(\ve{x}):= F_{\as',\as,\bs,\frs}(\Theta_{\as',\as}^{\ve{z}},\ve{x})
	\] 
	for $\frs\in \Spin^c(Y)$. The diagram $(\Sigma,\as',\as,\ps,\qs)$ represents an unlink in $(S^1\times S^2)^{\# g(\Sigma)}$, such that all components have exactly two basepoints, except for one basepoint, which has four. Here $\Theta^{\ve{w}}_{\as',\as}$ is a distinguished element in the top $\gr_{\ws}$-graded subset of $\cHFL^-(\Sigma, \ve{\alpha}',\ve{\alpha},\ps,\qs)$, and $\Theta_{\as',\as}^{\ve{z}}$ is a distinguished element in the top $\gr_{\zs}$-graded subset (see \cite{ZemCFLTQFT}*{Lemma~3.7}). The map $F_B^{\ve{w}}$ corresponds to a saddle cobordism where the band $B$ we attach is a subset of $S_{\ve{w}}$. The map $F_{B}^{\ve{z}}$ corresponds to the case where the band $B$ is a subset of $ S_{\ve{z}}$. Analogous to the quasi-stabilization maps, it is easy an easy computation to see that $\eta$ maps $\Theta^{\ve{w}}_{\as',\as}$ to $\Theta^{\ve{z}}_{\bar{\as},\bar{\as}'}$, and $\eta$ maps $\Theta^{\zs}_{\as',\as}$ to $\Theta^{\ws}_{\bar{\as},\bar{\as}'}$. Paying attention to the change of $\Spin^c$ structures, as we did with the 2-handle maps, we arrive at the relation
	\[
	\eta( F_{\as',\as,\bs,\frs}(\Theta^{\ve{w}}_{\as',\as},\ve{x}))= F_{\bar{\bs},\bar{\as},\bar{\as}',\bar{\frs}+\PD[ S_{\bar{\bs},\bar{\as},\bar{\as}'}]}(\eta(\ve{x}),\Theta^{\ve{z}}_{\bar{\as},\bar{\as}'}).
	\]
	 However we note that the 4-manifold $X_{\as',\as,\bs}$ is diffeomorphic to $[0,1]\times  Y$ once we fill in $Y_{\as',\as}$ with 3-handles and 4-handles. Hence
	\begin{equation}
	\Spin^c(X_{\as',\as,\bs})\iso \Spin^c(Y).\label{eq:isomorphismofspinc}
	\end{equation}
	 It is not hard to see that the cohomology class $\PD[ S_{\bar{\bs},\bar{\as},\bar{\as}'}]$ is equal to $\iota^*\PD[ [0,1]\times  L]$, where $\iota$ is the inclusion $\iota\colon X_{\as',\as,\bs}\hookrightarrow [0,1]\times  Y$. Hence, under the  isomorphism of $\Spin^c$ structures from Equation~\eqref{eq:isomorphismofspinc}, we have that $\PD[ S_{\bar{\bs},\bar{\as},\bar{\as}'}]$ acts by $\PD[L]$. This is, however, just the correction in $\Spin^c$ structures due to $\eta$ on the link Floer complexes, from Lemma  \ref{lem:changeSpincstructure}. Hence 	
	\[
	\eta\circ F_B^{\ve{w}}\simeq  F_B^{\ve{z}}\circ \eta \qquad \text{and} \qquad F_B^{\ve{w}}\circ \eta\simeq \eta\circ F_B^{\ve{z}}.
	\]
	
	Notice too that the $\alpha$-band has turned into a $\beta$-band after conjugating, analogously to how the 2-handle maps change after conjugation in \cite{OSTriangles}*{Lemma 5.2}. However the $\alpha$-band maps and the $\beta$-band maps are related by a basepoint moving map, according to \cite{ZemCFLTQFT}*{Proposition~9.10}, and hence either can be used to compute the link cobordism maps.
	
	Combining these observations with the composition law from Equation~\eqref{eq:compositionlaw}, we conclude that
	\[
	\eta\circ F_{W,\cF,\frs}\simeq F_{W,\bar{\cF},\bar{\frs}+\PD[ S]}\circ \eta,
	\] 
	for a link cobordism where each component of $ S$ intersects both $Y_1$ and $Y_2$ non-trivially.
	
	For a link cobordism where some components of $ S$ do not intersect one or both of $Y_1$ or $Y_2$, the link cobordism maps are defined by puncturing the link cobordism at points along the dividing set, adding extra copies of $(S^3,U,p,q)$ to $(Y_1,\bL_1)$ and $(Y_2,\bL_2)$ using the 0-handle and 4-handle maps defined in \cite{ZemCFLTQFT}*{Section~5.2}. It is easy to check that the conjugation map $\eta$ commutes with the 0-handle and 4-handle maps, and hence statement  follows for general link cobordisms.
	\end{proof}

\subsection{The maps $\Phi$ and $\Psi$}
\label{sec:PhiPsi}

In this section we describe the maps $\Phi$ and $\Psi$ which feature in Theorem \ref{thm:B}, and prove that they are induced by two link cobordisms with relatively simple dividing sets.

In general, if $\bL=(L,\ve{p},\ve{q})$ is a multi-based link in $Y$, we can define endomorphisms $\Phi_p$ and $\Psi_q$ of $\cCFL^\infty(Y,\bL,\frs)$ for each $p\in \ve{p}$ and each $q\in \ve{q}$. They are defined by counting holomorphic disks on a diagram using the formulas
\[
\Phi_p(\ve{x})=U^{-1}\sum_{\ve{y}\in \bT_{\as}\cap \bT_{\bs}}\sum_{\substack{\phi\in \pi_2(\ve{x},\ve{y})\\ \mu(\phi)=1}} n_p(\phi) \# \hat{\cM}(\phi)U^{n_{\ps}(\phi)} V^{n_{\qs}(\phi)}\cdot \ve{y},
\]
and
\[
\Psi_q(\ve{x})=V^{-1}\sum_{\ve{y}\in \bT_{\as}\cap \bT_{\bs}}\sum_{\substack{\phi\in \pi_2(\ve{x},\ve{y})\\ \mu(\phi)=1}} n_q(\phi) \# \hat{\cM}(\phi)U^{n_{\ps}(\phi)} V^{n_{\qs}(\phi)}\cdot \ve{y}.
\]

When $\bL=(K,p,q)$ is a doubly based knot, as in Theorem \ref{thm:B}, then we will often write $\Phi$ for $\Phi_p$ and $\Psi$ for $\Psi_q$. In the case of a doubly based knot, these are the same maps considered in Section \ref{sec:algebraicpreliminaries}, in the context of $\iota_K$-complexes, for the basis $B$ consisting of the set of intersection points $\ve{x}\in \bT_{\as}\cap \bT_{\bs}$ with $\frs_{\ve{p}}(\ve{x})=\frs$ on a given Heegaard diagram.

We now show that $\Phi_p$ and $\Psi_q$ have a simple interpretation in terms of decorated link cobordisms:

\begin{lem}\label{lem:cobordismsforPhiPsi}The link cobordism maps for $([0,1]\times  Y,[0,1]\times  L)$ with the dividing sets shown in Figure \ref{fig::17} are filtered, chain homotopy equivalent to the maps $\Phi_p$ and $\Psi_q$. This holds without any assumption on the number of basepoints on the link component, though if there are more than two basepoints, the remaining arcs of the dividing set (not shown in Figure \ref{fig::17}) are vertical arcs from $\{0\}\times L$ to $\{1\}\times L$.
\end{lem}

\begin{proof}Let us focus on the map $\Phi_p$, since the map $\Psi_q$ can be handled in an analogous fashion. Let $(W,\cF)$ denote the link cobordism on the left side of Figure \ref{fig::17}. If there are more than two basepoints on the link component, then the link cobordism map is a composition of two quasi-stabilization maps, and we can use the computation of the quasi-stabilized differential from \cite{ZemQuasi} to compute $\Phi_p$ directly. To make this explicit, suppose that $q$ is the basepoint immediately preceding $p$, and  suppose $\cH$ is a Heegaard diagram for the link with $p$ and and $q$ removed. Let $\cH^+$ be a diagram obtained from $\cH$ by quasi-stabilizing the link component at $p$ and $q$, similar to Figure \ref{fig::14}. Write $p'$ and $q'$ for the basepoints adjacent to $p$ and $q$, which are already on $\cH$. We consider the version of the full link Floer complex where we have a variable for each basepoint (this is mentioned in Section~\ref{sec:curvedcomplexesandeta}, and is described in detail in \cite{ZemCFLTQFT}). By \cite{ZemQuasi}*{Proposition 5.3}, the differential on the quasi-stabilized diagram $\cH^+$ can be written as
\begin{equation}\d_{\cH^+}=\begin{pmatrix}\d_{\cH} & U_p +U_{p'} \\
V_{q}+V_{q'}& \d_{\cH}
\end{pmatrix},\label{eq:quasistabdiff}\end{equation} under the identification (of modules) 
\[
\cCFL^-(\cH^+)=\cCFL^-(\cH)\otimes_{\bF_2} \bF_2[U_p,V_q]\otimes_{\bF_2} \langle \theta^{\ve{w}},\xi^{\ve{w}} \rangle.
\]
 The above matrix notation is with respect to writing the coefficient of $\theta^{\ws}$ in the first row and column, and the coefficient of $\xi^{\ws}$ in the second. Using this, we see that only the upper right component has any terms involving nonzero powers of $U_p$, and hence
 \begin{equation}
S_{p,q}^+S_{p,q}^-\simeq \Phi_p\label{eq:S^+S^-=Phi}
 \end{equation}
  using the formulas for $S_{p,q}^+$ and $S_{p,q}^-$ from Equation \eqref{eq:quasidef1'}. Since $F_{W,\cF,\frs}\simeq S_{p,q}^+S_{p,q}^-$ by construction, we conclude that $F_{W,\cF,\frs}\simeq \Phi_p$, concluding the proof in the case $p$ and $q$ aren't the only basepoints on their link component.

In the case where there are exactly two basepoints on the link component, our strategy is to move the dividing set around to reduce to the previous case. Our goal will be to move the dividing set around so that $\{t\}\times L $ intersects the dividing set non-trivially for each $t\in [0,1]$. Let $\phi$ be diffeomorphism of $(Y,L)$ which is the identity outside a small neighborhood of $p$, and which moves $p$ to a nearby point $p'\not \in \ps\cup \qs$. Let $q'$ be a new basepoint between $p$ and $p'$, and assume that the ordering of these points is $(p',q,p)$, ordered right to left.  We can move the dividing set around, as in Figure \ref{fig::20}.

\begin{figure}[ht!]
\centering
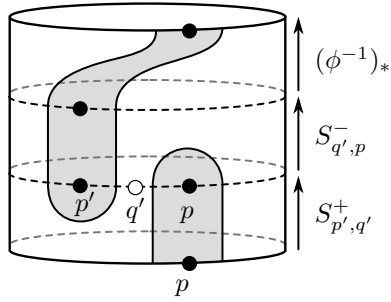
\caption{\textbf{Computing the link cobordism map for $\Phi_p$} by manipulating the dividing set for the link cobordism from Figure \ref{fig::17}. The induced link cobordism map in Lemma \ref{lem:cobordismsforPhiPsi} is denoted $F_{W,\cF,\frs}$.\label{fig::20}}
\end{figure}

   By construction, using the decomposition shown in Figure~\ref{fig::20}, we have
\begin{equation}
F_{W,\cF,\frs}\simeq (\phi^{-1})_*S^{-}_{q',p} S^+_{p',q'}.\label{eq:expressionforPhicobordism}
\end{equation}
  According to \cite{ZemCFLTQFT}*{Lemma~4.26} we have 
  \begin{equation}
  (\phi^{-1})_*\simeq T_{p',q'}^- S_{q',p}^+.
  \label{eq:basepointmovingmapformula}
\end{equation}  
  
   Combining Equations~\eqref{eq:expressionforPhicobordism} and \eqref{eq:basepointmovingmapformula}, we compute
\[
F_{W,\cF,\frs}\simeq (\phi^{-1})_*S^{-}_{q',p} S^+_{p',q'}\simeq  T_{p',q'}^- S_{q',p}^+S^{-}_{q',p} S^+_{p',q'}\simeq T_{p',q'}^- \Phi_p S^+_{p',q'}\simeq T_{p',q'}^-  S^+_{p',q'}\Phi_p\simeq \Phi_p.
\] 
The third chain homotopy is justified by Equation~\eqref{eq:S^+S^-=Phi}. The fourth chain homotopy follows from \cite{ZemCFLTQFT}*{Lemma~4.16}. The fifth chain homotopy follows from the relation $T_{p',q'}^-S_{p',q'}^+\simeq \id$, which follows from the formulas defining the quasi-stabilization maps in Equations~\eqref{eq:quasidef1'} and \eqref{eq:quasidef2'}. This completes the proof for $\Phi_p$.

The same strategy works for the map $\Psi_q$ and the corresponding link cobordism.
\end{proof}

\subsection{The bypass relation}\label{sec:bypassrelation}

In this section we prove Lemma \ref{thm:D}, the bypass relation. 

\begin{customlem}{\ref{thm:D}}If $(W,\cF_1),$ $(W,\cF_2)$ and $(W,\cF_3)$ are decorated link cobordisms which fit into a bypass triple, i.e. the underlying undecorated link cobordisms are equal, and the decorations differ only inside a disk on the surface, as shown in Figure \ref{fig::33}, then
\[
F_{W,\cF_1,\frs}+F_{W,\cF_2,\frs}+F_{W,\cF_3,\frs}\simeq 0.
\]
\end{customlem}
\begin{proof}We will interpret the relation
\[
T_{p,q}^{+} S_{p,q}^-+S_{p,q}^+T_{p,q}^-+\id\simeq 0,
\]
 in terms of surfaces with divides. The above relation is proven in \cite{ZemCFLTQFT}*{Lemma~4.13}, and is immediate from the formulas for the maps (Equations~\eqref{eq:quasidef1'}~and~\eqref{eq:quasidef2'}). The corresponding dividing sets for these maps form a bypass triple on the undecorated link cobordism $( [0,1]\times Y,[0,1]\times L)$, as shown in Figure~\ref{fig::34}. Combining the above result with the composition law from Equation~\eqref{eq:compositionlaw}, the bypass relation follows for general link cobordisms.
\end{proof}

	\begin{figure}[ht!]
		\centering
		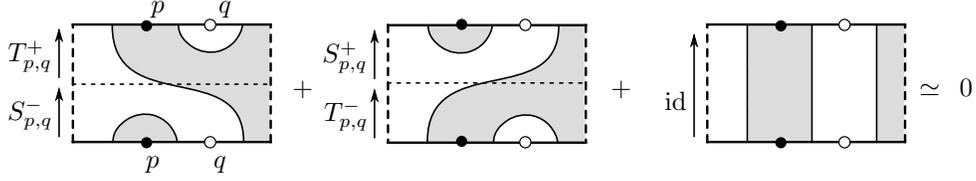
		\caption{\textbf{The bypass triple formed by the dividing sets for the three summands in the relation $T_{p,q}^+S_{p,q}^-+S^+_{p,q}T_{p,q}^-+\id\simeq 0$.}\label{fig::34}}
	\end{figure}

As a corollary, we can give a quick proof of a special case of Sarkar's formula \cite{SarkarMovingBasepoints}*{Theorem~1.1} for the diffeomorphism map induced by twisting a link component in one full twist, when the link component has exactly two basepoints. A similar pictorial argument using more bypasses could presumably be used to prove the formula in full generality, though we will not pursue the argument further.

\begin{cor}Suppose $\bL$ is a multi-based link in $Y$, and $K$ is a component of $\bL$ with exactly two basepoints, $p$ and $q$. If $\rho$ denotes the diffeomorphism resulting from twisting $K$ in one full twist, in the direction of its orientation, then the induced map 
\[
\rho_*\colon \cCFL^\infty(Y,\bL,\frs)\to \cCFL^\infty(Y,\bL,\frs)
\] has the filtered chain homotopy type
\[
\rho_*\simeq \id+\Psi_q\circ\Phi_p.
\]
\end{cor}

\begin{proof}We use the bypass triple shown in Figure \ref{fig::15}, relating three decorated link cobordisms with underlying undecorated link cobordism $([0,1]\times Y,[0,1]\times L)$. We use Lemma \ref{lem:cobordismsforPhiPsi} to interpret the resulting dividing sets as maps, and the formula follows.

\begin{figure}[ht!]
\centering
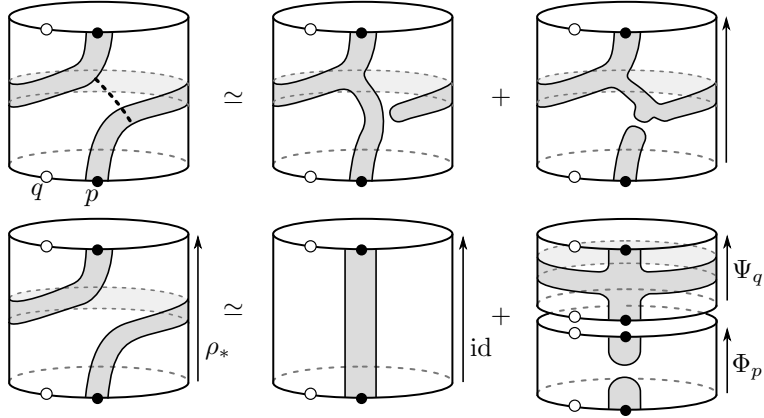
\caption{\textbf{Obtaining the formula $\rho_*\simeq \id+\Psi_q\circ \Phi_p$ from the bypass relation.} The bypass region is a neighborhood of the dashed arc in the top left. On the bottom row the same dividing sets have been manipulated to demonstrate the desired relation. These are dividing sets on the surfaces $[0,1]\times  L$ inside of $[0,1]\times Y$. \label{fig::15}}
\end{figure}
\end{proof}


\section{Link cobordisms and connected sums}\label{sec:connectedsumsandcobordisms}

  If $\bK_1=(K_1,p_1,q_1)$ is a knot in $Y_1$, and $\bK_2=(K_2,p_2,q_2)$ is a knot in $Y_2$, we will define two link cobordism maps
\[
G_1,\,G_2\colon \cCFL^\infty(Y_1\sqcup Y_2, \bK_1\sqcup \bK_2 ,\frs_1\sqcup \frs_2)\to \cCFL^\infty(Y_1\# Y_2,\bK_1\# \bK_2, \frs_1\# \frs_2)
\] 
as well as two maps, $E_1$ and $E_2$, in the opposite direction. Here $\bK_1\# \bK_2=(K_1\# K_2,p,q)$ is the connected sum, with exactly two basepoints. The main goal of this section is to show that $E_i\circ G_i\simeq \id$ and $G_i\circ E_i\simeq \id$ through filtered, equivariant, chain homotopies. This gives an alternate proof of the well known K\"{u}nneth formula for knot Floer homology \cite{OSKnots}*{Theorem~7.1}.

\subsection{The decorated link cobordisms for connected sum maps}
\label{sec:constructcobs}
We now describe the link cobordisms which we will use to induce the chain homotopy equivalences $E_1,G_1,E_2$ and $G_2$. Briefly, the link cobordisms are obtained by adding a 1-handle containing a band, but there is an interesting ambiguity in the dividing set which is important for our connected sum formula, leading to four maps, instead of two. They are depicted in Figure \ref{fig::39}.

More explicitly, the link cobordisms are constructed as follows. Pick points $x_1\in K_1\setminus \{p_1,q_1\}$ and $x_2\in K_2\setminus \{p_2,q_2\}$ where the connected sum will take place. We pick an embedding 
\[
\bS^0\colon S^0\times D^3\to Y_1\sqcup Y_2
\]
 so that the images of  $\{1\}\times D^3$ and $\{-1\}\times D^3$ are coordinate balls in $Y_1$ and $Y_2$, centered at $x_1$ and $x_2$, respectively. Viewing $D^3$ as the closed unit ball in $\R^3$, we  assume further that 
\[
(\bS^0)^{-1}(K_1)=\{1\}\times \{(0,0,t):t\in [-1,1]\}\subset S^0\times  D^3,
\]
  and similarly for $K_2$.  We can form $Y_1\# Y_2$ by removing the interior of $\im (\bS^0)$ and gluing in $[-1,1]\times S^2$. Inside of $Y_1\# Y_2$ we can form the connected sum $K_1\# K_2$ by gluing in $ [-1,1]\times \{(0,0,\pm 1)\}$. We assume further that the map $\bS^0$ is chosen so that $Y_1\#Y_2$ and $K_1\# K_2$ can be oriented compatibly with the orientations of $Y_1$ and $Y_2$.

Write $W$ for the 1-handle cobordism obtained by gluing $ [-1,1]\times D^3$ to $[0,1]\times (Y_1\sqcup Y_2)$ along $\{1\}\times \im (\bS^0)$. Inside of $W$, we  construct the surface
\[
 S= ([0,1]\times (K_1\sqcup K_2))\cup B,
\] 
where  $B$ is the band
\[
B:=  [-1,1]\times \{(0,0,t):t\in [-1,1]\}\times\subset [-1,1]\times D^3.
\] 

We now decorate $ S$ with a dividing set. However the decoration depends on the positions of $x_1$ and $x_2$ with respect to the basepoints. There are two natural configurations  which we will consider:
\begin{enumerate}
\item[$(C1)$] $x_1$ is in a component of $K_1\setminus \{p_1,q_1\}$ going from $p_1$ to $q_1$, while $x_2$ is in a component of $K_2\setminus \{p_2,q_2\}$ going from $q_2$ to $p_2$.
\item[$(C2)$] $x_1$ is in a component of $K_1\setminus \{p_1,q_1\}$  going from $q_1$ to $p_1$, while $x_2$ is in a component of $K_2\setminus \{p_2,q_2\}$ going from $p_2$ to $q_2$.
\end{enumerate}

Given a pair of points $x_1$ and $x_2$ on $K_1\sqcup K_2$,  which together satisfy either $(C1)$ or $(C2)$, we can decorate the surface $ S$ by adding three dividing arcs, as follows. One arc is constructed as the union of $[0,1]\times \{x_1\},$ $[0,1]\times \{x_2\}$ and the core $[-1,1]\times \{(0,0,0)\}$ of the band $ B$. The other two dividing arcs are obtained by picking points $x_i'\in K_i\setminus \{p_i,q_i\}$ in the components opposite to $x_i$, and defining the arc to be $[0,1]\times \{x_i'\}$. Writing $\cA$ for the dividing set consisting of these three arcs, the set $ S\setminus \cA$ has two components. To $K_1\# K_2\subset Y_1\# Y_2$ we add two basepoints, $p$ and $q$, so that $p$, $p_1$ an $p_2$ are in the same component of $S\setminus \cA$, and $q$, $q_1$ and $q_2$ are in the same component of $S\setminus \cA$.

Write $\cF_1=( S,\cA)$ for the surface with divides inside of the 1-handle cobordism $W$, constructed above using connected sum points $x_1\in K_1$ and $x_2\in K_2$ satisfying the configuration condition $(C1)$. We construct the link cobordism $(W',\cF_1')$ from $(Y_1\# Y_2, \bK_1\# \bK_2)$ to $(Y_1\sqcup Y_2, \bK_1\sqcup \bK_2)$ by turning around and reversing the orientation of $(W,\cF_1)$. 

We now define the maps
\[
G_1=F_{W,\cF_1,\frs} \qquad \text{and}\qquad E_1=F_{W',\cF_1',\frs},
\]
 where $\frs$ is the unique $\Spin^c$ structure on the 1-handle or 3-handle cobordism which extends $\frs_1\sqcup \frs_2$.

Naturally, we would like to define the maps $E_2$ and $G_2$ to be the link cobordism maps for the link cobordism maps constructed when the points $x_1$ and $x_2$ satisfy configuration condition $(C2)$. This however, makes for an awkward comparison with the maps $E_1$ and $G_1$, since moving the connected sum point changes both the 3-manifold $Y_1\# Y_2$ and the knot $K_1\# K_2$. Instead we define
\begin{equation} G_2=\tau_{K_1\# K_2}^{-1} \circ F_{W,\bar{\cF}_1,\frs}\circ (\tau_{K_1}|\tau_{K_2}),\qquad \text{and} \qquad  E_2=(\tau_{K_1}^{-1}|\tau_{K_2}^{-1})\circ  F_{W',\bar{\cF}_1',\frs}\circ  \tau_{K_1\# K_2}\label{eq:defEiGimaps}\end{equation} where $\tau_K$ denotes the half twist diffeomorphism of the knot $K$ (in the direction of the knot's orientation, switching the two basepoints). Also $\bar{\cF}_1$ and $\bar{\cF}'_1$ denote the conjugate link cobordisms of $\cF_1$ and $\cF'_1$ (obtained by switching the designation of regions as  type-$\ve{w}$ or  type-$\ve{z}$).

If we interpret the twisting diffeomorphisms in the definition of $G_2$ in terms of twisting dividing sets on cylindrical link cobordisms, we can write $G_2$ as the induced cobordism map for a decorated surface $\cF_2$ inside of $W$. Note that $(W,\cF_1)$ is not diffeomorphic as a decorated link cobordism to $(W,\cF_2)$, even for a diffeomorphism which is not required to be the identity on the boundary, since the cyclic order that $Y_1,$ $Y_2$ and $Y_1\# Y_2$ appear along the boundary of $S_{\ve{w}}$ is reversed between $(W,\cF_1)$ and $(W,\cF_2)$.

The link cobordisms for $E_1,$ $G_1,$ $E_2$ and $G_2$ are shown in Figure~\ref{fig::39}.

\begin{figure}[ht!]
\centering
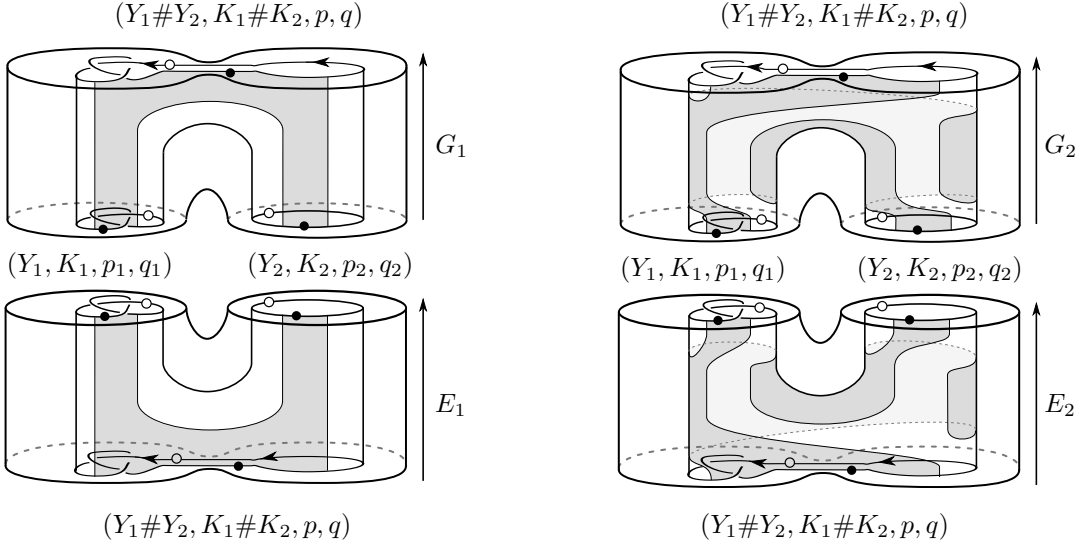
\caption{\textbf{The link cobordisms used to define the maps $E_1,$ $G_1,$ $E_2$ and $G_2$.} See Equation \eqref{eq:defEiGimaps} for the definition of $G_2$ and $E_2$.\label{fig::39}}
\end{figure}

\begin{prop}\label{prop:chainhomotopyequivandconnectedsums}As endomorphisms of $\cCFL^\infty(Y_1,\bK_1)\otimes_{\cR} \cCFL^\infty(Y_2,\bK_2)$, one has \[
E_1\circ G_1\simeq E_2\circ G_2\simeq \id
\]
 through  filtered, equivariant chain homotopies. As endomorphisms of $\cCFL^\infty(Y_1\# Y_2,\bK_1\# \bK_2)$, one has
\[
G_1\circ E_1\simeq G_2\circ E_2\simeq \id,
\] 
through filtered equivariant chain homotopies. 
\end{prop}

The proof of  Proposition~\ref{prop:chainhomotopyequivandconnectedsums} will take several steps. The key observation is that by performing 4-dimensional surgery operations to the link cobordisms $(W',\cF')\circ (W,\cF)$ and $(W,\cF)\circ (W',\cF')$, we can obtain the identity link cobordisms for $(Y_1\sqcup Y_2, \bK_1\sqcup \bK_2)$ and $(Y_1\# Y_2, \bK_1\# \bK_2)$.

\subsection{Surgering a link cobordism on a 0-sphere}
\label{subsec:4dimsurg0}

In this section, we describe the effect of taking the connected sum of two link cobordisms  at two points along the dividing sets. We phrase this in terms of surgering the link cobordism $(W_1\sqcup W_2,\cF_1\sqcup \cF_2)$ on an embedded $0$-sphere.
 
Consider  link cobordism $(S^0\times D^4, N_0)$ where $D^4$ is the unit ball in $\R^4$, $S^0=\{\pm 1\}$ (oriented as the boundary of $[-1,1]$) and $N_0=( S_0,\cA_0)$ is the surface with divides defined by 
\[
 S_0= S^0\times \{(x,y,0,0):x^2+y^2\le 1\}\qquad  \text{and} \qquad \cA_0= S^0\times \{(0,y,0,0):-1\le y\le 1\},
\]
 with designation of type-$\ve{w}$ and type-$\ve{z}$ regions given by
\[
 S_{0,\ve{w}}=S^0\times ( S_0\cap \{x\ge 0\}) \qquad \text{and} \qquad  S_{0,\ve{z}}=S^0\times ( S_0\cap \{x\le 0\}).
\]

Also define the decorated link cobordism $([-1,1]\times S^3,N_1)$ where $N_1=( S_1,\cA_1)$ is the surface with divides
\[
 S_1= [-1,1]\times \{(x,y,0,0): x^2+y^2=1\} \qquad \text{and} \qquad \cA_1= [-1,1]\times \{(0,\pm 1, 0,0)\},
\] 
with designation of type-$\ve{w}$ and type-$\ve{z}$ regions given by
\[
 S_{1,\ve{w}}= S_{1}\cap \{x\ge 0\} \qquad \text{and}\qquad  S_{1,\ve{z}}= S_1\cap \{x\le 0\}.
\]

Given an orientation preserving embedding of $(S^0\times D^4, N_0)$ into $(W_1\sqcup W_2,\cF_1\sqcup \cF_2)$ such that one component of $S^0\times D^4$ is mapped into $W_1$ and the other component is mapped into $W_2$, we can form the connected sum, by removing $( S^0\times D^4,N_0)$ and gluing in $([-1,1]\times S^3, N_1)$, according to the embedding of $( S^0\times D^4,N_0)$. We write $(W_1\# W_2, \cF_1\# \cF_2)$ for the surgered link cobordism.

We  prove the following result about the effect of this 4-dimensional surgery operation:

\begin{prop}\label{prop:connectedsumofcobordisms}Suppose that $(W_1,\cF_1)$ and $(W_2,\cF_2)$ are two link cobordisms, with chosen points $y_1\in \cA_1\subset W_1$ and $y_2\in \cA_2\subset W_2$, as well as an embedding of $(S^0\times D^4, N_0)$, centered at $\{y_1,y_2\}$, which is orientation preserving and maps type-$\ws$ regions to type-$\ws$ regions and maps type-$\zs$ regions to type-$\zs$ regions  Then
\[
F_{W_1\# W_2, \cF_1\# \cF_2, \frs_1\# \frs_2}\simeq F_{W_1\sqcup W_2, \cF_1\sqcup \cF_2,\frs_1\sqcup \frs_2}.
\]
\end{prop}

The proof is based on the fact that the maps are invariant under puncturing a link cobordism at a point along the dividing set, as well as a simple model computation which shows that we can connect up the punctures without changing the map.

\begin{lem}\label{lem:modelcomputationonsphere}Let $(W,\cF)\colon (S^3,\bU)\to (S^3,\bU_1)\sqcup (S^3,\bU_2)$ be the link cobordism obtained by splitting an unknot in $S^3$ into two unknots, and then adding a 3-handle which separates the two unknots (similar to the link cobordisms for connected sums  shown in Figure \ref{fig::39}, when $Y_1=Y_2=S^3$ and $K_1=\bU_1$ and $K_2=\bU_2$). The link cobordism $(W,\cF)$ induces a map
\[
F_{W,\cF,\frs_0}\colon \cCFL^{\infty}(S^3,\bU)\to \cCFL^\infty(S^3\sqcup S^3, \bU_1\sqcup \bU_2)
\]
 which is filtered chain homotopic to
\[
1\mapsto 1,
\] under the identification of both the domain and codomain as $\cR$.
\end{lem}
\begin{proof}It is straightforward to explicitly compute the map $F_{W,\cF,\frs_0}$ in terms of a quasi-stabilization, followed by a band map, and finally followed by a 3-handle map. It is easier, however, to simply use formal properties of the link Floer TQFT. Note that both the domain and range of $F_{W,\cF,\frs_0}$  are filtered chain homotopy equivalent to $\cR=\bF_2[U,V,U^{-1},V^{-1}]$, with vanishing differential. The cobordism map, being an equivariant map, is thus determined by its value on $1\in \cR$. However we can cap off either copy of $(S^3,\bU_i)$ with a 4-ball, and the result must be the identity map on $\cR$. This forces the image of $1\in \cCFL^\infty(S^3,\bU)$ in $\cCFL^\infty(S^3\sqcup S^3,\bU_1\sqcup \bU_2)$ to also be 1, completing the proof.
\end{proof}

We can now prove the connected sum formula for the link cobordism maps:

\begin{proof}[Proof of Proposition \ref{prop:connectedsumofcobordisms}]Using the composition law from Equation~\eqref{eq:compositionlaw}, it is sufficient to show the claim in the case that $(W_1,\cF_1)$ and $(W_2,\cF_2)$ are the identity cobordisms. Decompose the link cobordism representing $(W_1\# W_2,\cF_1\# \cF_2)$ as a 0-handle, followed by the cobordism considered in Lemma \ref{lem:modelcomputationonsphere}, followed by the link cobordisms $(W_1,\cF_1)$ and $(W_2,\cF_1)$, each with a puncture removed at a point along the divides. This is shown in Figure \ref{fig::8}.

\begin{figure}[ht!]
\centering
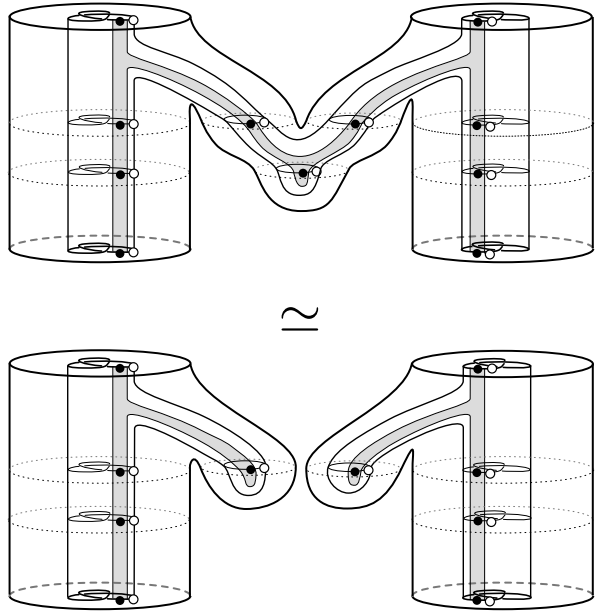
\caption{\textbf{The link cobordisms $(W_1\sqcup W_2,\cF_1\sqcup \cF_2)$ and $(W_1\# W_2, \cF_1\# \cF_2)$ considered in Proposition \ref{prop:connectedsumofcobordisms}.} Using Lemma \ref{lem:modelcomputationonsphere} and the composition law, we see that the two cobordism maps are equal.\label{fig::8}}
\end{figure}

Using the computation from Lemma \ref{lem:modelcomputationonsphere}, the middle portion of the cobordism $(W_1\# W_2, \cF_1\# \cF_2)$ obtained by attaching a 0-handle, quasi-stabilizing, attaching a band, and attaching a 3-handle can be replaced with two 0-handle attachments, without changing the cobordism map. By \cite{ZemCFLTQFT}*{Lemma~12.9}, removing a ball from $W$ which intersects $\cF$ along an arc of the dividing set (and adding the new copy of $(S^3,\bU)$ to one of the ends of $W$) does not affect the cobordism map. Hence the two cobordism maps are equal.
\end{proof}

\subsection{Surgering a link cobordism on a 1-sphere}
\label{subsec:4dimsurg1}

We now describe a second surgery relation for the link cobordism maps, for surgering on framed 1-spheres in the 4-manifold $W$.

If $\gamma$ is an embedded 1-manifold in the interior of $W$ and we are given an identification of a regular neighborhood of $\gamma$ with $S^1\times D^3$, we can remove the neighborhood of $\gamma$ and glue in a copy of $D^2\times S^2$. Let us call the surgered 4-manifold $W(\gamma)$. Our goal is to relate the invariants for link cobordisms in $W$ to those for the link cobordisms in the surgered 4-manifold $W(\gamma)$. To do this, we need to assume that $\gamma$ is embedded in the surface $ S\subset W$, and that in fact  $\gamma$ is one of the dividing curves in $\cA$. In this case, we can simultaneously surger both $W$ and $\cF$ to get a new decorated link cobordism. We describe precisely the procedure, below.

We define two link cobordisms, $(S^1\times D^3,M_1)$ and $(D^2\times S^2,M_2)$, both from $(\varnothing,\varnothing)$  to $(S^1\times S^2,\bL_0)$, for some link $\bL_0$ in $S^1\times S^2$. The decorated surfaces $M_1$ and $M_2$ are shown in Figure \ref{fig::36}. Write $M_i=( S_i,\cA_i)$. The underlying surface $ S_1$ is the annulus
\[
 S_1=S^1\times \{(y,0,0): -1\le y\le 1\}.
\] The underlying surface $ S_2$ is the pair of disks
\[
 S_2=D^2\times \{(\pm 1, 0 ,0)\}.
\]

\begin{figure}[ht!]
\centering
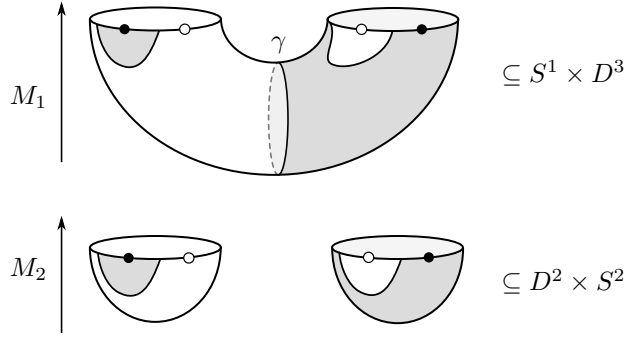
\caption{\textbf{The decorated surfaces $M_1$ and $M_2$.} The maps $f$ and $g$ in Lemma \ref{lem:4manifoldsurgery2} are the link cobordism maps induced by $(S^1\times D^3,M_1)$ and $(D^2\times S^2,M_2)$, respectively, viewed as link cobordisms from $(\varnothing,\varnothing)$ to $(S^1\times S^2,\bL_0)$. The closed curve $\gamma$ (which we think of as the curve which we perform surgery on) is shown on $M_1$. \label{fig::36}}
\end{figure}

  Notice that
\[
\d  S_1=\d  S_2=S^1\times \{(\pm 1,0,0)\}\subset S^1\times S^2. 
\]
 Write $\bL_0$ for the oriented link $\d S_1=\d S_2\subset S^1\times S^2$, decorated with two basepoints per component, oriented as the boundary of the surfaces $ S_1$ and $ S_2$. Note that $\bL_0$ is null-homologous in the sense that the total homology class is zero, since the two components go in opposite directions around $S^1\times S^2$.

Given a $\Spin^c$  structure $\frs$ on $W$, the restriction of $\frs$ to $W\setminus N(\gamma)$ is torsion on $\d N(\gamma)=S^1\times S^2$.  By considering the Mayer-Vietoris long exact sequence for cohomology, the obstruction to gluing a $\Spin^c$ structure on $W\setminus N(\gamma)$ and a $\Spin^c$ structure on $D^2\times S^2$ lies in $H^2(S^1\times S^2;\Z)$ (i.e. the restriction of the two $\Spin^c$ structures to $S^1\times S^2).$ There is a unique $\Spin^c$ structure on $D^2\times S^2$ which is torsion on $S^1\times S^2$, which we can thus glue to the $\Spin^c$ structure $\frs|_{W\setminus N(\gamma)}$. The ambiguity in gluing lies in $\delta H^1(S^1\times S^2;\Z)\subset H^2(W(\gamma);\Z)$.   By exactness of the Mayer-Vietoris sequence, the property that $\delta H^1(S^1\times S^2;\Z)$ vanishes is equivalent to the property that the restriction map $H^1(W\setminus N(\gamma);\Z)\to H^1(S^1\times S^2;\Z)$ is surjective. Surjectivity of the previous map is equivalent to the image of $[\gamma]$ being \emph{non-divisible} in $H_1(W;\Z)/\Tors$, i.e., having the property that if $n [\gamma']=[\gamma]$ for some $[\gamma']\in H_1(W;\Z)/\Tors$, then $n=\pm 1$.

When $[\gamma]$ is non-divisible in $H_1(W;\Z)/\Tors$, we will write $\frs(\gamma)$ for the $\Spin^c$ structure on $W(\gamma)$, as above.

\begin{prop}\label{prop:surgeryonclosedcurve} Suppose $(W,\cF)$ is a link cobordism with $\cF=( S,\cA)$ and $\gamma$ is a closed curve in $\cA\subset S$ such that the image $[\gamma]\in H_1(W;\Z)/\Tors$ is non-divisible. Further, suppose we are given an orientation preserving diffeomorphism of $\phi$ between a regular neighborhood $N(\gamma)$ of $\gamma$ and the decorated link cobordism $(S^1\times D^3, M_1)$, described above. Assume further that $\phi$ maps the dividing arcs of $\cF$ to the dividing arcs of $M_1$, and sends type-$\ws$ regions to type-$\ws$ regions, and similarly for type-$\zs$ regions. Writing $(W(\gamma),\cF(\gamma))$ for the surgered link cobordism (using  $\phi$ to glue in $(D^2\times S^2, M_2)$), we have
\[
F_{W,\cF,\frs}\simeq F_{W(\gamma),\cF(\gamma),\frs(\gamma)},
\] 
where $\frs(\gamma)$ is as above.
\end{prop}

\begin{rem}
The requirement that $[\gamma]\in H_1(W;\Z)/\Tors$ is non-divisible can be relaxed, though $\frs(\gamma)$ is no longer uniquely specified. Instead, using the composition law, we obtain
\[
F_{W,\cF,\frt}\simeq \sum_{\substack{\frt\in \Spin^c(W(\gamma))\\ \frt|_{W\setminus N(\gamma)}=\frs|_{W\setminus N(\gamma)}}} F_{W(\gamma),\cF(\gamma),\frt}.
\]
\end{rem}
The main ingredient of the proof of Proposition~\ref{prop:surgeryonclosedcurve} is the following lemma:

\begin{lem}\label{lem:4manifoldsurgery2}For the decorated surfaces with divides $M_1$ and $M_2$, described above, the two maps
\[
F_{S^1\times D^3, M_1,\frt},\quad F_{D^2\times S^2, M_2,\frt_0}\colon \cCFL^\infty(\varnothing,\varnothing)\to \cCFL^\infty(S^1\times S^2,\bL_0)
\]
 are filtered chain homotopy equivalent.  Here $\frt$ and $\frt_0$ are unique $\Spin^c$ structures on $S^1\times D^3$ and $D^2\times S^2$ which are torsion on $S^1\times S^2$. 
\end{lem}

\begin{proof}We will use formal properties of the link Floer TQFT. See Appendix \ref{app:directverification} for an alternate, explicit proof by decomposing the link cobordisms into pieces, and counting holomorphic curves to compute the map for each piece. A diagram for $(S^1\times S^2,\bL_0)$ is shown in Figure \ref{fig::40}. Using the intersection points labeled in that diagram, write
\[
\cCFL^\infty(S^1\times S^2,\bL_0,\frs_0)=\langle \theta_1^{\ws}\theta_2^{\ws},\theta_1^{\ws}\xi_2^{\ws},\xi_1^{\ws}\theta_2^{\ws},\xi_1^{\ws}\xi_2^{\ws}\rangle\otimes_{\bF_2} \cR.
\]
Here $\theta_{i}^{\ws}$ and $\xi_{i}^{\ws}$ respectively denote the higher and lower $\gr_{\ws}$-graded intersection points of $\alpha_i\cap \beta_i$.

\begin{figure}[ht!]
	\centering
	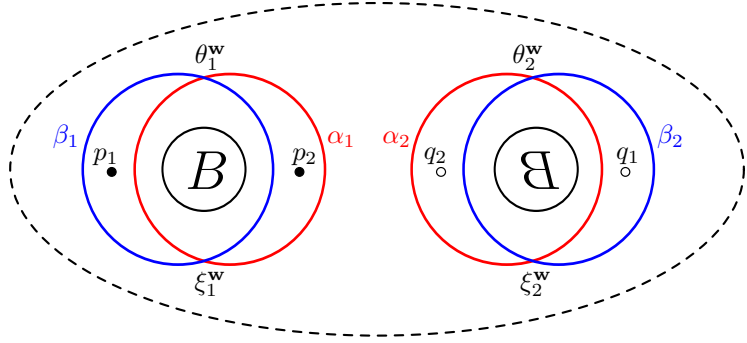
	\caption{\textbf{A genus one diagram for $(S^1\times S^2,\bL_0)$.} The intersection point $\theta^{\ws}_1\times \theta^{\ws}_2$ is the highest $\gr_{\ws}$-graded intersection point, while $\xi_1^{\ws}\times \xi_2^{\ws}$ is the lowest. \label{fig::40}}
\end{figure}

  Using the grading change formulas from Equations \eqref{eq:grading1}, \eqref{eq:grading2} and \eqref{eq:grading3}, one sees that both $F_{S^1\times D^3, M_1,\frt}$ and $F_{D^2\times S^2, M_2,\frt_0}$ induce $\gr_{\ve{w}}$ and $\gr_{\ve{z}}$-grading changes of $0$. For example, to compute the  $\gr_{\ve{w}}$-grading changes of the two maps, one computes that
\[
\tilde{\chi}( S_{1,\ve{w}})=0\qquad  \text{and} \qquad \chi(S^1\times D^3)=0,
\]
 while
\[
\tilde{\chi}( S_{2,\ve{w}})=1\qquad \text{and} \qquad \chi(D^2\times S^2)=2,
\]
 and that the other terms in the grading change formula vanish. Here $ S_{i,\ve{w}}$ denotes the  type-$\ve{w}$ sub-surface of $ S_i$.

Note that since $\cCFL^\infty(\varnothing,\varnothing)$ is by definition equal to $\cR$, it is sufficient to show that the value of the two maps on $1\in \cR$ are equal. Since the maps are filtered, they map $1$ into $\cCFL^-(S^1\times S^2,\bL_0,\frs_0)$. The set of elements of $\cCFL^-(S^1\times S^2,\bL_0,\frs_0)$ which are in both $\gr_{\ve{w}}$- and $\gr_{\ve{z}}$-grading zero is equal to the two dimensional vector space 
\[
V:=\Span_{\bF_2}(\theta_1^{\ws}\xi_2^{\ws},\xi_1^{\ws}\theta_2^{\ws}).
\]
 As such, it is sufficient to restrict to the hat flavor, where we set $U=V=0$. Let us write
\[
f=F_{S^1\times D^3,M_1,\frt}\qquad \text{and} \qquad g=F_{D^2\times S^2,M_2,\frt_0}.
\]
 Note that using the bypass relation, we can write 
\[
f=A+B,
\]
 where $A$ and $B$ are the link cobordism maps for the link cobordisms shown in Figure \ref{fig::38}, which have underlying, undecorated link cobordisms which are the same as those of $(S^1\times D^3,M_1)$.

\begin{figure}[ht!]
\centering
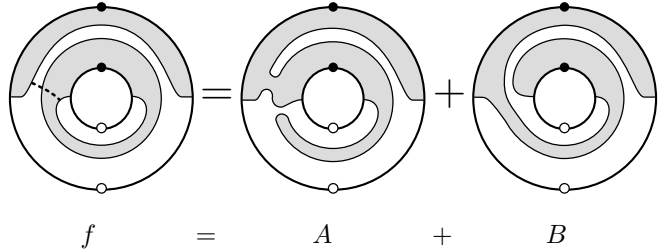
\caption{\textbf{The link cobordisms inducing the maps $f,$ $A$ and $B$, and a bypass relation between all three.} The bypass region is a neighborhood of the dashed line on the left. All three have underlying undecorated link cobordism $(S^1\times D^3,S^1\times \{(t,0,0):-1\le t \le 1\})$.\label{fig::38}}
\end{figure}

 Let $A'$ and $B'$ be the link cobordism maps obtained by turning around and reversing the orientation of the link cobordisms for $A$ and $B$ (i.e. viewing them as link cobordisms from $S^1\times S^2$ to $\varnothing$). We claim that we have the following relations:
\begin{align}
(A'\circ A)(1)&=(B'\circ B)(1)=0\label{eq:surgrels3}\\
(A'\circ B)(1)&=(B'\circ A)(1)=1\label{eq:surgrels4}\\
(A'\circ f)(1)&=(A'\circ g)(1)=1\label{eq:surgrels1}\\
(B'\circ f)(1)&=(B'\circ g)(1)=1.\label{eq:surgrels2}
\end{align}

Before we prove the above relations, we will show that they are sufficient to show that $f=g$. Equations \eqref{eq:surgrels3} and \eqref{eq:surgrels4} imply that $A'|_V$ and $B'|_V$ are linearly independent and hence form a basis  of the 2-dimensional vector space $V^\vee=\Hom_{\bF_2}(V,\bF_2)$. Equations \eqref{eq:surgrels1} and \eqref{eq:surgrels2} then imply that their values on $f(1)$ and $g(1)$ agree,  so $f=g$.

Hence it remains only to establish Equations \eqref{eq:surgrels3}--\eqref{eq:surgrels2}. This is done by interpreting each of the compositions as the link cobordism map for a link cobordism from $(\varnothing,\varnothing)$ to $(\varnothing,\varnothing)$. Notice that all compositions represent either a decorated $S^2$ which is unknotted in $S^4$, or a copy of $S^1\times U$ inside of $S^1\times S^3$ (where $U$ denotes the unknot). This is shown in Figure \ref{fig::37}. In Figure \ref{fig::37}, the link cobordism on the right side of the top row is just a 0-handle followed by a 4-handle (each containing standard disks). The composition induces the identity map. The maps associated to the two link cobordisms on left side of the top row can be computed by using the map from Lemma \ref{lem:modelcomputationonsphere} (splitting $(S^3,\bU)$ into two copies of itself), composed with the diffeomorphism associated to twisting one copy of the unknot in a full twist, composed with the link cobordism map from Lemma \ref{lem:modelcomputationonsphere}, turned upside down. This is clearly the identity. Finally the link cobordism shown on the bottom row, which induces the maps $A'\circ A$ and $B'\circ B$, induces the zero map, since by moving the dividing sets around we can factor the induced map through the map $\Phi|\Psi$ on $\hat{\CFL}(S^3,\bU)\otimes_{\bF_2} \hat{\CFL}(S^3,\bU)$, which trivially vanishes, since each of $\Phi$ and $\Psi$ vanish for an unknot in $S^3$ with exactly two basepoints. Hence Equations \eqref{eq:surgrels3}--\eqref{eq:surgrels2} hold, completing the proof.
\end{proof}

\begin{figure}[ht!]
\centering
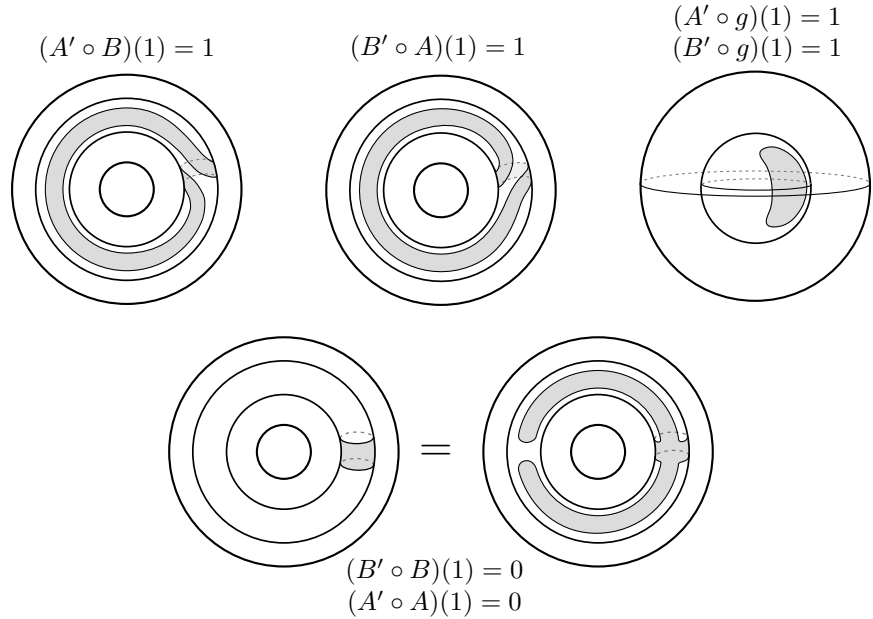
\caption{\textbf{Link cobordisms representing compositions appearing in Equations \eqref{eq:surgrels3}, \eqref{eq:surgrels4}, \eqref{eq:surgrels1} and \eqref{eq:surgrels2}.} The link cobordism in the top right consists of an unknotted $S^2$ inside of $S^4$. The remaining represent decorations of $S^1\times U\subset S^1\times S^3$. \label{fig::37}}
\end{figure}

\begin{proof}[Proof of Proposition \ref{prop:surgeryonclosedcurve}] The composition law implies that
\[
F_{W,\cF,\frs}\simeq F_{W\setminus N(\gamma),\cF\setminus N(\gamma),\frs|_{W\setminus N(\gamma)}}\circ F_{S^1\times D^3,M_1,\frt}
\]
 which is equal by Lemma \ref{lem:4manifoldsurgery2} to
\[
F_{W\setminus N(\gamma),\cF\setminus N(\gamma),\frs|_{W\setminus N(\gamma)}}\circ F_{D^2\times S^2,M_2,\frt_0}.
\]
The condition that $[\gamma]\in H_1(W;\Z)/\Tors$ is non-divisible is equivalent to the condition that there is a unique $\Spin^c$ structure $\frs(\gamma)$ on $W(\gamma)$ which restricts to $\frs|_{W\setminus N(\gamma)}$ on $W\setminus N(\gamma)$. Hence the composition law yields that the above composition is exactly $F_{W(\gamma),\cF(\gamma),\frs(\gamma)}$.
\end{proof}

\subsection{Proof of Proposition \ref{prop:chainhomotopyequivandconnectedsums}}

We now show that $E_i\circ G_i$ and $G_i\circ E_i$ are chain homotopic to the identity as consequences of the 4-dimensional surgery relations proven in the previous sections:

\begin{proof}[Proof of Proposition \ref{prop:chainhomotopyequivandconnectedsums}] The fact that $E_1\circ G_1\simeq \id$ follows from Proposition \ref{prop:connectedsumofcobordisms}, since the composition of the link cobordisms defining $E_1$ and $G_1$ is a connected sum of two identity link cobordisms (in the sense of Proposition \ref{prop:connectedsumofcobordisms}).

Similarly, the relation $G_1\circ E_1\simeq \id$  follows from Proposition \ref{prop:surgeryonclosedcurve}. Write $(W,\cF)$ for the link cobordism inducing the map $G_1\circ E_1$. There is a single divide $\gamma$ which is a closed curve. After performing an isotopy of the other two dividing arcs, we can arrange so that a neighborhood of this dividing arc looks like the surface $M_1$ in Figure \ref{fig::36}. After replacing this copy of $(S^1\times D^3,M_1)$ with $(D^2\times S^2,M_2)$ we are left with the identity link cobordism from $(Y_1\# Y_2, K_1\# K_2)$ to itself. We briefly describe the diffeomorphism between $W(\gamma)$ and $[0,1]\times  (Y_1\# Y_2)$.  Note that $W$ has a handle decomposition given by a 3-handle followed by a ``dual'' 1-handle. We pull the 1-handle below the 3-handle. The handle decomposition of $W(\gamma)$ is obtained by exchanging the 1-handle for a 2-handle attached along a 0-framed unknot. However, the original 3-handle cancels this 2-handle, and we are left with a handle decomposition for $W(\gamma)$ which has no 4-dimensional handles.

Showing that this diffeomorphism restricts to a diffeomorphism between the embedded surfaces is a straightforward. To describe this in terms of handle operations, one picks a handle decomposition of the surface in the link cobordism for $G_1\circ E_1$ with a band $B$ which disconnects $K_1\#K_2$, and a dual band $B'$ which connects $K_1\sqcup K_2$ to form $K_1\# K_2$.  Via an index 0/1 handle cancellation,  one replaces $B$  with a disk $D$ and two bands, $B_1$ and $B_2$. The band $B'$ has both ends on $D$. The effect of the 4-dimensional surgery operation is to trade $B'$ and $D$ for two disks, $D_1$ and $D_2$. After removing the canceling 2-handle and 3-handle of the surgered 4-manifold, as described in the previous paragraph, the disks $D_1$ and $D_2$ then cancel $B_1$ and $B_2$, and we are left with the identity link cobordism.

 Finally we note that the curve $\gamma$ which we surger on represents a non-divisible element of $H_1(W;\Z)/\Tors$ since $W$ has a handle decomposition as a 1-handle followed by a 3-handle, and $\gamma$ goes around the 1-handle exactly once. Hence by Proposition \ref{prop:surgeryonclosedcurve}, the composition $G_1\circ E_1$ is chain homotopy equivalent to the link cobordism map for the identity link cobordism from $(Y_1\# Y_2, K_1\# K_2)$ to itself.

Finally, to see that $G_2\circ E_2$ and $E_2\circ G_2$ are chain homotopy equivalent to the identity, we note that by definition (see Equation \eqref{eq:defEiGimaps}) the link cobordism maps for $G_2$ and $E_2$ are defined as a composition involving the conjugates of the link cobordisms used to define $G_1$ and $E_1$, as well as the half twist diffeomorphisms $\tau_{K_1},$ $\tau_{K_2}$ and $\tau_{K_1\# K_2}$. One can first cancel the  adjacent pair of half twists in the middle of either composition $E_2\circ G_2$ or $G_2\circ E_2$, since they go in opposite directions. Then one simply uses Proposition \ref{prop:connectedsumofcobordisms} or \ref{prop:surgeryonclosedcurve} to surger the cobordism, leaving only a composition of two half twists on the identity cobordism, which go in opposite directions. After canceling these, one is left with the identity cobordism.
\end{proof}

 \begin{rem}The 4-dimensional surgery formulas from Propositions \ref{prop:connectedsumofcobordisms} and \ref{prop:surgeryonclosedcurve} cannot be applied to the maps $E_i\circ G_j$ and $G_j\circ E_i$ if $i\neq j$. For example, the dividing set of the link cobordism corresponding to the composition $G_2\circ E_1$ does not contain a closed curve $\gamma$. Similarly the link cobordism corresponding to the composition $E_1\circ G_2$ is not diffeomorphic to a connected sum of two link cobordisms in the sense of Proposition \ref{prop:connectedsumofcobordisms}. This is because on a connected sum of two link cobordisms, two dividing arcs would have both endpoints in $\{0,1\}\times Y_1$ or both endpoints in $\{0,1\}\times Y_2$, while two arcs would have one endpoint in $\{0,1\}\times Y_1$ and another in $\{0,1\}\times Y_2$. In the decorated link cobordism for $E_1\circ G_2$, all four arcs have one end in $\{0,1\}\times Y_1$ and one end in $\{0,1\}\times Y_2$.
 \end{rem}
 
\section{The map $\iota_K$ on connected sums}\label{sec:connectedsumformulaforiotaK}

 In this section, we prove Theorem \ref{thm:B}:

\begin{customthm}{\ref{thm:B}}\label{thm:homotopytypeofinvolution}Suppose $\bK_1=(K_1,p_1,q_1)$ and $\bK_2=(K_2,p_2,q_2)$ are two doubly based, null-homologous knots in $Y_1$ and $Y_2$ respectively, and let $\bK_1\# \bK_2=(K_1\# K_2,p,q)$ denote their connected sum inside of $Y_1\# Y_2$, with two basepoints. If $\frs_1\in \Spin^c(Y_1)$ and $\frs_2\in \Spin^c(Y_2)$ are self-conjugate, then the two filtered chain homotopy equivalences 
\[
G_1,G_2\colon \cCFL^\infty(Y_1\sqcup Y_2,\bK_1\sqcup\bK_2, \frs_1\sqcup \frs_2)\to \cCFL^\infty(Y_1\# Y_2,\bK_1\# \bK_2, \frs_1\# \frs_2)
\](constructed in the last section) satisfy
\[
G_1(\iota_{K_1}| \iota_{K_2}+\Phi_{p_1}\iota_{K_1}| \Psi_{q_2}\iota_{K_2})\eqsim (\iota_{K_1\# K_2}) G_1,
\]
 and
\[
G_2(\iota_{K_1}| \iota_{K_2}+\Psi_{q_1}\iota_{K_1}| \Phi_{p_2}\iota_{K_2})\eqsim (\iota_{K_1\# K_2}) G_2.
\]
\end{customthm}

\begin{proof}[Proof of Theorem \ref{thm:homotopytypeofinvolution}] Let $\tau_{K_i}$ denote the half twist diffeomorphism on 
$(Y_i,\bK_i)$, and let $\tau_{K_1\# K_2}$ denote the half twist diffeomorphism on $(Y_1\# Y_2, \bK_1\# \bK_2)$. Let 
\[
(W,\cF_1)\colon (Y_1\sqcup Y_2, \bK_1\sqcup \bK_2)\to (Y_1\# Y_2,\bK_1\# \bK_2)
\] 
 be the link cobordism described in Section \ref{sec:constructcobs}. The decorations on the surface are redrawn schematically in Figure~\ref{fig::13}. Recall that we defined 
\[
G_1:=F_{W,\cF_1,\frs}
\]
 to be the link cobordism map for the unique $\Spin^c$ structure $\frs$ on $W$  which extends $\frs_1\sqcup \frs_2$ on $Y_1\sqcup Y_2$. Using the composition law, the composition $\tau_{K_1\# K_2}\circ G_1$ is also a link cobordism map, as shown in Figure \ref{fig::1}.

\begin{figure}[ht!]
\centering
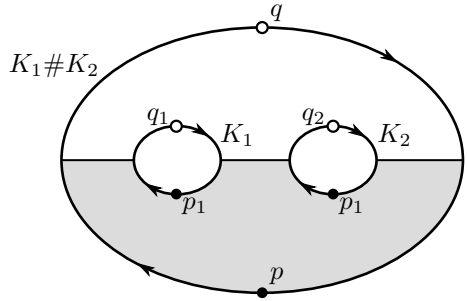
\caption{\textbf{The decoration on the surface used in the cobordism defining the map $G_1$.} This is the surface inside of the link cobordism from $K_1\sqcup K_2$ to $K_1\# K_2$ shown in Figure \ref{fig::39}. The arrows indicate the orientation of the knot components. \label{fig::13}}
\end{figure}

\begin{figure}[ht!]
\centering
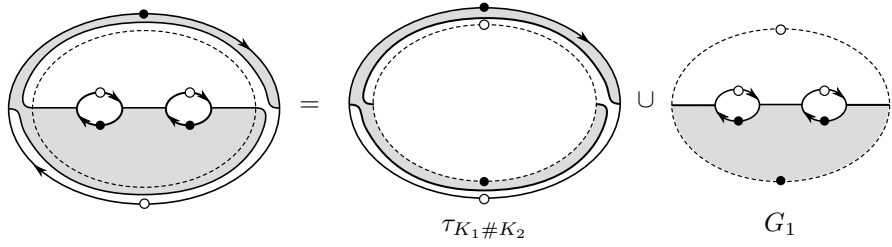
\caption{\textbf{The decoration on a link cobordism whose cobordism map is $\tau_{K_1\# K_2}\circ F_{W,\cF_1}$.} The arrows indicate the orientation of the knots. \label{fig::1}}
\end{figure}

We now use the bypass relation from Lemma \ref{thm:D} on the link cobordism for $\tau_{K_1\# K_2}\circ G_1$. The disk we pick for the bypass relation is a regular neighborhood of the dashed arc shown in Figure~\ref{fig::2}. The resulting three dividing sets, in a neighborhood of the dashed arc, are also shown in figure~\ref{fig::2}.

\begin{figure}[ht!]
\centering
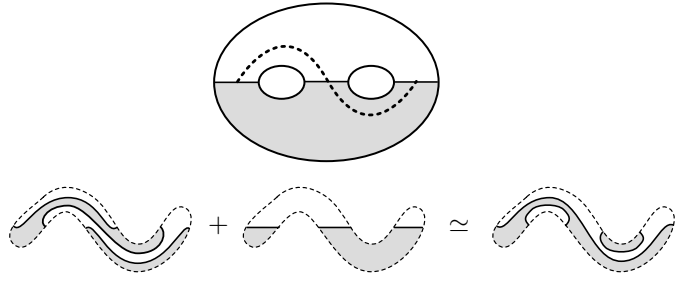
\caption{\textbf{A bypass relation.} We perform a bypass in a neighborhood of the dashed arc on top.  The resulting bypass triple of dividing sets are drawn on the line below. \label{fig::2}}
\end{figure}

\begin{figure}[ht!]
\centering
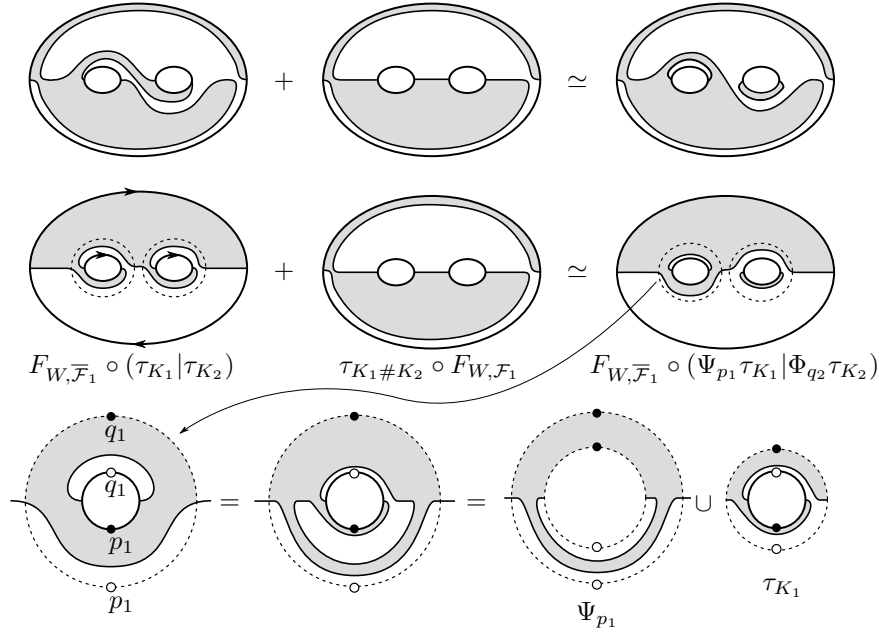
\caption{\textbf{Isotoping the dividing sets in the bypass relation, to derive the connected sum formula.} Composing with $\tau_{K_1\# K_2}$ and using the bypass relation for the bypass arc in Figure \ref{fig::2} yields the top row.  Performing some manipulations yields the middle row. The cobordisms go from the inner two circles to the outer circle. The bottom row is a zoomed-in manipulation of the one of the regions in the middle row, to illustrate the induced map.\label{fig::3}}
\end{figure}

In Figure \ref{fig::3}, we compose with the half twist map $\tau_{K_1\# K_2}$, and perform some manipulations. From Figure~\ref{fig::3} and Lemma~\ref{lem:cobordismsforPhiPsi} we see that
\[
F_{W,\bar{\cF}_1} \circ \tau_{K_1\sqcup K_2}+\tau_{K_1\# K_2}\circ  F_{W,\cF_1}\simeq F_{W,\bar{\cF}_1}\circ (\Psi_{p_1}\tau_{K_1}|\Phi_{q_2}\tau_{K_2}).
\]
 Here $\cF_1$ is the surface with divides used to define $G_1$, and $\bar{\cF}_1$ is the conjugate surface with divides (i.e. the one resulting from changing the types of regions from $\ve{w}$ to $\ve{z}$, and vice versa). We now compose with the conjugation map, $\eta$, on the left, which yields
\[
\eta\circ  F_{W,\bar{\cF}_1}\circ  \tau_{K_1\sqcup K_2}+\eta\circ  \tau_{K_1\# K_2}\circ F_{W,\cF_1}\simeq\eta\circ  F_{W,\bar{\cF}_1}\circ (\Psi_{p_1}\tau_{K_1}|\Phi_{q_2}\tau_{K_2}).
\]
 Using Theorem \ref{thm:C}, we have $\eta \circ F_{W,\bar{\cF}_1}\simeq F_{W,\cF_1}\circ \eta$. Since  $\iota_K=\tau_K\circ \eta$, by definition, we arrive at
\[
F_{W,\cF_1}\circ (\iota_{K_1}|\iota_{K_2})+\iota_{K_1\# K_2}\circ  F_{W,\cF_1}\simeq  F_{W,\cF_1}\circ(\Phi_{p_1}\iota_{K_1}|\Psi_{q_2}\iota_{K_2}).
\] 
Since $G_1=F_{W,\cF_1}$ by definition, the first statement of the theorem now follows.

The second statement, involving the map $G_2$, follows from an easy modification of the above argument.
\end{proof}

\section{The homomorphism $\cC\to \frI_K$}\label{sec:homomorphism}
We finally put together the pieces of our theory and describe the homomorphism $\cC\to \frI_K$. For a knot $K\in S^3$, let $[\cCFL^\infty(S^3,K)]$ denote the local equivalence class in $\frI_K$ determined by the tuple $(\cCFL^\infty(\Sigma,\ve{\alpha},\ve{\beta},p,q),\d, \bT_{\as}\cap \bT_{\bs},\iota_K)$ for a diagram $(\Sigma,\ve{\alpha},\ve{\beta},p,q)$ of $(S^3,K,p,q)$.  Note that with no additional effort, we can extend this homomorphism to the group $\cC_3^\Z$ consisting of pairs of integer homology spheres with knots, modulo integer homology concordance (i.e. integer homology cobordisms with a smoothly embedded genus zero surface). The homomorphism $\cC\to \frI_K$ factors through the natural map $\cC\to \cC_3^\Z$.

\begin{customthm}{\ref{thm:E}}The map 
	\[
	K\mapsto [\cCFL^\infty(S^3,K)],
	\] 
	defines a homomorphism from the smooth concordance group $\cC$ to $\frI_K$.
\end{customthm}

\begin{proof} We first claim that the map $K\mapsto [\cCFL^\infty(S^3,K)]$ descends to a well defined map on the smooth concordance group. If $(W,\cF)$ is a concordance from $K$ to $K'$ (or more generally an integer homology concordance) then we consider the map $F_{W,\cF,\frs}$ for the unique $\Spin^c$ structure on $W=[0,1]\times  S^3$. Note that by reversing the orientation and turning around $(W,\cF)$, we get a map $F_{-W,-\cF,\frs}$ in the opposite direction. We claim that $F_{W,\cF,\frs}$ and $F_{-W,-\cF,\frs}$ induce local equivalences between $\cCFL^\infty(S^3,K)$ and $\cCFL^\infty(S^3,K')$.

Note that $\cCFL^\infty(S^3,K_i)$ decomposes over Alexander gradings as $\bigoplus_{k\in \Z} \CFK^\infty(S^3,K_i)\{k\}$ where we think of $\CFK^\infty(S^3,K_i)$ as being concentrated in Alexander grading zero, and $\{k\}$ denotes a shift in the Alexander grading. Hence the homology of each $\cCFL^\infty(S^3,K_i)$ decomposes as
\[
\bigoplus_{k\in \Z} H_*(\CFK^\infty(S^3,K_i))\{k\}\iso \bigoplus_{k\in \Z}\HF^\infty(S^3)\{k\}.
\]
 The Alexander grading change of $F_{W,\cF,\frs}$ is equal to zero, by the Alexander grading change formula from Equation \eqref{eq:grading1}. The map $F_{W,\cF,\frs}$ on $\cHFL^\infty$ thus becomes simply the regular cobordism map $F_{W,\frs}^\infty$ in each Alexander grading. The map $F_{W,\frs}^\infty\colon \HF^\infty(S^3)\to \HF^\infty(S^3)$ is the identity map since $W=[0,1]\times S^3$. More generally, if $(W,\cF)\colon (Y_1,\bK_1)\to (Y_2,\bK_2)$ is a homology concordance, then by \cite{OSIntersectionForms} the map $F_{W,\frs}^\infty\colon\HF^\infty(Y_1,\frs_1)\to \HF^\infty(Y_2,\frs_2)$ will be an isomorphism, and induces Maslov grading change zero by the grading change formula from \cite{OSTriangles}.  We note that the link cobordism maps for a concordance intertwine the $\iota_K$ maps on the two ends, using Theorem \ref{thm:C} and the fact that the genus of the surface in a concordance is zero, so the half twist in the divides can be pulled from one end to the other. The map $F_{-W,-\cF,\frs}\colon \cCFL^\infty(S^3,K_2)\to \cCFL^\infty(S^3,K_1)$ intertwines the involutions and is an isomorphism on homology by the same reasoning. This implies that the map $\cC\to \frI_K$ is well defined.

Finally, we claim that the map $\cC\to \frI_K$ is a homomorphism. By Theorem \ref{thm:B}, the local equivalence class $[\cCFL^\infty (S^3,K_1\# K_2)]$ is equal to the product $[\cCFL^\infty(S^3,K_1)]\times [\cCFL^\infty(S^3,K_2)]$ in $\frI_K$, completing the proof.
\end{proof}

\section{Computations and examples}
\label{sec:examples}
In this section we apply Theorem \ref{thm:B} to compute involutive concordance invariants $\bar{V}_0$ and $\underline{V}_0$ from \cite{HMInvolutive} of connected sums of various knots. We compute the invariants for $T_r\# T_r$  ``by hand'' to demonstrate the formula, and then describe the results of some computer computations.

 We recall the definition of $V_0,\bar{V}_0$ and $\underline{V}_0$. If $(C,\d)$ is a $\Q$-graded chain complex, freely generated over $\bF_2[U]$, with a homotopy involution $\iota$, we consider the complex
\[
I^-(C,\iota)=\Cone(C\xrightarrow{Q(1+\iota)} Q\cdot C),
\]
 a chain complex freely generated over the ring $\bF_2[U,Q]/(Q^2)$. The grading on the cone is obtained by shifting grading of the domain on the cone up by one. Similarly one can define complexes $I^\infty(C,\iota)$ and $I^+(C,\iota)$. There is a long exact sequence
\[
\cdots \to H_*(I^-(C,\iota))\xrightarrow{i} H_*(I^\infty(C,\iota))\xrightarrow{\pi} H_*(I^+(C,\iota))\to \cdots.
\]

 Assuming  that $H_*(C\otimes\bF_2[U,U^{-1}])$ is isomorphic to $\bF_{2}[U,U^{-1}]$,  we can define $d(C)$ as the maximal grading of a non-torsion element of $H_{*}(C)$. Similarly, following \cite{HMInvolutive}, we define
 \[
 \underline{d}(C,\iota):=\max \{r\equiv d(C)+1 (\Mod{2})\, |\, \, H_r(I^-(C,\iota))\to H_r(I^\infty(C,\iota)) \text{ is nontrivial}\}-1,
 \]
  and 
 \[
 \bar{d}(C,\iota):=\max \{r\equiv d(C)(\Mod{2})\, |\, \, H_r(I^-(C,\iota))\to H_r(I^\infty(C,\iota)) \text{ is nontrivial}\}.
 \]
 Note that unlike \cite{HMInvolutive} and \cite{HMZConnectedSum}, the convention in this paper is that $\HF^-(S^3)$ has top degree generator in grading $0$.

To define the concordance invariants $\bar{V}_0$ and $\underline{V}_0$, one first defines $A_0^-(K)$ as the subcomplex of $\CFK^\infty(S^3,K)=\cCFL^\infty(Y,K)_0$ generated by terms $[\ve{x},i,j]$ with $i\le 0$,  $j\le 0$ and $A(\ve{x})+(i-j)=0$.  In the notation used elsewhere in this paper, one would write this as the module generated by monomials $\ve{x}\cdot U^iV^j$ with $i,j\ge 0$ and $A(\ve{x})+j-i=0$. The map $\iota_K$ on $\CFK^\infty(K)$ induces a homotopy involution $\iota_0$ on $A_0^-(K)$. The complex $A_0^-(K)$ inherits a homological grading from the Maslov grading on $\CFK^\infty(S^3,K)$. One then defines
\[
V_0(K)=-\frac{1}{2} d(A_0^-(K)),\qquad \underline{V}_0(K)=-\frac{1}{2}\underline{d}(A_0^-(K),\iota_0),\qquad \text{and} \qquad \bar{V}_0(K)=-\frac{1}{2} \bar{d}(A_0^-(K),\iota_0).
\]

\subsection{A simple example: $T_{r}\# T_{r}$} 
\label{sec:simpleexamples}

We consider the connected sum of the right handed trefoil with itself. The complexes $\cCFL^\infty(T_r)$ and $\CFK^\infty(T_r)$, as well as the maps $\Phi$ and $\Psi$ are shown in Figure \ref{fig::22}. In \cite{HMInvolutive}*{Section~7} it is shown that the knot involution $\iota_K$ is homotopic to reflection across the dashed line.  It's easily computed that $\bar{V}_0(T_r)=V_0(T_r)=\underline{V}_0(T_r)=1$.

\begin{figure}[ht!]
	\centering
	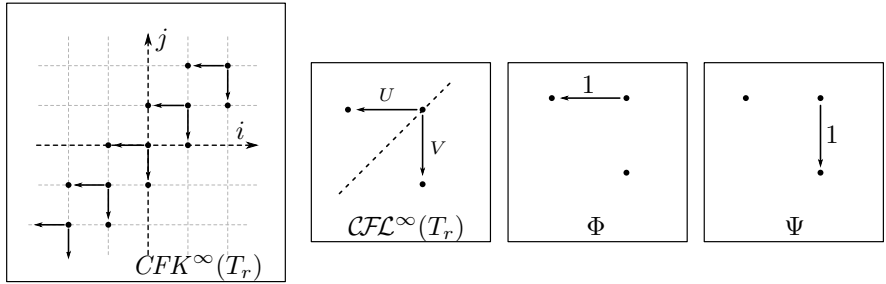
	\caption{\textbf{The complexes for $T_r$, and the maps $\Phi$ and $\Psi$.} On the left is the standard picture for $\CFK^\infty(T_r)$ for the right handed trefoil, drawn as a complex over $\bF_2$. The right three boxes show the complex $\cCFL^\infty(T_{r})$, a module over $\bF_2[U,V,U^{-1},V^{-1}]$, as well as the maps $\Phi$ and $\Psi$. Reflection across the dashed line yields $\iota_K$. Note that $\Phi$ and $\Psi$ individually do not define maps on $\CFK^\infty(T_r)$ since they shift Alexander grading.\label{fig::22}}
\end{figure}

 In Figure \ref{fig::23} we see the conjugation action on $\cCFL^\infty(T_{r}\# T_{r})$. Note that $\iota_K$ on the tensor product complex appears somewhat asymmetric: the extra arrow on the right could instead point to the other generator nearby. This reflects the two formulas for $\iota_{T_r\#T_r}$ from Theorem \ref{thm:B}. The arrow shown in Figure \ref{fig::22} corresponds to a summand of $\iota_1\Phi_1|\iota_2\Psi_2$. If we replaced that summand by $\iota_1\Psi_1|\iota_2\Phi_2$, we'd replace the arrow with one pointing to the nearby generator. Note that the two $\iota_K$-complexes are homotopy equivalent, by Lemma \ref{lem:productsarestronglyequivalent}.

In computing the invariants $\underline{d}$ and $\overline{d}$, the following lemma is convenient. It is modified slightly from \cite{HMZConnectedSum}, because of our grading convention:

\begin{lem}[\cite{HMZConnectedSum}*{Lemma 2.12}]\label{lem:dupperdlowercriteria} Suppose $\iota$ is a homotopy involution on a $\Z$-graded complex $(C,\d)$ which is freely generated over $\bF_2[U]$.
\begin{enumerate}[(a)]
\item\label{dupperlower-a} The quantity $\underline{d}(C,\iota)$ is the maximum grading of a homogeneous $v\in C$ such that $\d v=0$, $[U^n v]\neq 0$ for all $n\ge 0$ and there exists a $w\in C$ such that $\d w=(\id+\iota)v$.
\item\label{dupperlower-b} The quantity $\bar{d}(C,\iota)$ can be computed by considering triples $(x,y,z)$ of homogeneous elements in $C$ with at least one of $x$ and $y$ nonzero such that $\d y=(\id+\iota) x, \d z=U^m x$ for some $m\ge 0$ and $[U^n(U^m y+(\id+\iota)z)]\neq 0$ for all $n\ge 0$. The quantity $\overline{d}(C,\iota)$ is equal to the maximum of $\gr(x)+1$ (ranging over triples with $x\neq 0$) and $\gr(y)$ (ranging over triples where $x=0$). 
\end{enumerate}
\end{lem}

We now compute the involutive invariants of $T_r\# T_r$. We note that $T_r\# T_r$ is alternating, and hence by \cite{HMInvolutive}*{Proposition~8.1} the map $\iota_K$ is uniquely determined by its grading and filtration properties, and by the relation $\iota_K^2\simeq \id+\Phi\circ\Psi$. Indeed one can use the computation from \cite{HMInvolutive}*{Theorem~1.7} to compute the invariants for $T_r\# T_r$, since it is alternating. Nonetheless, we include this example as a demonstration of the formula in a simple case.

\begin{prop} The invariants for $T_r\# T_r$ from large surgeries are 
\[
\bar{V}_0(T_r\# T_r)=V_0(T_r\# T_r)=1\qquad \text{and}\qquad \underline{V}_0(T_r\# T_r)= 2.
\]
\end{prop}

\begin{proof}The complex $\cCFL^\infty(T_r\# T_r)$ and the involution $\iota_{T_r\# T_r}$ are shown in Figure \ref{fig::23}. The complex $A_0^-(T_r\# T_r)$ is shown in Figure \ref{fig::27}.
\begin{figure}[ht!]
	\centering
	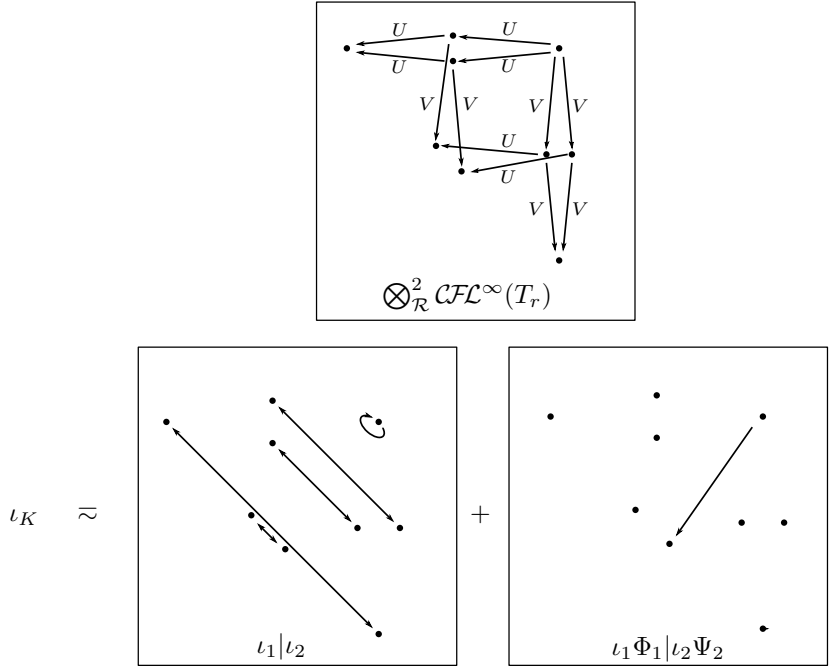
	\caption{\textbf{The complex $\cCFL^\infty(T_{r}\# T_{r})$, as well the conjugation map $\iota_K\eqsim \iota_1|\iota_2+\iota_1\Phi_1|\iota_2\Psi_2$.}\label{fig::23}}
\end{figure}

\begin{figure}[ht!]
	\centering
	\scriptsize{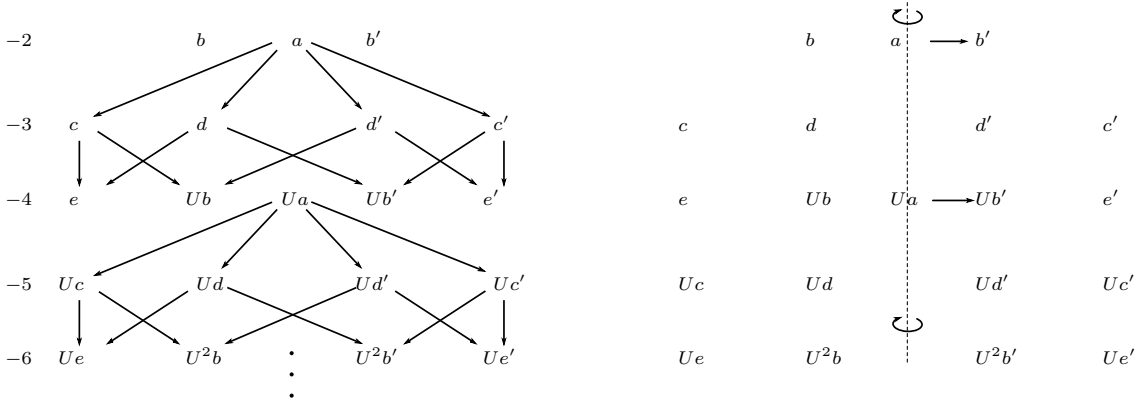}
	\caption{\textbf{The complex $A_0^-(T_r\# T_r)$ on the left and the involution $\iota_0$, on the right.} The involution on the right is equal to reflection across the dashed line plus the arrows.\label{fig::27}}
\end{figure}

Note that in the complex $A_0^-$, clearly $[b]$ is an element such that $U^n[b]\neq 0$, and $[b]$ has maximal grading in the complex. Hence
\[
V_0=-\frac{1}{2}d(A_0^-)=-\frac{1}{2}\gr([b])=1.
\]

To compute $\bar{V}_0=-\tfrac{1}{2}\bar{d}_0(AI_0^-)$, we note that $(0,b,0)$ is a triple as in part \eqref{dupperlower-b} of Lemma \ref{lem:dupperdlowercriteria}, so $\bar{d}(AI^-_0)\ge -2$. For this not to be an equality, we need there to be a triple $(x,y,z)$ satisfying condition \eqref{dupperlower-b} and with $\gr(x)\ge -2$ and $x\neq 0$.  This implies that $\gr(x)=-2$, as that is the top grading of elements in the chain complex. Similarly, grading considerations imply that such a triple must have $y=0$ and $x=\iota (x)$. There is only one nonzero element of grading $-2$ with $x=\iota( x)$, namely $x=b+b'$ (note that $a$ has grading $-2$ but is not preserved by $\iota$). The elements $z$ such that $\d z= U^m(b+b')$ are all of the form  $z=U^{m-1}(c+d)$ or $z=U^{m-1}(d'+c')$. But
\[
[U^nU^{m-1}(\id+\iota)(c+d)]=[U^{n+m-1}(c+d+d'+c')]=[0]
\] 
so there are no triples $(x,y,z)$ satisfying the criteria of part \eqref{dupperlower-b} of Lemma \ref{lem:dupperdlowercriteria} with $x=b+b'$ and $z=U^{m-1}(c+d)$. The same argument rules out triples with $x=b+b'$ and $z=U^{m-1}(c'+d')$, so we conclude that
\[
\bar{V}_0=-\frac{1}{2} \gr([b])=1.
\]

Finally we consider $\underline{V}_0=-\tfrac{1}{2}\underline{d}(AI^-_0)$ and part \eqref{dupperlower-a} of Lemma \ref{lem:dupperdlowercriteria}. Note that $[b]$ and $[b']$ are the only elements $x$ in grading $-2$ such that $[U^n x]\neq 0$. Since neither is fixed by the involution and there are no nonzero boundaries of degree $-2$, we see that neither can satisfy part \eqref{dupperlower-a} of Lemma \ref{lem:dupperdlowercriteria}. Hence $\underline{d}<-2$, and since it is an even integer, we must have $\underline{d}\le -4$. On the other hand, $x=Ub$ satisfies $U^n[Ub]\neq 0$ for all $n\ge 0$ and $\d(d'+c')=Ub+Ub'=(\id+\iota)(Ub)$, so $\underline{d}=-4$ and 
\[
\underline{V}_0=2.
\]
\end{proof}

\subsection{Further computer based computations}
The procedure for computing $V_0,$ $\bar{V}_0$ and $\underline{V}_0$ is amenable to computer computations when $\CFK^\infty(K)$ and the involution $\iota_K$ are known. For knots with relatively simple $\CFK^\infty(K)$ complex, such as thin knots, $L$-space knots and mirrors of $L$-space knots,  it is shown in \cite{HMInvolutive} that the involution $\iota_K$ is determined up to chain homotopy by the fact that $\iota_K^2\simeq \id+\Phi\circ  \Psi$.  Using the formula from Theorem \ref{thm:B}, we can compute the involution for any connected sum of such knots. In this section we provide some example computations using the assistance of \textit{Macaulay2} \cite{Mac2}.  The code can be found at the authors website.

Using our computations, we can prove the following:

\begin{customprop}{\ref{prop:applicationsofcomp}}None of the knots in Figure \ref{fig::table1}  is concordant to a thin knot, an $L$-space knot, or the mirror of an $L$-space knot.
\end{customprop}

\begin{proof}According to \cite{HMInvolutive}*{Proposition 8.2}, if $K$ is an thin knot, an $L$-space knot, or the mirror of an $L$-space knot, then the triple $(\bar{V}_0,V_0,\underline{V}_0)$ satisfies one of the two following conditions:

\begin{enumerate}\item All three of $\bar{V}_0,$ $V_0,$ and $\underline{V}_0$ are nonnegative and $0\le \underline{V}_0-\bar{V}_0\le 1$;
\item $\bar{V}_0\le 0$ while $V_0=\underline{V}_0=0$.
\end{enumerate}

The complexes for $T_{2,5},$ $T_{3,4},$ $T_{4,5},$ $T_{5,6}$ and $T_{6,7}$ are staircase complexes and are shown in Figure \ref{fig::31}. As shown in \cite{HMInvolutive}*{Section~7}, the involution is uniquely determined on these complexes, and is reflection across the diagonal in the $(i,j)$ plane. Using \textit{Macaulay2}, we can compute the involutive complexes associated to the connected sums. The relevant correction terms from large surgeries are shown in Figure \ref{fig::table1}. Note that for the knots shown in that figure, none of them have the pattern listed above for the $\bar{V}_0,$ $V_0$ and $\underline{V}_0$ invariants of thin knots, $L$-space knots or mirrors of $L$-space knots.
\end{proof}

\begin{figure}[ht!]
	\centering
\begin{tabular}{| c | c | c | c |}
\hline
  Knot & $\bar{V}_0$ &$V_0$ & $\underline{V}_0$ \\ 
 \hline  $T_{4,5}\# T_{4,5}$&$4$ &$4$ &$6 $\\
 \hline  $T_{4,5}\# T_{4,5}\# T_{5,6}$&$7$ &$7$ &$9 $\\
 \hline $T_{6,7}\# T_{6,7}$ &$9$ &$9$ &$12 $\\
  \hline $T_{4,5}\# T_{6,7}$ &$7$ &$7$ &$9 $\\
    \hline $T_{3,4}^{-1}\# T_{4,5}^{-1}\# T_{5,6}$ &$-1$ &$1$ &$1 $\\
        \hline

\end{tabular}
	\caption{\textbf{Some knots and their involutive invariants of large surgeries.}\label{fig::table1}}
\end{figure}

Some caution is warranted making conclusions about connected sums of torus knots from these computations. For example $T_{5,6}\# T_{5,6}$ has $(\bar{V}_0,$ $V_0,\underline{V}_0)$ equal to $(6,6,6)$.

\begin{figure}[ht!]
	\centering
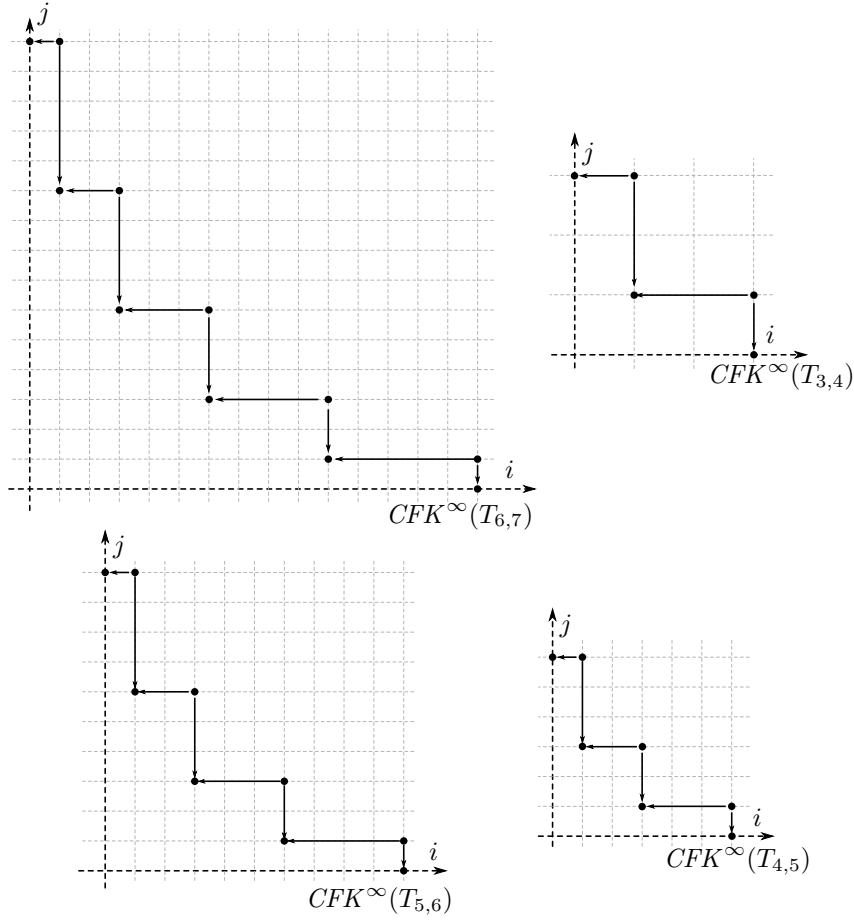
	\caption{\textbf{Model complexes for $T_{3,4},$ $T_{4,5},$ $T_{5,6}$ and $T_{6,7}$.} The $\CFK^\infty$ complexes are formed by tensoring with $\bF_2[U,U^{-1}]$.\label{fig::31}}
\end{figure}

\begin{question}Are there knots in $S^3$ where all three $\bar{V}_0,$ $V_0$ and $\underline{V}_0$ are all different?
\end{question}

The above question seems likely to have an affirmative answer, in light of the examples from \cite{HMZConnectedSum}.

\appendix 
\section{Holomorphic proof of the 4-dimensional surgery relations}
\label{app:directverification}

In this section, we give a direct verification of Lemma \ref{lem:4manifoldsurgery2}, a  key step in the proof of Proposition \ref{prop:surgeryonclosedcurve}. The proof we gave followed from formal properties of the link Floer TQFT. We have included this proof in the appendix since it gives a different perspective of the maps, but involves some subtle counts of holomorphic curves.

\begin{customlem}{\ref{lem:4manifoldsurgery2}}For the decorated surfaces with divides $M_1$ and $M_2$ in Section \ref{subsec:4dimsurg1}, the two maps
\[
F_{S^1\times D^3, M_1,\frt},\quad F_{D^2\times S^2, M_2,\frt_0}\colon \cCFL^\infty(\varnothing,\varnothing)\to \cCFL^\infty(S^1\times S^2,\bL_0)
\] 
are filtered chain homotopy equivalent.
\end{customlem}
\begin{proof}

First, note that by grading considerations, as in the argument from Section \ref{subsec:4dimsurg1}, it is sufficient to restrict to $\hat{\CFL}$, the version obtained from $\cCFL^-$ by formally setting $U=V=0$. On $\hat{\CFL}$, the maps $\Phi_p$ and $\Psi_q$ are defined by counting disks going over $p$ or $q$ exactly once.

To organize the proof, we list the main steps in the following sublemma:

\begin{sublem}\label{sublemma:modelcomp} In terms of the intersection points shown in Figure \ref{fig::40}, the following hold:
\begin{enumerate}
\item \label{subclaim:comp-1}One has
\[
F_{S^1\times D^3,M_3,\frs}(1)=\theta_1^{\ws}\xi_2^{\ws}+(C_2+C_3)\cdot \xi_1^{\ws}\theta_2^{\ws}
\]
 for two constants $C_2,C_3\in \bF_2$;
\item\label{subclaim:comp-2} One has 
\[
F_{D^2\times S^2,M_2,\frs_0}(1)=\theta_1^{\ws}\xi_2^{\ws}+(C_2'+C_3')\cdot \xi_1^{\ws}\theta_2^{\ws}
\]
 for two constants $C_2',C_3'\in \bF_2$;
\item\label{subclaim:comp-3} For any choice of almost complex structure, the constants $C_2,$ $C_3,$ $C_2',$ and $C_3'$ satisfy 
\[
C_2+C_3=C_2'+C_3'.
\]
\end{enumerate}
\end{sublem}

Once we complete the proof of the above sublemma, it will follow immediately that $F_{S^1\times D^3,M_3,\frs}=F_{D^2\times S^2,M_2,\frs_0}$. Below, we prove each part of the sublemma  separately.
\end{proof}

\begin{proof}[Proof of Part \eqref{subclaim:comp-1} of Sublemma \ref{sublemma:modelcomp}]
We now compute the map $F_{S^1\times D^3,M_1,\frs}$. The four dimensional cobordism is simply obtained by adding a 0-handle and a 1-handle. A decomposition of the decorated surface $M_1$ is shown in Figure \ref{fig::44}. The decorated surface $M_1$ is constructed by adding a disk (inside the 4-dimensional 0-handle), then performing a $T^+_{p_1,q_1}$ quasi-stabilization, then attaching a type-$\ve{z}$ band, which goes around the 1-handle once, then attaching a cylindrical link cobordism with dividing set corresponding to the composition $\Phi_{p_1}\circ \Psi_{q_1}$ for $p_1$ and $q_1$ both on the same components of $\bL_0$. 

\begin{figure}[ht!]
\centering
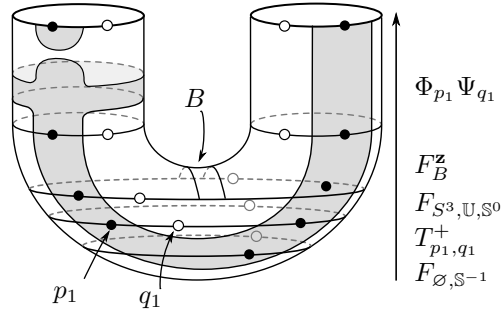
\caption{\textbf{A decomposition of the link cobordism $(S^1\times D^3, M_1)$.} Here we have drawn only the surface (and not the 4-manifold $S^1\times D^3$, obtained by adding a 0-handle and a 1-handle). Also $F_{S^3,\bU,\bS^0}$ and $ F_{\varnothing, \bS^{-1}}$ denote the 1-handle and 0-handle maps, respectively. \label{fig::44}}
\end{figure}

The 0-handle map $F_{\varnothing,\bS^{-1}}$ is the tautological map from $\bF_2$ to $\hat{\CFL}(S^3,\bU)\iso \bF_2$, and the composition $F_{S^3, \bU,\bS^0}\circ T_{p_1,q_1}^+$ of the 1-handle map and the quasi-stabilization map  sends $F_{\varnothing,\bS^{-1}}(1)$ to $\Theta_{\as,\bs}^{\ve{z}}$, the top $\gr_{\ve{z}}$-graded intersection point. The intersection point $\Theta_{\as,\bs}^{\ve{z}}$ can be seen in the subdiagram $(\Sigma,\as,\bs)$ of the Heegaard triple shown at the top of Figure \ref{fig::45}. The map $F_B^{\ve{z}}$ can be computed by using the Heegaard triple  $(\Sigma,\as',\as,\bs)$ from Figure \ref{fig::45}, using the formula
\[
F_B^{\ve{z}}(\ve{x})=F_{\as',\as,\bs}(\Theta^{\ve{w}}_{\as',\as},\ve{x}),
\]
where $\Theta^{\ws}_{\as',\as}$ is the top $\gr_{\ws}$-graded intersection point.

 It is straightforward to verify that, ranging over all intersection points $\ve{y}$, there is only one homology class in any $\pi_2(\Theta^{\ve{w}}_{\as',\as},\Theta^{\ve{z}}_{\as,\bs},\ve{y})$ which is nonnegative and has Maslov index zero. This class is shown in Figure \ref{fig::45} and has $\ve{y}=\xi_1^{\ws}\theta_2^{\ws}$. It has a unique holomorphic representative by the Riemann mapping theorem. Hence $F_{B}^{\ve{z}}(\Theta_{\as,\bs}^{\ve{z}})=\xi_1^{\ws}\theta_2^{\ws}$.

\begin{figure}[ht!]
\centering
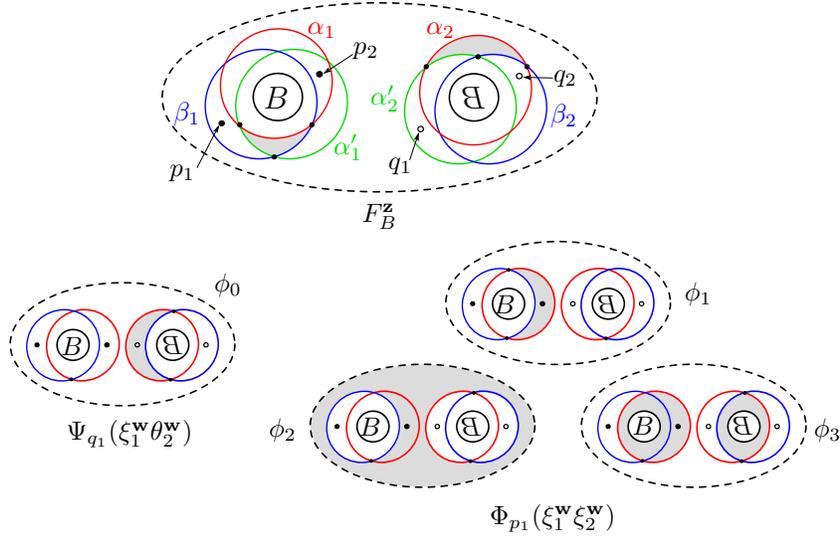
\caption{\textbf{Computing some terms of the maps $F_B^{\ve{z}}$ and $\Psi_{q_1}$ and $\Phi_{p_1}$.} On the top is a Heegaard triple $(\Sigma,\as',\as,\bs)$ which can be used to compute  $F_{B}^{\ve{z}}$. On the bottom left is the only Maslov index zero homology class of disks which contribute to $\Psi_{q_1}(\xi_1^{\ws}\theta_2^{\ws})$. On the bottom right are the three homology classes which contribute to  $\Phi_{p_1}\Psi_{q_1}(\xi_1^{\ws}\theta_2^{\ws})$\label{fig::45}}
\end{figure}

It remains to compute $\Phi_{p_1}\Psi_{q_1}(\xi_1^{\ws}\theta_2^{\ws})$. By inspection, we see that there is exactly one nonnegative Maslov index one homology class $\phi_0$ which contributes to  $\Psi_{q_1}(\xi_1^{\ws}\theta_2^{\ws})$ on $\hat{\CFL}(\Sigma,\as',\bs)$. This is shown in Figure \ref{fig::45}. Since $\# \hat{\cM}(\phi_0)=1$, have 
\[
\Psi_{q_1}(\xi_1^{\ws}\theta_2^{\ws})=\xi_1^{\ws}\xi_2^{\ws}.
\] 
 On the other hand, there are three nonnegative Maslov index one homology classes which potentially contribute to $\Phi_{p_1}(\xi_1^{\ws}\xi_2^{\ws})$. By the Riemann mapping theorem $\#\hat{\cM}(\phi_1)=1$. Write $C_2=\# \hat{\cM}(\phi_2)$ and $C_3=\# \hat{\cM}(\phi_3)$. Then
\[
\Phi_{p_1}(\xi_1^{\ws}\xi_2^{\ws})=\theta_1^{\ws}\xi_2^{\ws}+(C_2+C_3)\cdot \xi_1^{\ws}\theta_2^{\ws},
\] 
and hence,
\begin{equation}F_{S^1\times D^3,M_1,\frs}(1)=\theta_1^{\ws}\xi_2^{\ws}+(C_2+C_3)\cdot \xi_1^{\ws}\theta_2^{\ws}.\label{eq:FM1=}
\end{equation}
\end{proof}

\begin{proof}[Proof of Part \eqref{subclaim:comp-2} of Sublemma \ref{sublemma:modelcomp} ] The map $F_{D^2\times S^2, M_2,\frt_0}$ can be computed by starting with a 0-handle map, which sends $1$ to the generator of $\hat{\CFL}(S^3,\bU)$ (where $\bU$ is a doubly pointed unknot). Then one adds an additional unknotted component $\bU'$ using the ``birth cobordism map'' from \cite{ZemCFLTQFT}*{Section~7.1}. We will use the diagram $(\Sigma,\ds, \bs, \ps,\qs)$ for $(S^3,\bU\cup \bU')$ shown in Figure~\ref{fig::52}.

\begin{figure}[ht!]
\centering
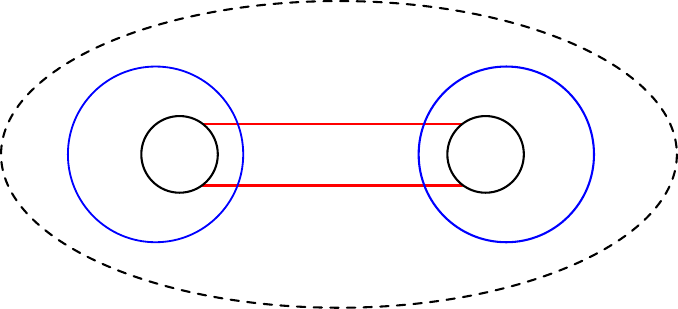
\caption{\textbf{The diagram $(\Sigma, \ds,\bs,\ps,\qs)$ for $(S^3, \bU\cup \bU')$.} The homology $\hat{\CFL}(\Sigma,\ds,\bs,\ps,\qs)$ is rank 2 over $\bF_2$, and the top degree intersection point $\Theta^+_{\ds,\bs}$ is marked.\label{fig::52}}
\end{figure}

 The composition of the 0-handle map and the birth cobordism maps sends the generator $1$ to $\Theta^+_{\ds,\bs}$, the top graded element of $\hat{\CFL}(S^3,\bU\cup \bU')$ (the gradings $\gr_{\ve{w}}$ and $\gr_{\ve{z}}$ coincide for $\hat{\CFL}(S^3,\bU\cup \bU')$). Finally, this is composed with a 2-handle map for a surgery on a 0-framed unknot which has linking number $+1$ with $\bU$ and $-1$ with $\bU'$.  The 2-handle map is computed by counting holomorphic triangles in the Heegaard triple $(\Sigma,\as',\ds,\bs,\ve{p},\ve{q})$  shown in Figure~\ref{fig::43}. Shown also in  Figure~\ref{fig::43} are three homology classes of triangles. One can easily check that these are the only homology classes with Maslov index zero in any $\pi_2(\Theta^{+}_{\as',\ds},\Theta_{\ds,\bs}^+,\ve{y})$ which have only nonnegative multiplicities.

Note that the maps 
\[
\frs_{\ve{w}},\, \frs_{\ve{z}}\colon\pi_2(\Theta^+_{\as',\ds},\Theta^+_{\ds,\bs},\ve{y})\to \Spin^c(D^2\times S^2)
\] 
are equal, since $\Spin^c$ structures on $D^2\times S^2$ are uniquely determined by their restriction to $S^1\times S^2$ and $\frs_{\ve{w}}(\ve{y})=\frs_{\ve{z}}(\ve{y})$ for an intersection point $\ve{y}$ on $(\Sigma,\as',\bs)$ by Lemma \ref{lem:changeSpincstructure} since $\bL_0$ is null-homologous. Furthermore, all three classes in Figure \ref{fig::43} represent $\frt_0$, since they restrict to the torsion $\Spin^c$ structure on $S^1\times S^2$.

\begin{figure}[ht!]
\centering
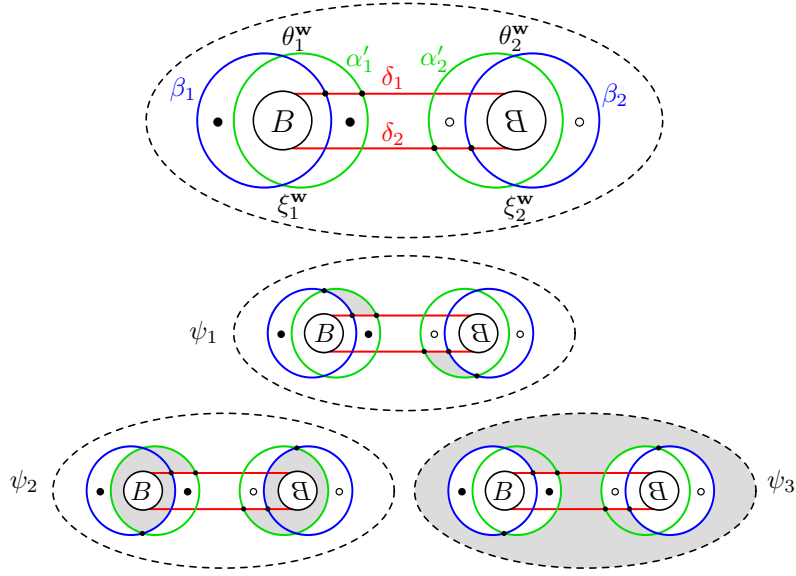
\caption{\textbf{Computing the 2-handle map.} On top is the Heegaard triple for computing the 2-handle map for the cobordism $D^2\times S^2$. The intersection points $\Theta_{\as',\ds}^+$ and $\Theta_{\ds,\bs}^+$ are labeled with dots, in the top diagram. On the bottom are the three nonnegative Maslov index zero homology classes in $\pi_2(\Theta_{\as',\ds}^+,\Theta^+_{\ds,\bs},\ve{x})$ for some $\ve{x}$.\label{fig::43}}
\end{figure}

The homology class $\psi_1$ is in $\pi_2(\Theta_{\as',\ds}^{+},\Theta_{\ds,\bs}^+,\theta^{\ws}_1\xi_2^{\ws})$, while $\psi_2$ and $\psi_3$ are in $\pi_2(\Theta_{\as',\ds}^{+},\Theta_{\ds,\bs}^+,\xi_1^{\ws}\theta_2^{\ws})$. The homology class $\psi_1$ consists of only small triangles, and has a unique holomorphic representative by the Riemann mapping theorem. The homology classes $\psi_2$ and $\psi_3$ each consist of annular domains.  Write $C_2'=\# \cM(\psi_2)$ and $C_3'=\# \cM(\psi_3)$.

We conclude
\begin{equation}F_{D^2\times S^2,M_2,\frs_0}(1)=\theta_1^{\ws}\xi_2^{\ws}+(C_2'+C_3')\cdot \xi_1^{\ws}\theta_2^{\ws}.\label{eq:FM2=}\end{equation}
\end{proof}

\begin{proof}[Proof of Part \eqref{subclaim:comp-3} of Sublemma \ref{sublemma:modelcomp} ] We now show that
\[
C_2+C_3=C_2'+C_3'.
\]
One could hope to find precise values for $C_2,$ $C_3,$ $C_2'$ and $C_3'$ (which we defined to be the number of representatives of $\phi_2,$ $\phi_3,$ $\psi_2$ and $\psi_3$, respectively) by using conformal mapping techniques, however this is challenging to make precise. Instead, we will use a (somewhat indirect) Gromov compactness argument to show that $C_2+C_3=C_2'+C_3',$ for any choice of almost complex structure.

We will prove the following subclaims:
\begin{enumerate}[\hspace{22mm}]
\item[\textbf{Claim (1):}] The map $\Phi_{p_1}$ vanishes on both $\hat{\CFL}(\Sigma,\as',\ds)$ and $\hat{\CFL}(\Sigma,\ds,\bs)$.
\item[\textbf{Claim (2):}] $\hat{\d}_{\as',\ds},$ $ \hat{\d}_{\ds,\bs}$ and $\hat{\d}_{\as',\bs}$ all vanish on $\hat{\CFL}$.
\item[\textbf{Claim (3):}] There is a filtered, equivariant map $H_{p_1}\colon \cCFL^-(\Sigma,\as',\ds)\otimes \cCFL^-(\Sigma,\ds,\bs)\to \cCFL^-(\Sigma,\as',\bs)$ such that
\[
F_{\as',\ds,\bs,\frt_0}(\Phi_{p_1}|1+1|\Phi_{p_1})+\Phi_{p_1} F_{\as',\ds,\bs,\frt_0}+H_{p_1}(\d_{\as',\ds}|1+1|\d_{\ds,\bs})+\d_{\as',\bs} H_{p_1}=0.
\]
 Furthermore, this holds on the curved version of the link Floer complexes over the ring $\bF_2[U_{p_1}, U_{p_2}, V_{q_1}, V_{q_2}]$.
\item[\textbf{Claim (4):}] As maps on $\hat{\CFL}(\Sigma,\ve{\alpha}',\ve{\beta})$, one has 
\begin{equation}
\Phi_{p_1}(\xi_1^{\ws}\theta_2^{\ws})=\theta_1^{\ws}\theta_1^{\ws} \qquad \text{and}\qquad  \Phi_{p_1}(\theta_1^{\ws}\xi_2^{\ws})=(C_1''+C_2'')\cdot  \theta_1^{\ws}\theta_2^{\ws},\label{eq:valuesofPhi}
\end{equation}
where $C_1''=\# \hat{\cM}(\phi_2'')$ and $C_2''=\# \hat{\cM}(\phi_3'')$ for the classes $\phi_2'',\phi_3''\in \pi_2(\theta_1^{\ws}\xi_2^{\ws},\theta_1^{\ws}\theta_2^{\ws})$ shown in Figure~\ref{fig::47}.
\end{enumerate}

\begin{figure}[ht!]
\centering
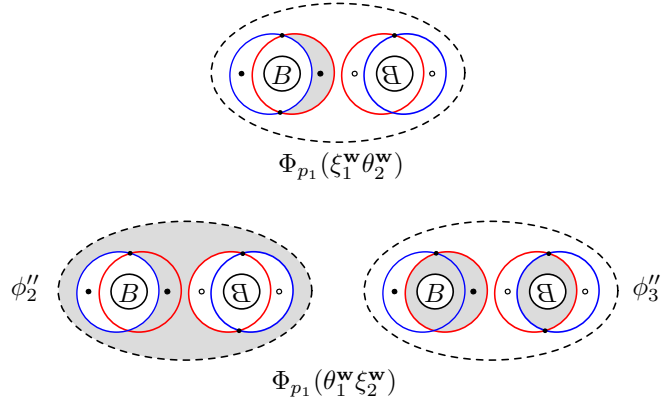
\caption{\textbf{The classes counted by $\Phi_{p_1}(\xi_1^{\ws}\theta_2^{\ws})$ ad $\Phi_{p_1}(\theta_1^{\ws}\xi_2^{\ws})$.}\label{fig::47}}
\end{figure}

We now verify Claims (1)--(4).

Claim (1) is obvious, since on $(\Sigma,\as',\ds)$ and $(\Sigma,\ds,\bs)$ the region with $p_1$ also has $q_1$ in it (so no disks are counted by $\Phi_{p_1}$ on the hat flavor).

Claim (2) is also straightforward. On $(\Sigma,\ve{\alpha}',\ve{\beta})$ there are no nonnegative, Maslov index 1 domains which have zero multiplicity over all of the basepoints. Similarly, on each of the diagrams $(\Sigma,\as',\ds)$ and $(\Sigma,\ds,\bs)$, there are exactly two nonnegative, Maslov index 1 classes which do not go over any basepoints. Each has domain equal to a rectangle, so each class has a unique representative and modulo two, and hence the total count between the two classes cancels, modulo two.

Claim (3) follows by noting that Gromov compactness yields that
 \[
 F_{\as',\ds,\bs,\frt_0}(\d_{\as',\ds}|1+1|\d_{\ds,\bs})+\d_{\as',\bs}F_{\as',\ds,\bs,\frt_0}=0,
 \] 
 on $\cCFL^-$, before setting any variables equal to each other. One then differentiates the above expression with respect to $U_{p_1}$, and then sets all of the $U_{p_i}$ and $V_{q_i}$ variables to zero. Claim (3) follows.

Claim (4) follows from direct examination of the diagram in Figure \ref{fig::47}, where we show all the index 1 classes that can have representatives and can contribute to $\Phi_{p_1}(\xi_1^{\ws}\theta_2^{\ws})$ or $\Phi_{p_1}(\theta_1^{\ws}\xi_2^{\ws})$.

Having proven Claims (1)--(4), we proceed with the proof that $C_2+C_3=C_2'+C_3'$. Combining Claims (1), (2) and (3), we conclude that
\[
\Phi_{p_1}( F_{\as',\ds,\bs,\frt_0}(\Theta^+_{\as',\ds},\Theta^+_{\ds,\bs}))=0.
\]
 In particular, 
\begin{equation}\Phi_{p_1}(F_{D^2\times S^2,M_2,\frt_0}(1))=0.\label{eq:phivanishesonimage}\end{equation} By Part \eqref{subclaim:comp-2} of Sublemma \ref{sublemma:modelcomp}, we have
\[
F_{D^2\times S^2,M_2,\frt_0}(1)= \theta_1^{\ws}\xi_2^{\ws}+(C_2'+C_3')\cdot \xi_1^{\ws}\theta_1^{\ws}.
\]
 Composing with $\Phi_{p_1}$ and using Equations \eqref{eq:valuesofPhi} and \eqref{eq:phivanishesonimage}, we see that
\[
0=\Phi_{p_1}(F_{D^2\times S^2,M_2,\frt_0}(1))=(C_2''+C_3''+C_2'+C_3')\cdot\theta_1^{\ws}\theta_1^{\ws}=0.
\] 
Thus we conclude that
\begin{equation}C_2'+C_3'=C_2''+C_3''.\label{eq:C'C''equal}\end{equation}

 The final step is to show that $C_2=C_2''$ and $C_3=C_3''$, which combined  \eqref{eq:C'C''equal} will finish the proof of Part \eqref{subclaim:comp-3} of Sublemma \ref{sublemma:modelcomp}. To this end, we will show that $C_3=C_3''$; the argument that $C_2=C_2''$ is nearly identical. This will follow from a holomorphic curve chasing argument (i.e. by using Gromov compactness). Recall that $C_3=\# \hat{\cM}(\phi_3)$ and $C_3''=\hat{\cM}(\phi_3'')$. Also note that $\phi_3$ and $\phi_3''$ have the same domain, but they represent different homology classes since they have different incoming and outgoing intersection points. We count the ends of two Maslov index 2 homology classes. This is shown in Figure \ref{fig::46}. Note that the domains are the same, though the homology classes have different choices of incoming and outgoing intersection points. Here $\#\hat{\cN}^\alpha(A)$ denotes the count of holomorphic $\as$-boundary degenerations in a class $A\in \pi_2^\alpha(\ve{x})$ (see \cite{OSLinks}*{Section~5.3}). By \cite{OSLinks}*{Theorem~5.5}, the count of each of the two spaces of boundary degenerations is 1 (modulo 2). Each the two 1-dimensional moduli spaces in Figure \ref{fig::46} has boundary components corresponding to a single class of $\as$-boundary degenerations, which hence each contribute equally to the boundaries of the two 1-dimensional moduli spaces.

Thus the left columns of Figure \ref{fig::46} have moduli spaces contributing $1$ to the total numbers of ends. The ends in the middle column feature the same pair of holomorphic disks resulting from strip breaking, just in reversed order. Hence the contributions from the middle columns are equal. In the right column, the end on the top of the Figure contributes $C_3''=\#\hat{ \cM}(\phi_3'')$. On the bottom, the contribution is $C_3=\# \hat{\cM}(\phi_3)$. Comparing all of these contributions, we conclude that $C_3=C_3''$, concluding the proof.
\end{proof}

\begin{figure}[ht!]
\centering
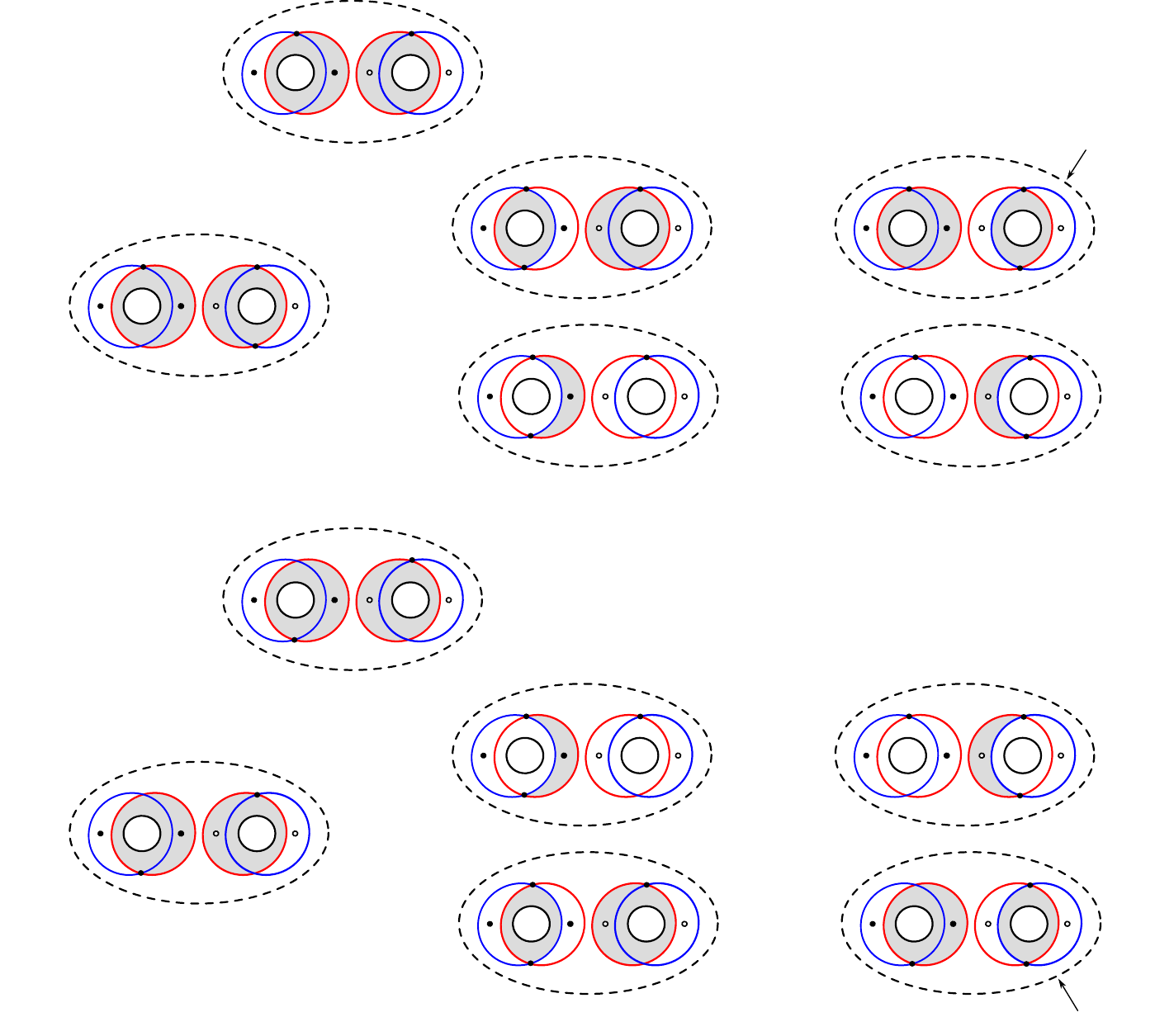
\caption{\textbf{Showing $C_3=C_3''$ by counting the ends of two 1-dimensional moduli spaces.} There are ends corresponding to $\as$ boundary degenerations, as well as strip breaking. \label{fig::46}}
\end{figure}

\bibliographystyle{custom} 
\bibliography{biblio}

\begin{bibdiv}
\begin{biblist}

\bib{BHRationalCuspidal}{article}{
      author={Borodzik, Maciej},
      author={Hom, Jennifer},
      author={Schinzel, Andrzej},
       title={Involutive {H}eegaard {F}loer homology and rational cuspidal
  curves},
        date={2018},
     journal={Proceedings of the London Mathematical Society},
}

\bib{CGHHF=ECH0}{unpublished}{
      author={Colin, Vincent},
      author={Ghiggini, Paolo},
      author={Honda, Ko},
       title={Embedded contact homology and open book decompositions},
        date={2010},
        note={e-print, \url{arXiv:1008.2734}},
}

\bib{CGHHF=ECH1}{unpublished}{
      author={Colin, Vincent},
      author={Ghiggini, Paolo},
      author={Honda, Ko},
       title={The equivalence of {H}eegaard {F}loer homology and embedded
  contact homology via open book decompositions {I}},
        date={2012},
        note={e-print ,\url{arXiv:1208.1074}},
}

\bib{CGHHF=ECH2}{unpublished}{
      author={Colin, Vincent},
      author={Ghiggini, Paolo},
      author={Honda, Ko},
       title={The equivalence of {H}eegaard {F}loer homology and embedded
  contact homology via open book decompositions {II}},
        date={2012},
        note={e-print, \url{arXiv:1208.1077}},
}

\bib{CGHHF=ECH3}{unpublished}{
      author={Colin, Vincent},
      author={Ghiggini, Paolo},
      author={Honda, Ko},
       title={The equivalence of {H}eegaard {F}loer homology and embedded
  contact homology via open book decompositions {III}: from hat to plus},
        date={2012},
        note={e-print, \url{arXiv:1208.1526}},
}

\bib{Mac2}{misc}{
      author={Grayson, Daniel~R.},
      author={Stillman, Michael~E.},
       title={Macaulay2, a software system for research in algebraic geometry},
         how={Available at \url{http://www.math.uiuc.edu/Macaulay2/}},
}

\bib{HKMTQFT}{unpublished}{
      author={Honda, Ko},
      author={Kazez, William},
      author={Mati\'c, Gordana},
       title={Contact structures, sutured {F}loer homology and {TQFT}},
        date={2008},
        note={e-print, \url{arXiv:0807.2431}},
}

\bib{HMInvolutive}{article}{
      author={Hendricks, Kristen},
      author={Manolescu, Ciprian},
       title={Involutive {H}eegaard {F}loer homology},
        date={2017},
     journal={Duke Math. J.},
      volume={166},
      number={7},
       pages={1211\ndash 1299},
}

\bib{HMZConnectedSum}{article}{
      author={Hendricks, Kristen},
      author={Manolescu, Ciprian},
      author={Zemke, Ian},
       title={A connected sum formula for involutive {H}eegaard {F}loer
  homology},
        date={2018},
        ISSN={1022-1824},
     journal={Selecta Math. (N.S.)},
      volume={24},
      number={2},
       pages={1183\ndash 1245},
      review={\MR{3782421}},
}

\bib{HomEpsilon}{article}{
      author={Hom, Jennifer},
       title={The knot {F}loer complex and the smooth concordance group},
        date={2014},
        ISSN={0010-2571},
     journal={Comment. Math. Helv.},
      volume={89},
      number={3},
       pages={537\ndash 570},
}

\bib{JTNaturality}{article}{
      author={Juh\'{a}sz, Andr\'{a}s},
      author={Thurston, Dylan},
      author={Zemke, Ian},
       title={Naturality and mapping class groups in {H}eegaard {F}loer
  homology},
     journal={Memoirs of the American Mathematical Society},
        note={To appear},
}

\bib{JCob}{article}{
      author={Juh\'asz, Andr\'as},
       title={Cobordisms of sutured manifolds and the functoriality of link
  {F}loer homology},
        date={2016},
        ISSN={0001-8708},
     journal={Adv. Math.},
      volume={299},
       pages={940\ndash 1038},
}

\bib{KLTHF=HM1}{unpublished}{
      author={Kutluhan, \c{C}a\u{g}atay},
      author={Lee, Yi-Jen},
      author={Taubes, Clifford~H.},
       title={{HF=HM I : H}eegaard {F}loer homology and {S}eiberg--{W}itten
  {F}loer homology},
        date={2010},
        note={e-print, arXiv:1007.1979},
}

\bib{KLTHF=HM2}{unpublished}{
      author={Kutluhan, \c{C}a\u{g}atay},
      author={Lee, Yi-Jen},
      author={Taubes, Clifford~H.},
       title={{HF=HM II}: {R}eeb orbits and holomorphic curves for the
  {ECH/H}eegaard-{F}loer correspondence},
        date={2010},
        note={e-print, arXiv:1008.1595},
}

\bib{KLTHF=HM3}{unpublished}{
      author={Kutluhan, \c{C}a\u{g}atay},
      author={Lee, Yi-Jen},
      author={Taubes, Clifford~H.},
       title={{HF=HM III:} holomorphic curves and the differential for the
  {ECH/H}eegaard {F}loer correspondence},
        date={2010},
        note={e-print, arXiv:1010.3456},
}

\bib{KLTHF=HM4}{unpublished}{
      author={Kutluhan, \c{C}a\u{g}atay},
      author={Lee, Yi-Jen},
      author={Taubes, Clifford~H.},
       title={{HF=HM IV: S}eiberg--{W}itten {F}loer homology and {ECH}
  correspondence},
        date={2011},
        note={e-print, arXiv:1107-2297},
}

\bib{KLTHF=HM5}{unpublished}{
      author={Kutluhan, \c{C}a\u{g}atay},
      author={Lee, Yi-Jen},
      author={Taubes, Clifford~H.},
       title={{HF=HM V: S}eiberg--{W}itten {F}loer homology and handle
  additions},
        date={2012},
        note={e-print, arXiv:1204.0115},
}

\bib{KMMonopole}{book}{
      author={Kronheimer, Peter},
      author={Mrowka, Tomasz},
       title={Monopoles and three-manifolds},
      series={New Mathematical Monographs},
   publisher={Cambridge University Press, Cambridge},
        date={2007},
      volume={10},
}

\bib{LinConnSums}{article}{
      author={Lin, Francesco},
       title={$\mathrm{Pin}(2)$-monopole {F}loer homology, higher compositions
  and connected sums},
        date={2016},
     journal={Journal of Topology},
      volume={10},
}

\bib{LinPin2Monopole}{article}{
      author={Lin, Francesco},
       title={A {M}orse-{B}ott approach to monopole {F}loer homology and the
  triangulation conjecture},
        date={2018},
     journal={Memoirs of the American Mathematical Society},
      volume={255},
       pages={1\ndash 174},
}

\bib{LipshitzCylindrical}{article}{
      author={Lipshitz, Robert},
       title={A cylindrical reformulation of {H}eegaard {F}loer homology},
        date={2006},
     journal={Geom. Topol.},
      volume={10},
       pages={955\ndash 1097},
}

\bib{ManPin2Triangle}{article}{
      author={Manolescu, Ciprian},
       title={Pin(2)-equivariant {S}eiberg-{W}itten {F}loer homology and the
  triangulation conjecture},
        date={2016},
     journal={J. Amer. Math. Soc.},
      volume={29},
      number={1},
       pages={147\ndash 176},
}

\bib{MOIntSurg}{unpublished}{
      author={Manolescu, Ciprian},
      author={Ozsv{\'a}th, Peter~S.},
       title={Heegaard {F}loer homology and integer surgeries on links},
        date={2010},
        note={e-print, \url{arXiv:1011.1317}},
}

\bib{OSIntersectionForms}{article}{
      author={Ozsv{\'a}th, Peter~S.},
      author={Szab{\'o}, Zolt{\'a}n},
       title={Absolutely graded {F}loer homologies and intersection forms for
  four-manifolds with boundary},
        date={2003},
     journal={Adv. Math.},
      volume={173},
      number={2},
       pages={179\ndash 261},
}

\bib{OSKnots}{article}{
      author={Ozsv\'ath, Peter},
      author={Szab\'o, Zolt\'an},
       title={Holomorphic disks and knot invariants},
        date={2004},
     journal={Adv. Math.},
      volume={186},
      number={1},
       pages={58\ndash 116},
}

\bib{OSDisks}{article}{
      author={Ozsv\'ath, Peter},
      author={Szab\'o, Zolt\'an},
       title={Holomorphic disks and topological invariants for closed
  three-manifolds},
        date={2004},
     journal={Ann. of Math. (2)},
      volume={159},
      number={3},
       pages={1027\ndash 1158},
}

\bib{OSProperties}{article}{
      author={Ozsv{\'a}th, Peter~S.},
      author={Szab{\'o}, Zolt{\'a}n},
       title={Holomorphic disks and three-manifold invariants: properties and
  applications},
        date={2004},
     journal={Ann. of Math. (2)},
      volume={159},
      number={3},
       pages={1159\ndash 1245},
}

\bib{OSTriangles}{article}{
      author={Ozsv\'ath, Peter},
      author={Szab\'o, Zolt\'an},
       title={Holomorphic triangles and invariants for smooth four-manifolds},
        date={2006},
     journal={Adv. Math.},
      volume={202},
      number={2},
       pages={326\ndash 400},
}

\bib{OSLinks}{article}{
      author={Ozsv\'ath, Peter},
      author={Szab\'o, Zolt\'an},
       title={Holomorphic disks, link invariants and the multi-variable
  {A}lexander polynomial},
        date={2008},
     journal={Algebr. Geom. Topol.},
      volume={8},
      number={2},
       pages={615\ndash 692},
}

\bib{OSSUpsilon}{article}{
      author={Ozsv\'ath, Peter~S.},
      author={Stipsicz, Andr\'as~I.},
      author={Szab\'o, Zolt\'an},
       title={Concordance homomorphisms from knot {F}loer homology},
        date={2017},
        ISSN={0001-8708},
     journal={Adv. Math.},
      volume={315},
       pages={366\ndash 426},
}

\bib{RasmussenKnots}{thesis}{
      author={Rasmussen, Jacob},
       title={{F}loer homology and knot complements},
        type={Ph.D. Thesis},
        date={2003},
        note={\url{arXiv:math/0306378}},
}

\bib{RasLSpaceSurgeries}{unpublished}{
      author={Rasmussen, Jacob~A.},
       title={Lens space surgeries and {L}-space homology spheres},
        date={2007},
        note={e-print, \url{arXiv:0710.2531}},
}

\bib{SarkarMovingBasepoints}{article}{
      author={Sarkar, Sucharit},
       title={Moving basepoints and the induced automorphisms of link {F}loer
  homology},
        date={2015},
     journal={Algebr. Geom. Topol.},
      volume={15},
      number={5},
       pages={2479\ndash 2515},
}

\bib{StoffregenSeifertFibered}{unpublished}{
      author={Stoffregen, Matthew},
       title={Pin(2)-equivariant {S}eiberg-\!{W}itten {F}loer homology of
  {S}eifert fibrations},
        date={2015},
        note={e-print, \url{arXiv:1505.03234}},
}

\bib{TaubesECH=SW1}{article}{
      author={Taubes, Clifford~Henry},
       title={Embedded contact homology and {S}eiberg-{W}itten {F}loer
  cohomology {I}},
        date={2010},
        ISSN={1465-3060},
     journal={Geom. Topol.},
      volume={14},
      number={5},
       pages={2497\ndash 2581},
         url={https://doi.org/10.2140/gt.2010.14.2497},
      review={\MR{2746723}},
}

\bib{ZemGraphTQFT}{unpublished}{
      author={Zemke, Ian},
       title={Graph cobordisms and {H}eegaard {F}loer homology},
        date={2015},
        note={e-print, \url{arXiv:1512.01184}},
}

\bib{ZemQuasi}{article}{
      author={Zemke, Ian},
       title={Quasistabilization and basepoint moving maps in link {F}loer
  homology},
        date={2017},
        ISSN={1472-2747},
     journal={Algebr. Geom. Topol.},
      volume={17},
      number={6},
       pages={3461\ndash 3518},
         url={https://doi.org/10.2140/agt.2017.17.3461},
      review={\MR{3709653}},
}

\bib{ZemCFLTQFT}{article}{
      author={Zemke, Ian},
       title={Link cobordisms and functoriality in link {F}loer homology},
        date={2019},
     journal={Journal of Topology},
      volume={12},
      number={1},
       pages={94\ndash 220},
}

\bib{ZemAbsoluteGradings}{article}{
      author={Zemke, Ian},
       title={Link cobordisms and absolute gradings on link {F}loer homology},
     journal={Quantum Topology},
        note={To appear},
}

\end{biblist}
\end{bibdiv}

\end{document}